\documentclass[11pt]{article}
\usepackage{amscd,amsmath,amsfonts,amssymb,amsthm,cite,mathrsfs} 
\usepackage[T1]{fontenc}
\usepackage{lmodern} 
\usepackage{appendix}
\usepackage{graphicx}
\usepackage{hyperref}
\usepackage[left=2.cm, right=2.cm, top=2.5cm, bottom=2.5cm]{geometry}
\usepackage{xcolor}
\usepackage{enumitem}
\usepackage{setspace}
\usepackage{authblk}

 \allowdisplaybreaks

\numberwithin{equation}{section} 

\theoremstyle{plain}
\newtheorem{theorem}{Theorem}[section]
\newtheorem{lemma}[theorem]{Lemma}
\newtheorem{corollary}[theorem]{Corollary}
\newtheorem{proposition}[theorem]{Proposition}

\theoremstyle{definition}
\newtheorem{definition}[theorem]{Definition}
\newtheorem{remark}{Remark}[section]
\newcommand{\tr}{\mathrm{Tr}}

\title{\bf Edge Universality for Inhomogeneous Random Matrices}

\author[1]{Dang-Zheng Liu
  \thanks{dzliu@ustc.edu.cn}}
\author[1,2]{Guangyi Zou
  \thanks{zouguangyi2001@gmail.com}}

\affil[1]{School of Mathematical Sciences, University of Science and Technology of China
}
\affil[2]{Department of Mathematics, University of California, Irvine
}

\date{\today}

\begin{document}

 \maketitle

\begin{abstract}
We consider symmetric and Hermitian random matrices whose entries are   independent and symmetric random variables with an arbitrary   variance  pattern. Under   a  novel Short-to-Long Mixing condition, which is sharp in the sense that it precludes a corrected shift at the spectral edge,  we establish GOE/GUE edge universality for such inhomogeneous random matrices. This condition effectively reduces the universality problem to verifying the mixing properties of  random walk governed by the variance profile matrix.

Our universality results are  applicable  to 
 a remarkably broad class of random matrix ensembles   that may be highly inhomogeneous, sparse or  
 far beyond the mean-field setting of classical random matrix theory. Notable examples include 
 \begin{enumerate}
 \item   Inhomogeneous Wishart-type random  matrices;
  \item   
  Random band matrices whose  entries are independent random variables   with   general  variance profile, particularly with an optimal bandwidth in dimensions $d \le 2$;

    \item   Sparse  random matrices with structured variance profiles;
    
    \item 
    Generalized Wigner matrices under significantly weaker sparsity constraints and heavy-tailed entry distributions;

    \item 
     Wegner orbital   models under sharp mixing assumptions; 
    
    \item Random 2-lifts of random $d$-regular graphs where    $d\geq 
    N^{2/3+\epsilon}$ for any $\epsilon>0$.
    
\end{enumerate}
\end{abstract}

\hfill{\textit{Guangyi dedicates this work to his years at USTC.}}

\tableofcontents

\section{Introduction}

\subsection{Inhomogeneous random matrices}

Random Matrix Theory (RMT), which lies at the intersection of matrix theory and probability theory---both ubiquitous in mathematics, science, and engineering---plays a central role in numerous areas of mathematics and physics; see the excellent survey \cite{akemann2011oxford}. When 
  regarded as non-commutative random variables, the eigenvalues of random matrices    form strongly correlated and highly complex systems. Remarkably, this correlation structure exhibits universality, depending only on the symmetry type of the matrix. This universality class, commonly referred to as \textit{random matrix statistics}, differs significantly from the Gaussian universality class of the classical Central Limit Theorem \cite{MR2334189}.
Such universality phenomena lie at the very core of RMT and extend far beyond the realm of random matrices themselves, including the KPZ universality class~\cite{MR2930377,PhysRevLett.56.889}  in stochastic surface growth and the BGS conjecture in quantum chaos~\cite{BGS1984}. 
Now, random matrices have proven to be invaluable tools, with critical applications in number theory, statistics, machine learning, population dynamics, wireless communications, and mathematical finance \cite{MR2760897,bai2010spectral,erdos2017dynamical,forrester2010log}.

The physical origins of  RMT can be traced back to Wigner’s seminal work in the 1950s~\cite{Wigner1955Characteristic}, in which he proposed large random Hamiltonians---including both mean-field and band matrices with random signs (Rademacher entries)---as models for complex quantum systems. 
Later,   Wigner refined two fundamental questions concerning random matrices~\cite{wigner1967random}:
\begin{quote}
\textit{The first question, then, is what are the admissible Hamiltonians, and what is the proper measure in the ensemble of these Hamiltonians. The second question is, of course, whether, given the ensemble of admissible Hamiltonians with a proper measure, the properties in which we are interested are common for the vast majority of them.}
\end{quote}
These questions have become central to modern random matrix theory. The first is addressed by defining specific matrix ensembles---such as real symmetric or Hermitian matrices with independent entries---while the second leads to the celebrated \textit{\bf Universality Conjecture}: local eigenvalue statistics of large random matrices are asymptotically independent of the entry distributions and coincide with those of the Gaussian orthogonal and unitary ensembles (GOE/GUE).

While universality is well understood for mean-field ensembles---such as Wigner matrices \cite{erdos2017dynamical}, with nearly identical entry variances, and invariant ensembles \cite{MR2514781}, with conjugation-invariant   law---real-world systems typically deviate from these idealized settings. 
Prominent examples---such as random band matrices
\cite{bourgade2018random},  adjacency matrices of sparse  networks \cite{MR3966517}, and  
  structured covariance matrices \cite{MR3966514,MR2334195}---exhibit strong inhomogeneity and locality, which present substantial analytical challenges.
  This naturally raises a  question:  {\bf why can mean-field ensembles still effectively predict the behavior of such non-mean-field systems? }

This line of inquiry, combined with Wigner's original question cited above, naturally leads to introduction of a very general class of matrix  models:   inhomogeneous random matrices (IRM).
Informally, an inhomogeneous Wigner-type matrix $H_N$ is an $N \times N$ symmetric/Hermitian  matrix with independent centered entries (up to symmetry) defined by
\[
H_{xy} = \sigma_{xy} W_{xy}, \quad x,y\in [N]:=\{1,2,\ldots,N\},
\]
where $W_{xy}$ are independent   random variables with finite  variance and the non-negative values $\sigma_{xy}$ form a variance profile matrix $(\sigma_{xy}^2)$. By appropriately choosing this profile, one recovers several classical models:
\begin{itemize}[leftmargin=*]
\item Wigner matrix: $\sigma_{ij}^2 = N^{-1}$;
\item Band matrix: $\sigma_{ij}^2 \propto W^{-d} f((i-j)/W)$ for some bandwidth $W$;
\item Sparse matrix (e.g., Erdős-Rényi-type): $\sigma_{ij}^2 \in \{0, p^{-1}\}$ with sparsity parameter $p \ll N$.
\end{itemize}

Recently, inhomogeneous random matrices have attracted considerable attention, with notable progress on spectral norms \cite{MR3878726,MR3837269} and non-asymptotic analysis \cite{bandeira2024matrix,MR4635836,brailovskaya2024extremal,MR4823211}. Despite these advances, the universality of local eigenvalue statistics---crucial  for both theoretical insights and practical applications---remains far less understood in inhomogeneous settings.

\medskip
\noindent

This paper addresses a fundamental   question concerning the local spectral statistics of these inhomogeneous  ensembles, with    particular emphasis   on the spectral edge due to its role in governing the matrix norm.
\begin{enumerate}[leftmargin=*,label=\textbf{Question }$\spadesuit$]
\item\label{rmt_question}
\textit{Under what conditions do inhomogeneous random matrices exhibit universal   eigenvalue statistics  identical to those of  mean-field  matrix ensembles? }
\end{enumerate}

\subsection{Model and main results}

The primary models we consider in addressing \ref{rmt_question} are additive deformations of symmetric and Hermitian inhomogeneous random matrices (IRM). Their independent entries are symmetrically distributed and $\theta$-sub-Gaussian, subject to the standard symmetry constraints.

\begin{definition}[Inhomogeneous (deformed) random matrices] \label{def:inhomo}
Let $W_N$ be an $N\times N$ real symmetric (or complex Hermitian) Wigner matrix, let $\Sigma_N$ be a deterministic variance-profile matrix and $A_N$ be a finite-rank perturbation.  Define
\begin{equation} \label{HN}
H_N=\Sigma_N \circ W_N,  
\end{equation}
where “\(\circ\)” is the Hadamard (entrywise) product, and the corresponding deformed inhomogeneous random  matrix
\begin{equation}\label{eq:deformed-iid}
    X_N=H_N+A_N.
\end{equation}
Here,  the following conditions are assumed: 

\begin{enumerate}[leftmargin=*,label=(A\arabic*)]
  \item \label{itm:A1} \textbf{(Wigner matrix)}  
    The entries $\{W_{ij}\}_{i\le j}$ are independent, symmetric, and satisfy
    \begin{equation}
      \begin{cases}
        \mathbb{E}[W_{ii}^2]=2,\;\mathbb{E}[W_{ij}^2]=1, & \text{real 
        case},\\
        \mathbb{E}[W_{ii}^2]=1,\;\mathbb{E}[|W_{ij}|^2]=1,\;\mathbb{E}[W_{ij}^2]=0, & \text{complex 
        case}.
      \end{cases}
    \end{equation}
    Moreover, there exists $\theta\ge 1$ such that for all $k\ge2$,
    \begin{equation}
      \mathbb{E}\bigl[|W_{ij}\bigr|^{2k}]
      \;\le\;
      \theta^{\,k-1}(2k-1)!!,
     \end{equation}
    where  $\theta:=\theta_N\rightarrow\infty$ may be  allowed   if needed. 
  \item \label{itm:A2} \textbf{(Variance profile)}  
    The profile matrix $\Sigma_N=(\sigma_{ij})$ has nonnegative entries such that 
    \(\sum_{j=1}^N\sigma_{ij}^2=1\) for each~$i$.  Equivalently, 
    $P_N:=(\sigma_{ij}^2)_{i,j=1}^N
    $ is a Markov transition matrix.

  \item \label{itm:A3} \textbf{(Finite-rank deformation)}  
    The perturbation $A_N$ is symmetric (resp.\ Hermitian) of rank $r$, with spectral decomposition
    \begin{equation}\label{SpetralA}
        A_N = Q\,\Lambda\,Q^*, 
      \qquad
      \Lambda = \mathrm{diag}(a_1,\dots,a_r,0,\dots,0),
    \end{equation}

    where $Q$ is an orthogonal (resp.\  unitary) matrix.  
\end{enumerate}
\end{definition}

To establish the universality of Tracy-Widom distributions \cite{TW1994,TW1996} and their BBP deformations   \cite{baik2005phase,bloemendal2013limits,bloemendal2016limits}, we introduce a key concept: the mixing property of a Markov chain $([N], P_N)$, which characterizes its convergence rate to the uniform distribution on $[N]$. We remark that $[N]$ can be replaced throughout by an arbitrary finite state space $S$ (with $N = |S|$), and the index $i \in [N]$ by any $i \in S$.

\begin{definition}[Short-to-Long Mixing] \label{mixingdef}
Let $t_N$ be a non-decreasing sequence of positive integers. A Markov chain $([N], P_N)$  is said to be \emph{Short-to-Long Mixing} at time \(t_N\) if its $n$-step transition probability $p_n(x,y)$ satisfies the following two conditions:

\begin{enumerate}[leftmargin=*,label=(B\arabic*)]
    \item \label{itm:B1} \textbf{(Short-time average mixing)}  
    There exists a constant $\gamma \ge 1$ and $N_0 > 0$ such that for all $N \ge N_0$ and all $x, y \in [N]$,
    \begin{equation}
        \frac{1}{t_N} \sum_{n=1}^{t_N} p_n(x, y) \le \frac{\gamma}{N}.
    \end{equation}

    \item \label{itm:B2} \textbf{(Long-time uniform mixing)} 
    There exists  $\delta \in (0, 0.1)$, such that for all $n \ge t_N$ and all $x, y \in [N]$,
    \begin{equation}
        \left| p_n(x,y) - \frac{1}{N} \right| \le \frac{\delta}{N}.
    \end{equation}
\end{enumerate}
The additional assumption that  $\delta = \delta_N \to 0$  may be imposed   when necessary.
\end{definition}

Our principal theorem  establishes the Tracy-Widom law and BBP transition for a broad class of deformed inhomogeneous random matrices.   
This can be interpreted as a Lindeberg-type central limit theorem for the Tracy-Widom law  within random matrix theory,   
even in the non-deformed case.    

\begin{theorem}[{\bf Edge Universality}]\label{thm:main_thm} 
For the matrix  $X_N=\Sigma_N \circ W_N +A_N$ as in Definition \ref{def:inhomo}, let $r$ be any fixed nonnegative integer and $0 \leq q \leq r$, assume that eigenvalues of $A_N$ satisfy   
        \begin{equation}
            a_j = 1 + \tau_j N^{-\frac{1}{3}}, \quad\tau_j \in \mathbb{R},\quad j=1,\ldots,q, 
        \end{equation} 
        and   
        \begin{equation}
           a_j \in (-1,1),  \quad j=q+1,\ldots,r.
        \end{equation} 
          If Markov chain   $([N],P_N)$ satisfies  the Short-to-Long Mixing conditions \ref{itm:B1} and \ref{itm:B2} with $\theta\, t_N\ll N^{\frac{1}{3}}$, then the first $k$ largest  eigenvalues of $X_N$ 
converge   in distribution  to those of  deformed GOE/GUE matrix. 
\end{theorem}

A universal  result analogous to Theorem~\ref{thm:main_thm} also holds for inhomogeneous Wishart-type ensembles, as presented in Theorem~\ref{thm:main_thm_Wishart} in Section~\ref{sec:section6}.
As direct applications, we establish edge universality
for  a remarkably broad class of random matrix ensembles 
that may be highly inhomogeneous, sparse or far beyond the mean-field setting of classical random
matrix theory, notable examples including 
(see Section~\ref{sec:section7} for more details):  
\begin{enumerate}[leftmargin=*]
    \item Inhomogeneous Wishart-type random matrices (Theorem \ref{thm:main_thm_Wishart});
    \item Random band matrices with independent entries and general variance profiles, including the optimal bandwidth in dimensions $d \le 2$ (Theorem \ref{thm:RBM});
    \item Sparse random matrices with structured variance profiles (Theorem \ref{thm:sparse});
    \item Generalized Wigner matrices under significantly weaker sparsity and heavy-tailed entry assumptions (Theorem \ref{thm:GW});
    \item Wegner orbital models under sharp mixing conditions (Theorem \ref{thm:block});
    \item Random $2$-lifts of $d$-regular graphs for $d \ge N^{2/3+\epsilon}$, for any $\epsilon > 0$ (Theorem \ref{thm:2lift}).
\end{enumerate}

We conclude with several remarks closely related to the aforementioned universality result, aimed at clarifying the core assumptions. 
\begin{remark}
Conditions~\ref{itm:B1} and~\ref{itm:B2} are both essential for edge universality, and the violation of either leads to non-universal behavior. For instance, the block-diagonal GUE model satisfies Condition~\ref{itm:B1} but violates Condition~\ref{itm:B2}, resulting in a clear failure of universality at the edge. Conversely, in dimensions $d > 2$, random band matrices satisfy Condition~\ref{itm:B2} but violate Condition~\ref{itm:B1}, necessitating a non-universal correction at the spectral edge, as shown in \cite{liu2023edge}.

Furthermore, the condition $\theta t_N\ll N^{{1}/{3}}$ is also sharp. In the sparse Wigner case with $t_N = 1$, if $\theta \gg N^{1/3}$,
 a non-universal correction emerges, as demonstrated in \cite{lee2018}.

\end{remark}

\begin{remark}
The doubly stochastic variance profile (Assumption \ref{itm:A2}) and the finite-rank condition (Assumption \ref{itm:A3}) in Definition \ref{def:inhomo} together guarantee the semicircle law, which facilitates the asymptotic analysis of the largest eigenvalue. Although these constraints may be relaxed in more general settings, doing so requires   additional insight.
\end{remark}

\begin{remark}
\label{rmk:thouless}
Our universality  results are consistent with the Thouless criterion for localization, adapted to the band matrix setting by Fyodorov and Mirlin \cite{fyodorov1994statistical}; see also \cite{sodin2010spectral,spencer2011random,Wang1992}. 
Informally, the Thouless criterion \cite{thouless1977maximum} compares a dynamical (diffusive/mixing) time scale with the spectral resolution scale and thereby predicts whether local spectral statistics follow GOE/GUE laws or degenerate to Poisson statistics.

More concretely, let  $t_{\mathrm{mix}}:=t_N$
denote the mixing time of the variance-profile Markov chain  $([N],P_N)$ and let  $t_{\mathrm{Th}}$ be the \textbf{Thouless time}. 
One expects the following dichotomy:
\begin{equation}\label{eq:rm-vs-poisson}
\begin{cases}
t_{\mathrm{mix}}\ll t_{\mathrm{Th}}
&\implies
\begin{aligned}
&\text{Random matrix statistics (extended eigenvectors),}
\end{aligned}\\[6pt]
t_{\mathrm{mix}}\gg t_{\mathrm{Th}}
&\implies
\begin{aligned}
&\text{Poisson statistics (localized eigenvectors).}
\end{aligned}
\end{cases}
\end{equation}
In particular, in the bulk ($\alpha_0 \in (-2,2)$) the Thouless time scales as $t_{\mathrm{Th}} \sim N$, while at the spectral edge ($\alpha_0 \approx \pm 2$), it satisfies $t_{\mathrm{Th}} \sim N^{1/3}$. For example, in random band matrices with mixing time $t_{\mathrm{mix}} \sim (L/W)^2$, the Thouless criterion predicts a critical bandwidth $W \sim \sqrt{L}$ in the bulk for $d = 1$, and $W \sim L^{1 - d/6}$ at the edge for $d < 6$.

Our long-time mixing condition~\ref{itm:B2} in Theorem~\ref{thm:main_thm} ensures that the model lies firmly in the regime where $t_{\mathrm{mix}} \ll t_{\mathrm{Th}}$. For further physical intuition regarding Conditions~\ref{itm:B1} and~\ref{itm:B2}, we refer to \cite[Section~VI, (iii)]{guhr1998random}.
 
\end{remark}

\subsection{Previous results}
\paragraph{Universality of local eigenvalue   
  statistics}
Local universality for Wigner matrices is by now well-established, with foundational breakthroughs achieved over the past three decades \cite{erdHos2010bulk, erdHos2011universality, johansson2001universality, soshnikov1999universality, tao2010random, tao2011random}, followed by extensions to generalized Wigner matrices \cite{bourgade2016fixed,erdos2012generalized, MR3253704}. We refer to surveys such as \cite{MR2760897, bourgade2018random, erdos2017dynamical} for comprehensive overviews. Parallel advances have established universality for sparse random matrices: in the bulk spectrum \cite{MR3098073, MR3429490}, and at the spectral edge after appropriate corrections \cite{MR2964770, MR4089498, MR4288336, huang2022edge, lee2018}. For random $d$-regular graphs, bulk universality is known for $d \geq N^{\epsilon}$ \cite{MR3729611}, and edge universality for all $d \geq 3$ was recently proved in \cite{huang2024ramanujan}.
Later, universality and delocalization  on some specific non-mean-field random matrix models, typically random band matrices, were explored. 
Important contributions include \cite{bao2017delocalization, bourgade2020random, benaych2014largest, erdHos2011quantum, EK11Quantum, erdHos2013delocalization, he2019diffusion, liu2023edge, sodin2010spectral,Shcherbina2021, shcherbina2014second,
yang2021delocalization,
yang2022delocalization, MR4736267} (though this list is not exhaustive). 
Very    recent breakthrough results on  universality and delocalization  of  random band matrices  have    been established  in \cite{dubova2025delocalization,dubova2025delocalization2,erdHos2025zigzag,fan2025blockreductionmethodrandom,liu2023edge,MR4736267,yang2025delocalization2,yau2025delocalization}. Most focus on the specific block variance profile or  complex Gaussian entries,  with the exception of the
one-dimensional case \cite{erdHos2025zigzag}. We refer to \cite{erdHos2025zigzag,dubova2025delocalization}   and references therein for state-of-the-art results and surveys.

Edge universality serves as the central focus of this paper.  The study of spectral edge statistics originated in the seminal contributions of Tracy and Widom \cite{TW1994,TW1996}, Forrester \cite{forrester1993spectrum},  which introduced the now ubiquitous Tracy-Widom distributions and Airy point processes for GUE and GOE  ensembles. These limiting distributions were later shown to hold for general Wigner matrices under suitable moment conditions \cite{soshnikov1999universality, johansson2012universality, tao2010random}, culminating in a necessary and sufficient characterization established by Lee and Yin \cite{MR3161313}. A parallel line of inquiry concerns universality in deformed random matrices and the BBP phase transition \cite{baik2005phase, peche2006largest, bloemendal2013limits, bloemendal2016limits, lee2015edge}; we refer to \cite{peche2014deformed} for a comprehensive survey.

\paragraph{On inhomogeneous random matrices}

Inhomogeneous (or structured) random matrices arise naturally in a wide range of applications,  such as   statistical inference, numerical linear algebra, compressed sensing, and network analysis. These have attracted significant interest in studies of extreme singular values~\cite{MR4779853,MR4255145,MR2827856}, spectral outliers~\cite{geng2024outliers,han2024outliers,MR4234995}, operator norms~\cite{bandeira2016sharp,MR3878726,MR3837269}, and concentration or deviation inequalities~\cite{MR4635836,bandeira2024matrix,tropp2015introduction}. 
For comprehensive overviews, we refer to~\cite{brailovskaya2024extremal,tropp2015introduction,Vershynin2012,MR3837269}.
 
For Wigner-type inhomogeneous matrices, a growing body of research has addressed both global and extremal spectral properties. In particular, \cite{MR4552698,MR3416062} demonstrated that the empirical spectral distribution  converges to the semicircle law under the mild condition that $\max_{x,y} \sigma_{xy} \to 0$. Under sharp moment conditions, almost sure convergence of the largest eigenvalue was established in~\cite{altschuler2024spectral}. The behavior of outliers in deformed inhomogeneous ensembles has been investigated in~\cite{geng2024outliers,bandeira2024matrix}, and several non-asymptotic deviation bounds for the operator norm have been derived in~\cite{brailovskaya2024extremal}.

Although much of the literature has focused on non-asymptotic bounds, such results are often suboptimal and may fail to capture the correct fluctuation scale in general settings (see~\cite[Section~1.2.2]{brailovskaya2024extremal}). In contrast, asymptotic fluctuation theory---such as Tracy-Widom-type laws for the spectral edge---remains underdeveloped, with the exception of results on outlier fluctuations in~\cite{geng2024outliers}.

\paragraph{The method of moments}
In essence, while the method of moments is central to our approach, it requires the analysis of very high-order moments, which presents significant technical challenges. The method of moments in probability theory was first rigorously treated by P. L. Chebyshev in 1887 \cite{chebyshev1890deux}, who used it to provide the first fully rigorous proof of the central limit theorem \cite{diaconis11987application,billingsley}. In random matrix theory, Wigner pioneered this method in 1955 to establish the semicircle law \cite{Wigner1955Characteristic,wigner1958distribution}, and it was later employed by Soshnikov in 1999 to prove the Tracy-Widom law \cite{soshnikov1999universality}.
Subsequently, Feldheim and Sodin \cite{feldheim2010universality} developed a more sophisticated Chebyshev polynomial moment approach, extending universality results to both Wigner and sample covariance matrices.  While studying the longest increasing subsequence problem in random permutations, Okounkov \cite{okounkov2000random} rederived the Tracy-Widom law and the corresponding Laplace transform for the classical GUE ensemble.

A major breakthrough was achieved in Sodin’s seminal work \cite{sodin2010spectral}, which identified the critical bandwidth scaling 
 $W_c=L^{5/6}$
  for one-dimensional random band matrices with unimodular entries and a cutoff variance profile. This work rigorously characterized edge statistics in both extended and localized regimes. Subsequently, the Chebyshev polynomial framework was extended into a polynomial moment method by \cite{liu2023edge}, enabling the analysis of edge statistics for random band matrices in higher dimensions. Furthermore, the method of large moments has also been applied to investigate outliers of deformed  IRM ensembles in very  general settings~\cite{geng2024outliers}.

\subsection{Key contributions}

This paper highlights the powerful polynomial moment method, developing a general strategy to 
resolve previously inaccessible universality problems  for random matrices with arbitrary variance profiles and arbitrary entry distributions, requiring only the essential sole assumption of independence.  This approach  builds upon the combinatorial and diagrammatic expansion techniques developed in \cite{feldheim2010universality, sodin2010spectral, EK11Quantum, liu2023edge, geng2024outliers}, employing a Chebyshev polynomial moment framework.   

The core contributions of this paper are summarized below:

\begin{enumerate}
\item \textbf{Short-to-Long Mixing Condition.}
 We introduce a novel Short-to-Long Mixing Condition, which is sharp in the sense that it prevents any corrected shift at the spectral edge. This condition effectively reduces the universality problem to verifying mixing properties of a random walk governed by the variance profile matrix. Violations of any aspect of this condition lead to natural counterexamples.

  \item \textbf{New Chebyshev expansion for inhomogeneous Gaussian matrices.}  
We establish  a closed-form diagrammatic expansion for Chebyshev polynomial moments in inhomogeneous Gaussian ensembles by combining Wick’s formula with classical identity for Chebyshev polynomials (see Proposition~\ref{prop:T=Gamma} and Theorem~\ref{Chebyshevmoment} below).
Unlike previous methods that rely on combinatorial recursions and non-backtracking paths within the unimodular framework \cite{feldheim2010universality,sodin2010spectral,EK11Quantum,liu2023edge}, our approach reveals a new combinatorial identity unique to the Gaussian setting. 
This formulation provides a clearer and more efficient reduction from trace moments to diagrammatic (i.e., reduced ribbon graph) representations in the Gaussian case.
Whereas diagram functions in the unimodular framework are only introduced after an approximation step, as in \cite{sodin2010spectral,liu2023edge}, our Gaussian-specific expansion avoids such intermediate approximations.

\item \textbf{A unified framework for edge spectral statistics.}
We develop a comprehensive framework for establishing the Tracy-Widom law---and even the BBP phase transition---for the largest eigenvalue of deformed inhomogeneous  matrix ensembles. Our strategy consists of three main steps: first, we derive diagram expansion in the Gaussian case via the Wick formula; next, we analyze  upper bounds and asymptotics of the diagram functions; finally, we extend these to the sub-Gaussian setting through moment control and graph reduction estimates.
Further extensions and more powerful implications of this framework will be detailed in forthcoming work.

\item \textbf{Wide-ranging applicability   and  sharpness.}
Our universality results 
 substantially extend recent important   advances on random band matrices in the delocalized regime---previously limited to unimodular entries and highly structured variance profiles \cite{sodin2010spectral,liu2023edge}. Moreover, our framework encompasses a broad spectrum of random matrix ensembles, including highly inhomogeneous, sparse, and non-mean-field models, extending well beyond classical settings. As detailed in Section~\ref{sec:section7}, Theorem~\ref{thm:main_thm} in several cases establishes \emph{sharp} criteria for edge universality.

\end{enumerate}

\subsection{Structure   of this paper}

This paper is structured  as follows. Section~\ref{sec:section2} treats the Gaussian case, where Wick’s formula is used to express mixed moments of Chebyshev polynomials via diagram functions arising from reduced ribbon graphs. Section~\ref{sec:section3} derives upper bounds and asymptotic estimates for these  {diagram functions}. These estimates are combined in Section~\ref{sec:section4} to establish edge universality for inhomogeneous Gaussian matrices. The universality result is extended to sub-Gaussian matrix ensembles in Section~\ref{sec:section5}. Section~\ref{sec:section6} adapts the method to inhomogeneous sample-covariance   matrices. Section~\ref{sec:section7} then illustrates applications to several central models, including inhomogeneous sparse random matrices, generalized Wigner matrices, random band and block matrices, and weighted $d$-regular random matrices. Additional technical material---including Chebyshev expansions in the deformed setting and supporting probabilistic inequalities---is provided in Appendices~\ref{Chebyshevdeform}--\ref{sec:clt-upper}.

{\bf Notation.}  Throughout  this  paper,  constants such  as  $C, C_1, \ldots$  are independent of  $N$.   $f(n)=o(g(n))$ or $f(n) \ll g(n)$ means that   $ {f(n)}/{g(n)}\to 0$ as $n\rightarrow\infty$, while   $f(n)=O(g(n))$  means  that   $ {f(n)}/{g(n)}$    is    bounded. 
The symbols   $a\wedge b=\min\{a,b\}$,   $a\lor b=\max\{a,b\}$.

\section{Chebyshev moments  for Gaussian   matrices}\label{sec:section2}

\subsection{Ribbon graph expansion} \label{sec:ribbow_expansion}

To compute the mixed moments $\prod_{j=1}^{s} \tr (X^{m_j})$ of the deformed Gaussian IRM ensemble via the Wick formula, one must consider all possible pairwise gluings of edges of several regular polygons, see e.g. Figure \ref{fig:ribbon_graph_example} and \ref{fig:ribbon_graph_example_open}. These gluings correspond to combinatorial structures called \emph{ribbon graphs}---also known as maps---which encode how graphs are embedded into surfaces (possibly non-orientable). The resulting surfaces form a natural framework for organizing such moment expansions.

The correspondence between ribbon graphs on orientable surfaces and Hermitian random matrix ensembles (e.g., GUE) is classical and well-studied; see, for example, \cite{harer1986euler,kontsevich1992intersection,okounkov2000random,lando2004graphs}. For real ensembles (e.g., GOE), similar but subtler relations exist, involving ribbon graphs on possibly non-orientable surfaces, as formalized via the notion of \emph{Möbius graphs} in \cite{mulase2003duality}.

Informally, a ribbon graph is a graph embedded into a compact surface (possibly with boundary), such that the complement of the embedding is a disjoint union of open disks, called faces. It therefore defines a cell decomposition of the surface and uniquely determines its topological type (orientable or not). Recent developments include the study of \emph{ribbon graphs with boundary}, which arise naturally in open intersection theory; see, e.g., \cite{buryak2017matrix,tessler2023combinatorial}.

\begin{figure}[htbp]
    \centering
    \begin{minipage}[t]{0.48\textwidth}
        \centering
        \includegraphics[width=\textwidth]{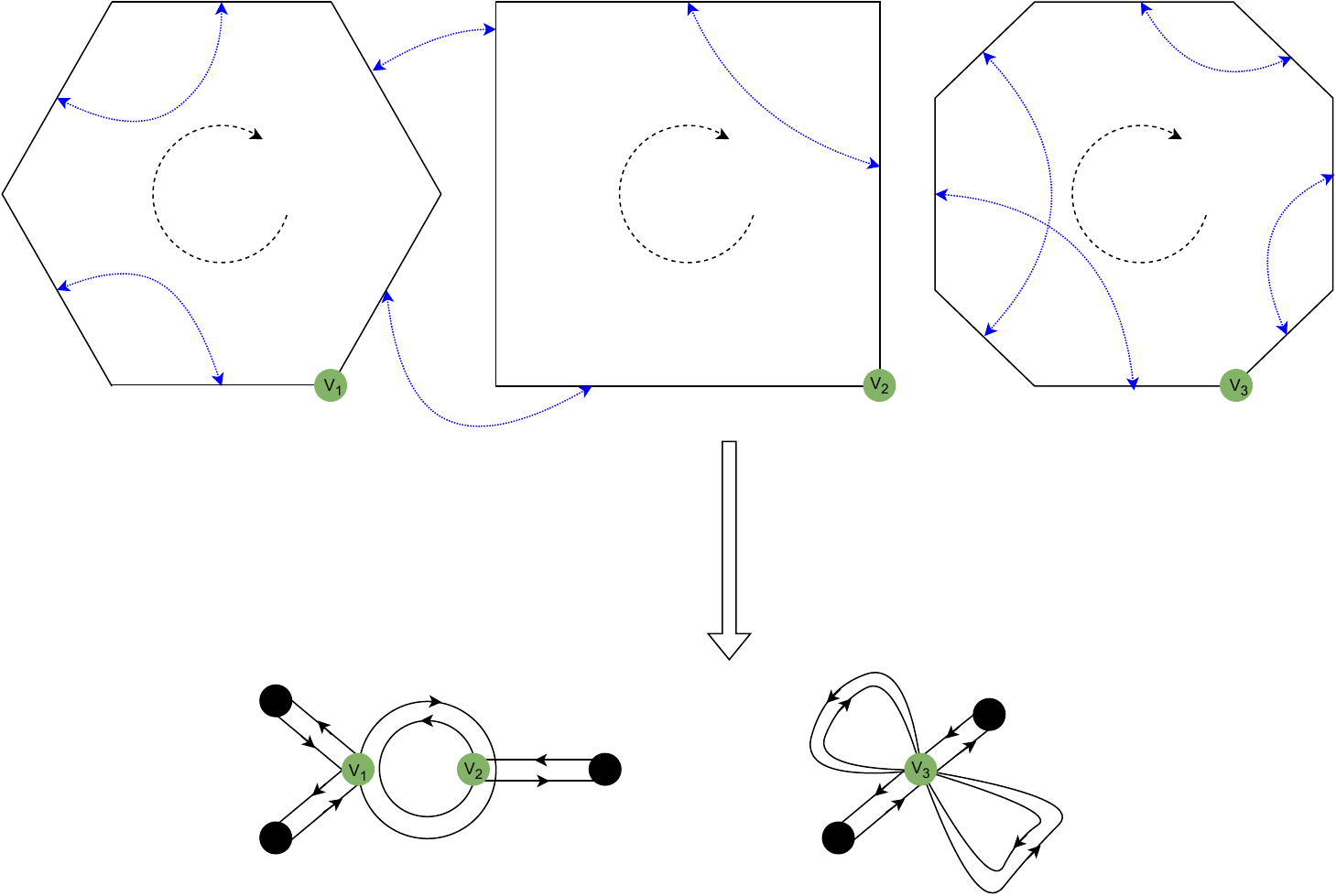}
        \caption{Example of one possible gluing in Hermitian case of $\mathbb{E}[\tr X^{6}\tr X^{4}\tr X^{8}]$. The green vertices are marked vertices.}
        \label{fig:ribbon_graph_example}
    \end{minipage}
    \hfill
    \begin{minipage}[t]{0.48\textwidth}
    \centering
    \includegraphics[width=0.8\textwidth]{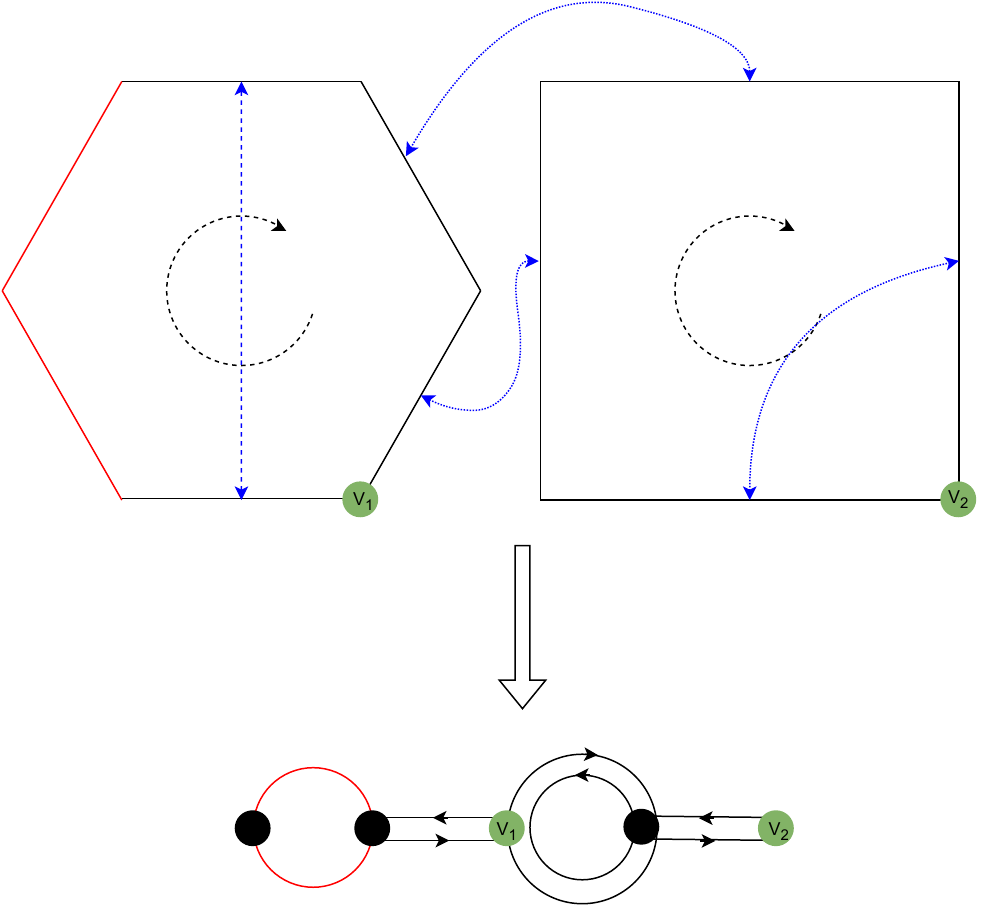}
    \caption{Example of one possible gluing of $\mathbb{E}[\tr X^{6}\tr X^{4}]$ with open edges.}
    \label{fig:ribbon_graph_example_open}
    \end{minipage}
\end{figure}
\begin{figure}[htbp]
    \centering
    \begin{minipage}{\textwidth}
        \centering
        \includegraphics[width=\textwidth]{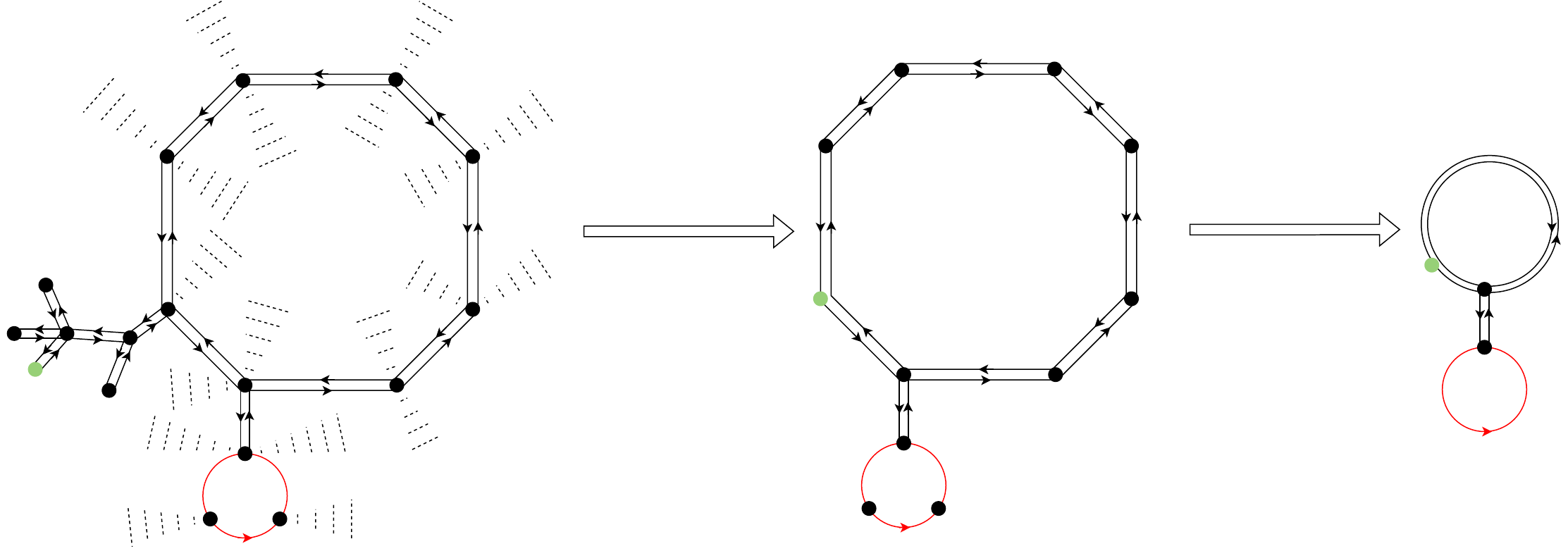}
        \caption{Okounkov contraction. The green vertex is the marked vertex, and dashed lines represent the Catalan trees. The red line indicates the open edge. Each interior point is traversed twice—hence two dashed lines—since a Catalan tree may grow on each traversal. In the first arrow (steps (i) and (ii) of Definition \ref{def:Okounkov_contraction}), all Catalan trees are collapsed and the marked vertex is moved to the tree root. In the second arrow (step (iii)), all degree-2 vertices are removed, yielding the reduced ribbon graph.}
        \label{fig:OK_contraction}
    \end{minipage}
\end{figure}
\begin{figure}[htbp]
    \centering
    \begin{minipage}{\textwidth}
        \centering
        \includegraphics[width=\textwidth]{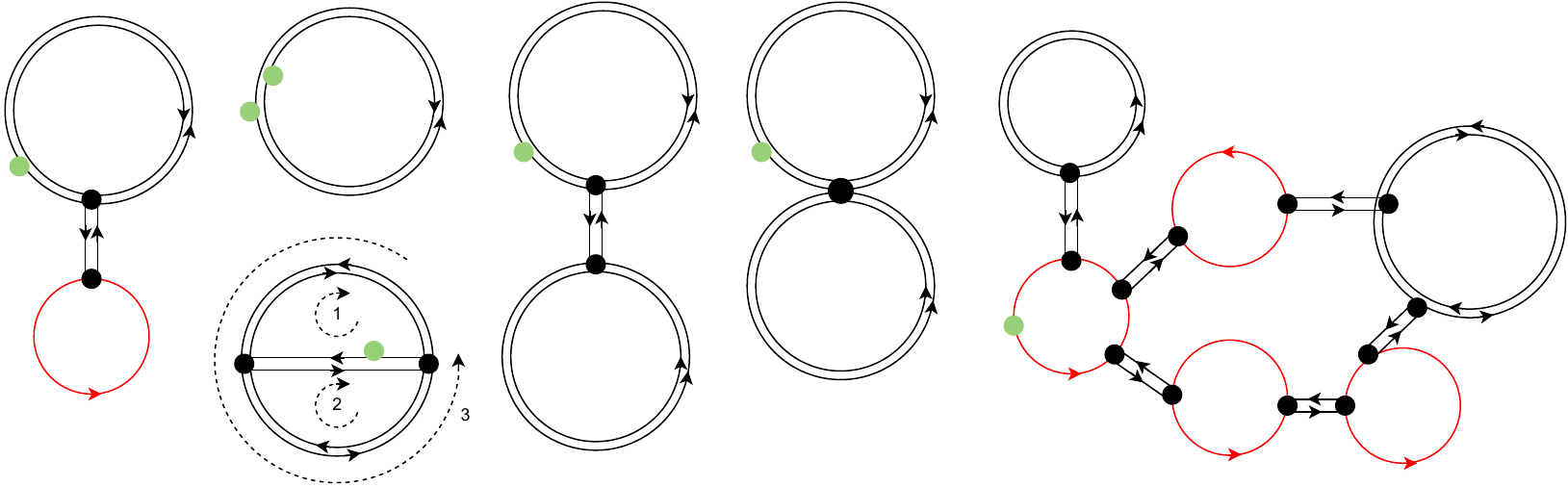}
        \caption{Examples of reduced ribbon graphs (diagrams). In the left lower graph, the circuit starts at the marked vertex and traverses with order $1,2,3$.}
        \label{fig:reduced_ribbon_example}
    \end{minipage}
\end{figure}
\begin{figure}[htbp]
    \centering
    \begin{minipage}{\textwidth}
        \centering
        \includegraphics[width=\textwidth]{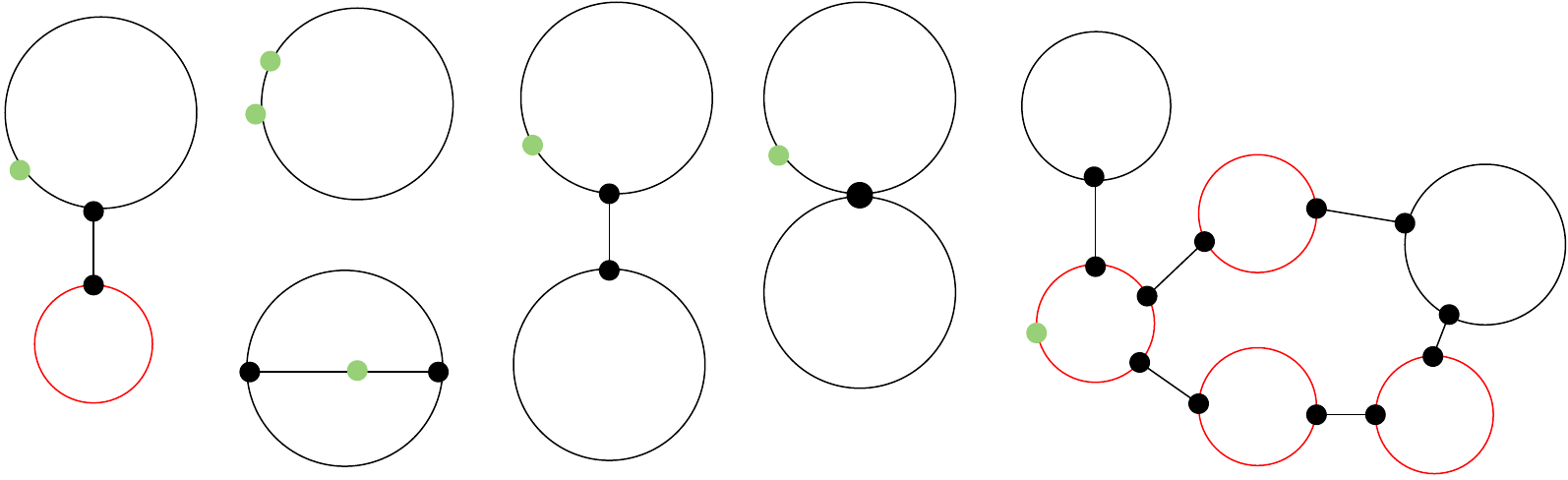}
        \caption{The underlying graphs corresponding to  diagrams in Figure~\ref{fig:reduced_ribbon_example}.}
        \label{fig:underlying_graph}
    \end{minipage}
\end{figure}

\paragraph{Canonical construction via polygon gluing.}
We describe a standard way to construct a punctured ribbon graph $\Upsilon$. Let $D_1,\dots,D_s$ be $s$ oriented polygons with $k_1,\dots,k_s$ sides, each with a marked vertex. Let $k = k_1 + \cdots + k_s$, and define the permutation
\[
\gamma = (1~2~\cdots~k_1)(k_1+1~\cdots~k_1+k_2)\cdots(k_1+\cdots+k_{s-1}+1~\cdots~k) \in S_k.
\]
Label the vertices   in cyclic order as $v_1,\dots,v_k$, and define directed edges as $\vec{e}_j = \overrightarrow{v_j v_{\gamma(j)}}$.

Let $J \subset [k]$ be a subset of edge labels and let $\pi$ be a pairing on $J$. For each $(s,t)\in \pi$,
\begin{itemize}
    \item In the complex case ($\beta=2$), glue $\vec{e}_s$ and $\vec{e}_t$ in opposite directions;
    \item In the real case ($\beta=1$), glue $\vec{e}_s$ and $\vec{e}_t$ either in the same or opposite direction (possibly yielding a non-orientable surface).
\end{itemize}
After gluing, we identify $v_s = v_{\gamma(t)}$ and $v_t = v_{\gamma(s)}$. Edges in $J^c$ are left unglued and correspond to the entries of external deformation $A_N$, which are called \textit{open edges}. This construction leads to   an $s$-cell ribbon graph with perimeter $(k_1,\dots,k_s)$ on a compact surface. For an illustration, see Figures \ref{fig:ribbon_graph_example} and \ref{fig:ribbon_graph_example_open}.

\paragraph{Ribbon graphs and diagrams.} Now we introduce some definitions that have appeared in the previous articles, for instance,\cite{geng2024outliers}.

\begin{definition}[\bf Ribbon graph]
Let $\Sigma$ be a compact surface (with or without boundary). A (punctured) ribbon graph $\Upsilon$ with $s$ faces and perimeters $(m_1,\dots,m_s)$ is a quadruple $(\mathcal{V}(\Upsilon), \mathcal{E}(\Upsilon), \iota, \phi)$, where
\begin{enumerate}
    \item[(1)] $(\mathcal{V}, \mathcal{E})$ is a graph;
    \item[(2)] $\iota: (\mathcal{V}, \mathcal{E}) \hookrightarrow \Sigma$ is an embedding;
    \item[(3)] $\phi: [s] \to \mathcal{V}$ assigns a marked vertex to each face;
\end{enumerate}
such that 
\begin{itemize}
    \item The boundary of the surface lies in the graph: $\partial \Sigma \subset \iota(\Upsilon)$;
    \item The complement $\Sigma \setminus \iota(\Upsilon) = \bigsqcup_{i=1}^{s} D_i$, where each $D_i$ is an oriented $m_i$-gon and $\phi(i) \in \partial D_i$.
\end{itemize}
\end{definition}

\begin{definition}[\bf Diagram/Reduced ribbon graph]\label{def:diagram}
A \emph{diagram} $\Gamma$ is a ribbon graph satisfying the following two conditions:
\begin{itemize}
    \item Every unmarked vertex has degree at least 3;
    \item Every marked vertex has degree at least 2.
\end{itemize}

For the diagram, we will frequently use   the following notation:
\begin{itemize}
    \item  {\bf Edge set}: let \(E(\Gamma)\) denote the set of all edges of \(\Gamma\), with \(E_{\mathrm{int}}(\Gamma)\) and \(E_{\mathrm{b}}(\Gamma)\) denoting the sets of interior and open (boundary) edges, respectively;
    \item {\bf Vertex   set}: let \(V(\Gamma)\), \(V_{\mathrm{int}}(\Gamma)\) and \(V_{\mathrm{b}}(\Gamma)\) denote the sets of all, interior, and boundary vertices, respectively;
    \item  {\bf Boundary  set}: for each \(j \in [s]\), let \(\partial D_j \subset E(\Gamma)\) denote the set of edges forming the boundary of face \(j\).
     \item  {\bf Diagram  set}: let $\mathscr{D}_{s;\beta}$ be the set of all $s$-cell diagrams (on orientable surfaces when $\beta=2$) and  \(\mathscr{D}_{s;\beta}^* \subset \mathscr{D}_{s;\beta}\) be the subset of connected diagrams.
\end{itemize}

\end{definition}

\paragraph{Okounkov contraction.}
We now describe a procedure to reduce ribbon graphs to diagrams by eliminating tree structures and divalent vertices. 

\begin{definition}[\bf Okounkov contraction]\label{def:Okounkov_contraction}
Let $\Upsilon$ be an $s$-cell ribbon graph with perimeters $(m_1,\dots,m_s)$. The \emph{Okounkov contraction} is a map
\[
\Phi: \left\{ \Upsilon \right\} \longrightarrow \left\{ \Gamma \in \mathscr{D}_{s;\beta} \right\}
\]
constructed via the following operations:
\begin{enumerate}[leftmargin=*,label=(\roman*)]
    \item If the marked vertex lies on a tree, move it to the root of that tree;
    \item Collapse all univalent (degree-1) vertices;
    \item Remove all divalent (degree-2) vertices, except for marked ones;
    \item Record a weight $w_e$ on each edge $e$, which is  equal to one plus the number of divalent vertices removed along it.
\end{enumerate}
\end{definition}

This contraction preserves topological features of the original ribbon graph. For trivial cases (e.g., isolated vertices), special treatment is required. See Figure \ref{fig:OK_contraction} for the contraction process and Figure \ref{fig:reduced_ribbon_example} for diagram examples.

\paragraph{Ribbon graph expansion for mixed moments.}
We now state the ribbon graph expansion for the deformed Gaussian IRM ensemble as established in \cite{geng2024outliers}.

\begin{proposition}[\cite{geng2024outliers}]\label{prop:ribbon}
Let $X$ be the deformed Gaussian random matrix from Definition \ref{def:inhomo}, where $W_N$ is GOE ($\beta=1$) or GUE ($\beta=2$). Define
\begin{equation}\label{bcat}
b_{m} = \begin{cases}
    \left(\frac{1}{2} - \frac{1}{k+1}\right)\binom{2k}{k}, & m=2k, \\
    0, & m=2k+1,
\end{cases}
\end{equation}
and for any diagram $\Gamma$,
\begin{equation}\label{frakF}
    \mathfrak{F}_{\Gamma}(\{l_i\}) =
\sum_{\eta: V(\Gamma) \to [N]} 
\sum_{\substack{w_e \ge 1 \\ \sum_{e \in \partial D_j} w_e = l_j \ \forall j}}
\prod_{(x,y) \in E_{\mathrm{int}}} p_{w_e}(\eta(x), \eta(y))
\prod_{(z,w) \in E_{\mathrm{b}}} (A^{w_e})_{\eta(z) \eta(w)},
\end{equation}
with the convention that for each trivial diagram (i.e., a single isolated vertex), the contribution is \(N \delta_{l_i, 0}\). Here, the outer sum is taken over all vertex labelings \(\eta : V(\Gamma) \to [N]\).
Then for any integers $m_1,\dots,m_s \ge 0$,
\[
\mathbb{E}\left[ \prod_{j=1}^{s} \tr(X^{m_j} + b_{m_j}\mathbf{I}) \right] =
{\sum_{\substack{0 \le l_j \le m_j \\ l_j \equiv m_j \!\!\!\!\pmod{2}}}}'
\prod_{j=1}^{s} \binom{m_j}{\frac{m_j - l_j}{2}}
\sum_{\Gamma \in \mathscr{D}_{s;\beta}} \mathfrak{F}_{\Gamma}(\{l_i\}),
\]
where in ${\sum}'$ the binomial coefficient is replaced by $\frac{1}{2}\binom{m_j}{m_j/2}$ if $l_j = 0$.
\end{proposition}

\begin{proof}
This follows directly from Propositions 2.5 and 2.6 of \cite{geng2024outliers}.
\end{proof}

\subsection{Graphical expansion for Chebyshev moments}

The Chebyshev polynomials of first and second kind $T_{n}(x)$  and $U_{n}(x)$ with  $n\ge 0$ are defined as 
\begin{equation} \label{TUdef}
    T_n(\cos \theta)=\cos (n\theta), \quad U_n(\cos\theta)=\frac{\sin(n+1)\theta}{\sin\theta},\quad x=\cos\theta.
\end{equation}
There is a well-known relation between them  
\begin{equation} \label{TU}
    2T_n(x)=U_{n}(x)-U_{n-2}(x),\quad n\ge 1.
\end{equation}
Also, by rescaling    $\widetilde{T}_n(x):=2T_n({x}/{2})$, the $m$-th power can be expressed   via Chebyshev polynomials of first  kind as 
\begin{equation}\label{equ:x=T}
    x^m={\sum_{0\le j\le m}}'\binom{m}{\frac{m-j}{2}}2T_j(\frac{x}{2})={\sum_{0\le j\le m}}'\binom{m}{\frac{m-j}{2}}\widetilde{T}_j(x),
\end{equation}
where the prime  notation  $\sum'$ is over all $j\equiv m\pmod2$ and  the term of $j=0$  is to be halved if there is one; see \cite[eq (2.14)]{mason2002chebyshev}. In the reverse direction the exact formula for   Chebyshev polynomials   in terms of powers of $x$ is  
    \begin{equation}\label{equ:T=x}
        \widetilde{T}_n(x)=\sum_{m=0}^{[\frac{n}{2}]}(-1)^m
        \frac{n}{n-m}\binom{n-m}{m}
        x^{n-2m},
    \end{equation}
 see \cite[eq (2.14)-(2.18)]{mason2002chebyshev}.

\begin{proposition}\label{prop:T=Gamma} With the same assumption and notation as in Proposition \ref{prop:ribbon}, 
let \begin{equation}
    b'_{n}=\begin{cases}
        -\frac{1}{2},&n=0,\\
        1,&n=2,\\
        0,&\text{otherwise},
    \end{cases}
\end{equation} then  for any  non-negative integers $n_1,\dots ,n_s$ we have 
    \begin{equation}
        \mathbb{E}\big[\prod_{j=1}^s \tr (\widetilde{T}_{n_j}(X)+b'_{n_j}\mathbf{I})\big]=\sum_{\Gamma\in \mathscr{D}_{s;\beta}} \mathfrak{F}_{\Gamma}(\{n_j\}).
    \end{equation}

\end{proposition}
\begin{proof}
Substitution of       \eqref{equ:x=T} into \eqref{equ:T=x} yields  the   orthogonal relation for $n\ge 1$
    \begin{equation}\label{equ:orth_relation}
        \sum_{m=0}^{[\frac{n}{{2}}]}(-1)^m   \frac{n}{n-m}\binom{n-m}{m}
        \binom{n-2m}{\frac{n-2m-j}{2}}=
        \delta_{n,j},
    \end{equation}
    where the  summation   runs   over all $n-2m-j \equiv 0\pmod2$ and $n-2m-j \geq 0$.  
   With $b_m$  defined in \eqref{bcat}, introduce a new sequence 
    \begin{equation}
        b'_n=\sum_{m=0}^{[\frac{n}{2}]}(-1)^m   \frac{n}{n-m}\binom{n-m}{m}
        b_{n-2m},
    \end{equation}
    for all  $n_j\ge 1$,  first by  \eqref{equ:T=x} and then by Proposition \ref{prop:ribbon}  we see   
    \begin{align}
        &\mathbb{E}\Big[\prod_{j=1}^s \tr (\widetilde{T}_{n_j}(X)+b'_{n_j}\mathbf{I})\Big]\notag\\
        &=\sum_{m_j\le n_j,j=1,\ldots s}\prod_{j=1}^{s}(-1)^{m_j} \frac{n_j}{n_j-m_j}\binom{n_j-m_j}{m_j}
        \mathbb{E}\big[\prod_{j=1}^{s}(\tr X^{m_j}+b_{m_j} N)\big]
        \\
        &=\sum_{m_j\le n_j,j=1,\ldots s}\prod_{j=1}^{s}(-1)^{m_j} \frac{n_j}{n_j-m_j}\binom{n_j-m_j}{m_j}{\sum_{t_j\le n_j-2m_j}}'\binom{n_j-2m_j}{\frac{n_j-2m_j-t_j}{2}} \sum_{\Gamma\in \mathscr{D}_{s;\beta}}\mathfrak{F}_{\Gamma}(\{t_j\}) 
        \\
        &=\sum_{t_j}\prod_{j=1}^s\delta_{n_j,t_j}\sum_{\Gamma\in \mathscr{D}_{s;\beta}}\mathfrak{F}_{\Gamma}(\{t_j\})
      =\sum_{\Gamma\in \mathscr{D}_{s;\beta}}\mathfrak{F}_{\Gamma}(\{n_j\}),
    \end{align}
    where in the last two  equalities  the orthogonal relation \eqref{equ:orth_relation} for $n\ge 1$ has been used.
    
    Next, we just compute the sequence $b'_n$  for even $n$ since it equals to zero for odd $n$.  Note that 
     \begin{equation}
         b_{n-2m}= 
          \binom{n-2m}{\frac{n}{2}-m-1}-\frac{1}{2}\binom{n-2m}{\frac{n}{2}-m},
     \end{equation}
    use the orthogonality  \eqref{equ:orth_relation} with $j=0,2$ and we  derive 
$  
        b'_n= \delta_{n,2}-\frac{1}{2} \delta_{n,0}.
  $   

This thus completes the proof. \end{proof}

Using the functional relation  \eqref{TU} between Chebyshev polynomials of the first kind 
  and second kind, we derive an explicit formula for the mixed moments associated with $U_n(x)$. To facilitate this analysis, we introduce a new family of diagrammatic functions via \eqref{frakF} defined as follows:
\begin{equation}\label{equ:F=f}
    F_{\Gamma}(\{n_j\}) = \sum_{\substack{
        1 \le m_j \le n_j \\
        m_j \equiv n_j \pmod{2}
    }}
    \mathfrak{F}_{\Gamma}(\{m_j\}).
\end{equation}
More explicitly,  
\begin{equation}\label{equ:F_formula}
    F_{\Gamma}(\{n_j\}) = 
    \sum_{\eta: V(\Gamma) \to [N]}
    \sum_{\substack{
        t_j \ge 0,\; w_e \ge 1 \\
        2t_j + \sum\limits_{e \in \partial D_j} w_e = n_j
    }}
    \prod_{(x, y) \in E_{\mathrm{int}}}
        p_{w_e}(\eta(x), \eta(y)) 
    \prod_{(z, w) \in E_b}
        (A^{w_e})_{\eta(z)\eta(w)},
\end{equation}
where the second summation is taken over all  possible  integers \( t_j \ge 0 \) and \( w_e \ge 1 \)  that satisfy the constraint     \( 2t_j + \sum_{e \in \partial D_j} w_e = n_j \) for each $j$.

Now we state the exact graphical expansion of   mixed moments for Chebyshev polynomials of second kind   that  is of  central  importance    in this paper.
\begin{theorem}  \label{Chebyshevmoment} Let   $X$  be the deformed   Gaussian      matrix  in   Definition \ref{def:inhomo},  
then for any non-negative integers $n_j\ge 0$, 
   we have  a diagrammatic expansion       
    \begin{equation}
        \mathbb{E}\Big[\prod_{j=1}^s \tr \big({U}_{n_j}\big(\frac{X}{2}\big)\big)\Big]=\sum_{\Gamma\in \mathscr{D}_{s;\beta}}F_{\Gamma}(\{n_j\}).
    \end{equation}
\end{theorem}
\begin{proof}
Using  the functional relation 
\begin{equation}
    \widetilde{T}_{n}(x)=U_{n}\big(\frac{x}{2}\big)-U_{n-2}\big(\frac{x}{2}\big), \quad n\ge 1,
\end{equation}
we see  
\begin{equation} \label{Usum}
  \sum_{\substack{
    1 \le m \le n \\
    m \equiv n \!\!\! \pmod{2}
  }}
  \bigl( \widetilde{T}_m(x) + b'_m \bigr)
  =
  U_n\!\left( \tfrac{x}{2} \right)
  + \bigl(
    -U_0\!\left( \tfrac{x}{2} \right) + b'_2
  \bigr)\mathbf{1}({n \equiv 0 \bmod 2}) 
  =
  U_n\!\left( \tfrac{x}{2} \right).
\end{equation}

Without loss of generality, we assume that $n_j\ge 1$ for all $j$. Then by \eqref{Usum} we have
\begin{align}
  \mathbb{E}\!\Bigl[\,
      \prod_{j=1}^{s}
      \tr \big(U_{n_j}\!\bigl(\tfrac{X}{2}\bigr)\big)
  \Bigr]
  &=
  \sum_{\substack{
          1 \le m_j \le n_j\\
          m_j \equiv n_j \pmod{2}
       }}
  \mathbb{E}\!\Bigl[\,
      \prod_{j=1}^{s}
      \tr\!\bigl(
        \tilde{T}_{n_j}(X) + b'_{n_j}\mathbf{I}
      \bigr)
  \Bigr]                                          \notag\\[4pt]
  &=
  \sum_{\substack{
          1 \le m_j \le n_j\\
          m_j \equiv n_j \pmod{2}
       }}
  \sum_{\Gamma \in \mathscr{D}_{s;\beta}}
      \mathfrak{F}_{\Gamma}\bigl(\{n_j\}_{j=1}^{s}\bigr)
      \notag\\[4pt]
  &=
  \sum_{\Gamma \in \mathscr{D}_{s;\beta}}
      F_{\Gamma}\bigl(\{n_j\}_{j=1}^{s}\bigr),
\end{align}
where in the last  two  equalities   Proposition~\ref{prop:T=Gamma}  and the definition  \eqref{equ:F=f} have been used respectively. 

This thus completes the proof.
\end{proof}

In order to convert full diagram sums into connected diagram sums, we introduce cluster decomposition of Chebyshev mixed moments.  

\begin{definition}[Cumulants] \label{Cumu}
The  joint cumulants of Chebyshev mixed moments can be defined recursively as follows:
\begin{equation}\label{equ:T}
\begin{gathered}
			\kappa_X(n_1)=\mathbb{E}\Big[\tr \big(U_{n_1}\big(\frac{X}{2}\big)\big)\Big],\\
			\kappa_X\big(n_1,n_2\big)=
            \mathbb{E}\Big[\tr \big(U_{n_1}\big(\frac{X}{2}\big)\big)\tr \big(U_{n_2}\big(\frac{X}{2}\big)\big)\Big]-\kappa_X\big(n_1\big)\kappa_X\big(n_2\big),\\
			\cdots	  \\
			\kappa_X\big(n_1,\cdots,n_s\big)=\mathbb{E}\Big[\tr\big( U_{n_1}\big(\frac{X}{2}\big)\big)\tr \big(U_{n_2}\big(\frac{X}{2}\big)\big)\cdots \tr \big(U_{n_s}(\frac{X}{2})\big)\Big]-\sum_{\Pi}\prod_{P\in \Pi}\kappa_X\big(\{n_j\}_{j\in P}\big),
\end{gathered}
\end{equation}
 where the summation  is over all nontrivial partitions $\Pi$  of $[s]$ (i.e., excluding the partition $[s]$ itself). 
 \end{definition}
 
\begin{lemma}\label{thm:u=F} For all integers $n_j\ge 1, j=1,\ldots,s$, we have    
\begin{equation}
    \kappa_X\big(n_1,\cdots,n_s\big)=\sum_{\Gamma\in \mathscr{D}_{s;\beta}^*} F_{\Gamma}(\{n_j\}_{j=1}^s).
\end{equation}
\end{lemma}
\begin{proof}

The case \( s = 1 \) is trivial. For \( s > 1 \),  we partition the set of diagrams   \(\mathscr{D}_{s;\beta}\) according to their connected components.
If \(\Gamma \notin \mathscr{D}_{s;\beta}^*\) (i.e., \(\Gamma\) is disconnected), 
then its associated diagram function factorizes into a product of diagram functions corresponding to each connected component. By construction, these terms are precisely canceled by the subtraction in \eqref{equ:T}.

Therefore, only the connected diagrams \(\Gamma \in \mathscr{D}_{s;\beta}^*\) contribute to the final expression, which concludes the proof.
\end{proof}

\section{Analysis of diagram  functions}\label{sec:section3}
In this section,  we focus on analyzing key analytic properties of diagram functions $F_{\Gamma}(\{n_j\})$ in \eqref{equ:F_formula} for  any connected diagram $\Gamma\in \mathscr{D}_{s;\beta}^*$, including upper-bound estimates and precise asymptotic behavior.

\subsection{Upper bound estimates}

\begin{proposition}\label{prop:F_upper_bound}
With the same notation and assumptions \ref{itm:B1}-\ref{itm:B2} in Definition \ref{mixingdef}, with $\|A_N\|_{\mathrm{op}}=a$, the following upper bounds hold.
\begin{enumerate}[label=(\roman*)] 
    \item \label{item:upper_noboundary} 
        If $\Gamma\in \mathscr{D}_{s;\beta}^*$  is a diagram without any open edge,  then 
         \begin{equation}\label{equ:3.2}
        \left|F_{\Gamma}(\{n_j\})\right|\le\frac{n^{|V|-1}}{(|V|-1)!}\Big(\frac{(\gamma t_N)\lor n}{N}\Big)^{|E|-|V|+1} N,
    \end{equation}
        while particularly 
        for $|E|\le \frac{n}{t_N}$, 
    \begin{equation}\label{equ:upper_no_boundary}
         \left|F_{\Gamma}(\{n_j\})\right|\le \frac{(5 \gamma n)^{|E|}}{(|E|-1)!}N^{|V|-|E|}.
    \end{equation}

    \item \label{item:upper_boundary} If   $\Gamma\in\mathscr{D}_{s;\beta}^{*}$ is a diagram with at least one open edge,  then  \begin{equation}
         \left|F_{\Gamma}(\{n_j\})\right|\le (1+a^n)r^{|V_{\mathrm{b}}|}\Big(\frac{(\gamma t_N)\lor n}{N}\Big)^{|E_{\mathrm{int}}|-|V_{\mathrm{int}}|}\frac{n^{|V|}}{(|V|)!},
    \end{equation}
    while particularly 
    for $|E_{\mathrm{int}}|\le \frac{n}{t_N}$, 
    \begin{equation}\label{equ:upper_boundary}
         \left|F_{\Gamma}(\{n_j\})\right|\le 10^{|E|}\gamma^{|E_{\mathrm{int}}|}(1+a^n)r^{|V_{\mathrm{b}}|}\frac{n^{|E|}}{(|E|-1)!}N^{|V_{\mathrm{int}}|-|E_{\mathrm{int}}|}.
    \end{equation}
\end{enumerate}

\end{proposition}

\begin{proof}  The proof is in three steps.  We first consider  \ref{item:upper_boundary} with $|E_{\mathrm{int}}|\le  {n}/{t_N}$, then \ref{item:upper_boundary} with general $|E_{\mathrm{int}}|$, and thirdly turn to case  \ref{item:upper_noboundary}.

{\bf Step 1: Case  \ref{item:upper_boundary} with $|E_{\mathrm{int}}|\le  {n}/{t_N}$.}

Starting from the exact expression  of $F_{\Gamma}$ defined in \eqref{equ:F_formula}, we  take the absolute value and divide vertices into interior and open ones to derive  
\begin{align}\label{equ:3.5}
    \left|F_{\Gamma}(\{n_j\})\right|
    &\le 
    \sum_{\eta: V_{\mathrm{b}}(\Gamma) \to [N]}
    \sum_{\substack{
        2 t_j + \sum_{e \in\partial D_j} w_e = n_j
    }}
         \prod_{(z,w) \in E_{\mathrm{b}}}  \left|(A^{w_e})_{\eta(z)\eta(w)}
     \right|
    \sum_{\eta: V_{\mathrm{int}}(\Gamma) \to [N]}
    \prod_{(x,y) \in E_{\mathrm{int}}} p_{w_e}(\eta(x), \eta(y)) \notag \\
    &\le 
    \sum_{\eta: V_{\mathrm{b}}(\Gamma) \to [N]}
    \sum_{\substack{
        \sum_{e \in E} w_e \le n 
    }}
        \prod_{(z,w) \in E_{\mathrm{b}}} \left|(A^{w_e})_{\eta(z)\eta(w)}
    \right|
    \sum_{\eta: V_{\mathrm{int}}(\Gamma) \to [N]}
    \prod_{(x,y) \in E_{\mathrm{int}}} p_{w_e}(\eta(x), \eta(y)),
\end{align} where $n:=\sum_{j=1}^s n_j$,  we have added  all $s$ linear restrictions together and forgot the evenness of the sum restriction  in the last inequality.

Fix a labeling $\eta$ on the open vertex set $V_{\mathrm{b}}$ and fix an edge weight $w_e$ for all $e\in E_{\mathrm{b}}$. Denote by  $m=n-\sum_{e\in E_{\mathrm{b}}}w_e$  the  weight sum on all interior edges,  and also introduce edge subset $S=\{e\in E_{\mathrm{int}}|w_e\le t_N\}$ and     $S^c=E_{\mathrm{int}}\backslash S$. Obviously, $|S^c|\le  {m}/{t_N}$ and  we further get  an upper bound according to the set $S$
\begin{equation}
\begin{aligned}
    &\sum_{\substack{\sum\limits_{e\in E_{\mathrm{int}}} w_e \le m}}
    \;\sum_{\eta:V_{\mathrm{int}}(\Gamma) \to [N]}
    \;\prod_{(x,y) \in E_{\mathrm{int}}} p_{w_e}(\eta(x), \eta(y))\\
    &\le
    \;\sum_{\eta: V_{\mathrm{int}}(\Gamma) \to [N]}
      \sum_{S \subset E_{\mathrm{int}}}
    \;\sum_{\substack{
      w_e \le t_N, \;\; \forall e \in S \\
      w_e \ge t_N, \;\; \forall e \in S^c \\
      \sum\limits_{e \in S^c} w_e \le m
    }}
    \prod_{e \in E} p_{w_e}(\eta(x), \eta(y)).
\end{aligned}
\end{equation}
\begin{figure}
    \centering
    \includegraphics[width=\linewidth]{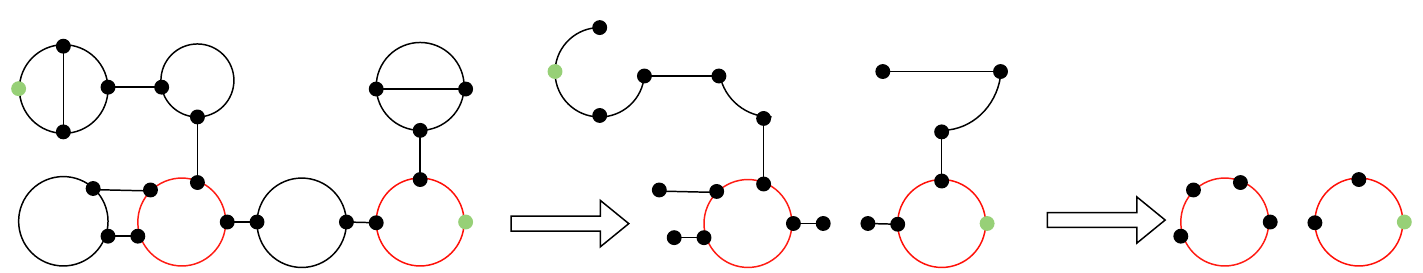}
    \caption{\textbf{Step 1:} Remove all edges in \(\mathcal{F}^c\), leaving a spanning forest \(\mathcal{F}\) where each tree has exactly one ``open’’ root vertex. \textbf{Step 2:} Sum over labels on each tree (collapsing them), so only the open edges and open vertices remain.
    }
    \label{fig:forest}
\end{figure}

Given the interior edge set $E_{\mathrm{int}}$ viewed as a graph, we select one spanning forest $\mathcal{F}$ such that each  connected component  of it contains exactly one open vertex, see Figure \ref{fig:forest} for example.
  In other word,  the subset of interior edges not in $\mathcal{F}$, denoted by $\mathcal{F}^c:=E_{\mathrm{int}}\backslash \mathcal{F}$,   is a minimal cut such that all open vertices are disconnected. 
    We first sum up all edges in $\mathcal{F}^c$. So for the edges in $S\cap \mathcal{F}^c$, by the assumption  \ref{itm:B1}   in Definition \ref{mixingdef}, we have
    \begin{equation}
        \sum_{w_e\le  t_N, e\in S\cap \mathcal{F}^c}\prod_{e\in S\cap \mathcal{F}^c}p_{w_e}(\eta(x),\eta(y))\le \big(\frac{\gamma t_N}{N}\big)^{|S\cap \mathcal{F}^c|},
    \end{equation}
    where $|S\cap \mathcal{F}|$   denotes the number of edges.
   While  for every  edge from  $S^c\cap \mathcal{F}^c$  we see from the assumption \ref{itm:B2} that  $p_{w_e}(\eta(x),\eta(y))\le (1+\delta)N^{-1}$,  which implies 
\begin{equation}\label{equ:3.8}
  \sum_{\substack{
    w_e \ge t_N\;, e \in S^c \cap \mathcal{F}^c \\
    \sum\limits_{e \in S^c \cap \mathcal{F}^c} w_e \le m
  }}
  \;\prod_{e \in S^c \cap \mathcal{F}^c}
    p_{w_e}(\eta(x), \eta(y))
  \;\le\;
  \frac{
    1
  }{
    \bigl( |S^c \cap \mathcal{F}^c| \bigr)!
  }\left( \dfrac{(1+\delta)m}{N} \right)^{|S^c \cap \mathcal{F}^c|}.
\end{equation}   
Here in the last inequality simple counting for integer linear eqnation  has been used.  
Hence,  use the above two upper bounds first and then sum over trees of the  forest $\mathcal{F}$,   we obtain 
\begin{align}
  &\sum_{S \subset E_{\mathrm{int}}}
    \;\sum_{\eta: V_{\mathrm{int}}(\Gamma) \to [N]}
    \;\sum_{\substack{
      w_e \le t_N, \;\;\forall e \in S \\
      w_e \ge t_N, \;\;\forall e \in S^c \\
      \sum\limits_{e \in S^c} w_e \le m
    }}
    \prod_{e \in E_{\mathrm{int}}}
      p_{w_e}(\eta(x), \eta(y))\notag
  \\
  &\le \sum_{S \subset E_{\mathrm{int}}}
    \left( \frac{(1+\delta)m}{N} \right)^{|S^c \cap \mathcal{F}^c|}
    \frac{1}{\bigl( |S^c \cap \mathcal{F}^c| \bigr)!}
    \left( \frac{\gamma t_N}{N} \right)^{|S \cap \mathcal{F}^c|}
    \sum_{\substack{
      \sum\limits_{e \in S^c \cap \mathcal{F}} w_e \le m \\
      w_e \le t_N,\;\;\forall e \in S \cap \mathcal{F}
    }}
    \sum_{\eta: V_{\mathrm{int}}(\Gamma) \to [N]}
    \prod_{e \in \mathcal{F}} p_{w_e}(\eta(x), \eta(y))\tag{by \eqref{equ:3.8}}
  \\
  &= \sum_{S \subset E_{\mathrm{int}}}
    \left( \frac{(1+\delta)m}{N} \right)^{|S^c \cap \mathcal{F}^c|}
    \frac{1}{\bigl( |S^c \cap \mathcal{F}^c| \bigr)!}
    \left( \frac{\gamma t_N}{N} \right)^{|S \cap \mathcal{F}^c|}
    \sum_{\substack{
      \sum\limits_{e \in S^c \cap \mathcal{F}} w_e \le m \\
      w_e \le t_N,\;\; e \in S \cap \mathcal{F}
    }} 1\tag{summing over trees}
  \\
  &\le \sum_{S \subset E_{\mathrm{int}}}
    \left( \frac{(1+\delta)m}{N} \right)^{|S^c \cap \mathcal{F}^c|}
    \frac{1}{\bigl( |S^c \cap \mathcal{F}^c| \bigr)!}
    \left( \frac{\gamma t_N}{N} \right)^{|S \cap \mathcal{F}^c|}
    \frac{m^{|S^c \cap \mathcal{F}|}}{(|S^c \cap \mathcal{F}|)!}
    t_N^{|S \cap \mathcal{F}|}.
\end{align}

Setting   $k=|S^c|\le \frac{m}{ t_N}\land |E_{\mathrm{int}}|$,   we further get from the above estimate that 
\begin{align}\label{equ:3.10}
  &\sum_{S \subset E_{\mathrm{int}}}
    \;\sum_{\eta: V_{\mathrm{int}}(\Gamma) \to [N]}
    \;\sum_{\substack{
      w_e \le t_N, \;\;\forall e \in S \\
      w_e \ge t_N, \;\;\forall e \in S^c \\
      \sum\limits_{e \in S^c} w_e \le m
    }}
    \prod_{e \in E}
      p_{w_e}(\eta(x), \eta(y))\\
  &\le \sum_{k \le \frac{m}{t_N} \wedge |E_{\mathrm{int}}|}
    \binom{|E_{\mathrm{int}}|}{k}
    \frac{(2+2\delta)^k}{k!}
    \cdot m^k (\gamma t_N)^{|E_{\mathrm{int}}| - k}
    \cdot N^{-|\mathcal{F}^c|}
    \tag{using \( k = |S^c| \) and \( \binom{k}{|\mathcal{F}|} \le 2^k \)} \\
  &\le (4+4\delta)^{|E_{\mathrm{int}}|} \cdot
    N^{|V_{\mathrm{int}}| - |E_{\mathrm{int}}|}
    \cdot (\gamma t_N)^{|E_{\mathrm{int}}|}
    \sum_{k \le \frac{m}{t_N} \wedge |E_{\mathrm{int}}|}
    \frac{1}{k!}
    \left( \frac{m}{\gamma t_N} \right)^k\tag{$\binom{|E_{\mathrm{int}}|}{k}\le 2^{|E_{\mathrm{int}}|}$, $|\mathcal{F}^c|=|E_{\mathrm{int}}| - |V_{\mathrm{int}}|$}\\
    &\le (5\gamma t_N)^{|E_{\mathrm{int}}|}
    \cdot N^{|V_{\mathrm{int}}| - |E_{\mathrm{int}}|}
    \sum_{k \le \frac{m}{t_N} \wedge |E_{\mathrm{int}}|}
    \frac{1}{k!}
    \left( \frac{m}{\gamma t_N} \right)^k, \notag
\end{align}
where  in the last inequality    $\delta\in(0,0.1)$ has been used.

This upper bound is irrelevant to the labeling on open vertices and weight on open edges. Since the open edges are also connected end to end into the union of several circles (as the boundary of Riemann surface are union of circles) for any given circle with $k$ open vertices, by the spectral decomposition $A_N=\sum_{t=1}^r a_t \mathbf{q}_{t}^{}\mathbf{q}_{t}^{*}$ with $Q=(\mathbf{q}_{1},\ldots,\mathbf{q}_{N})$ defined in  \eqref{SpetralA}, we have 
\begin{align}
    &\sum_{x_1,\ldots, x_k}
    \left|
        (A^{s_1})_{x_1x_2} (A^{s_2})_{x_2x_3} \cdots (A^{s_k})_{x_kx_1}
    \right|
    =
    \sum_{x_1,\ldots, x_k}
    \Big|
        \sum_{1 \le t_1, \ldots, t_k \le r}
        \prod_{i=1}^{k}
        a_{t_i}^{s_i} \,
        \big(\mathbf{q}_{t_i} \mathbf{q}_{t_i}^{*}\big)_{x_i x_{i+1}}
    \Big|\notag \\[4pt]
    &\le
    \sum_{1 \le t_1, \ldots, t_k \le r}
    \sum_{x_1,\ldots, x_k}
        \prod_{i=1}^{k} \left|
        a_{t_i}^{s_i} \,
        \big(\mathbf{q}_{t_i} \mathbf{q}_{t_i}^*\big)_{x_i x_{i+1}}
    \right| \le
    a^{s_1 + \cdots + s_k}
    \sum_{1 \le t_1, \ldots, t_k \le r}
    \sum_{x_1, \ldots, x_k}
    \Big|
        \prod_{i=1}^{k} \big(\mathbf{q}_{t_i} \mathbf{q}_{t_{i+1}}^*\big)_{x_i x_{i}}
    \Big| \notag
    \\[4pt]
    &\le
    a^{s_1 + \cdots + s_k}
    \sum_{1 \le t_1, \ldots, t_k \le r} 1  =
    a^{s_1 + \cdots + s_k} \, r^k,\label{equ:3.11}
\end{align}
where  the  Cauchy-Schwarz inequality has been used in the last inequality.
So by \eqref{equ:3.11}   we obtain 
\begin{align} 
    \sum_{\eta: V_{\mathrm{b}}(\Gamma) \to [N]}
    \sum_{\substack{
        \sum\limits_{e \in E_{\mathrm{b}}} w_e = n - m
    }}
    \Big|
        \prod_{(z, w) \in E_{\mathrm{b}}}
        (A^{w_e})_{\eta(z)\eta(w)}
    \Big|
    &\le
    a^{\sum_{e \in E_{\mathrm{b}}} w_e} \cdot r^{|V_{\mathrm{b}}|}
    \sum_{\substack{
        \sum\limits_{e \in E_{\mathrm{b}}} w_e = n - m
    }}
    1
     \notag\\[4pt]
    &\le
    a^{n - m} \cdot r^{|V_{\mathrm{b}}|} \cdot
    \frac{(n - m)^{|E_{\mathrm{b}}| - 1}}{(|E_{\mathrm{b}}| - 1)!} \label{bsum}.
\end{align}
Hence, combining \eqref{equ:3.10} and \eqref{bsum} yields 
    \begin{align}\label{equ:3.13}
         \left|F_{\Gamma}(\{n_j\})\right|\le (5\gamma t_N)^{|E_{\mathrm{int}}|}N^{|V_{\mathrm{int}}|-|E_{\mathrm{int}}|}r^{|V_{\mathrm{b}}|}\sum_{m=|E_{\mathrm{int}}|}^{n-|E_{\mathrm{b}}|}\sum_{k\le \frac{m}{ t_N}\land |E_{\mathrm{int}}|}\frac{1}{k!}\big(\frac{m}{\gamma t_N}\big)^k a^{n-m}\frac{(n-m)^{|E_{\mathrm{b}}|-1}}{(|E_{\mathrm{b}}|-1)!}.
    \end{align}
Note that 
\begin{align}
    &\sum_{m = |E_{\mathrm{int}}|}^{n - |E_{\mathrm{b}}|}
    \mathbf{1}_{\left\{ k \le \frac{m}{t_N} \wedge |E_{\mathrm{int}}| \right\}}
    \cdot m^k (n - m)^{|E_{\mathrm{b}}| - 1} \cdot a^{n - m} \notag \\
    &\le \sum_{m = |E_{\mathrm{int}}|}^{n - |E_{\mathrm{b}}|}
    \mathbf{1}_{\left\{ m \ge t_N k \right\}}
    \cdot m^k (n - m)^{|E_{\mathrm{b}}| - 1} \cdot a^{n - m} \notag \\
    &\le (1 + a^n) \cdot n^{|E_{\mathrm{b}}| - 1}
    \sum_{m = 1}^n m^k  \le (1 + a^n) \cdot \frac{(n + 1)^{|E_{\mathrm{b}}| + k}}{k + 1}\label{equ:3.14}, 
\end{align}
when   $|E_{\mathrm{int}}|\le  {n}/{ t_N}$,     combine \eqref{equ:3.13} and \eqref{equ:3.14} and we  obtain 
\begin{align}
     \left|F_{\Gamma}(\{n_j\})\right|
    &\le (5 \gamma t_N)^{|E_{\mathrm{int}}|} \,
    N^{|V_{\mathrm{int}}| - |E_{\mathrm{int}}|} \,
    r^{|V_{\mathrm{b}}|} \, (1 + a^n) \,
    \sum_{k \le \frac{m}{t_N} \wedge |E_{\mathrm{int}}|}
    \frac{1}{(k+1)!}
    \left( \frac{n+1}{\gamma t_N} \right)^k
    \frac{(n+1)^{|E_{\mathrm{b}}|}}{(|E_{\mathrm{b}}| - 1)!}\notag
     \\[6pt]
    &\le (5 \gamma t_N)^{|E_{\mathrm{int}}|} \,
    N^{|V_{\mathrm{int}}| - |E_{\mathrm{int}}|} \,
    r^{|V_{\mathrm{b}}|} \, (1 + a^n) \,
    \frac{1}{(|E_{\mathrm{int}}|)!}
    \left( \frac{n+1}{t_N} \right)^{|E_{\mathrm{int}}|}
    \frac{(n+1)^{|E_{\mathrm{b}}|}}{(|E_{\mathrm{b}}| - 1)!}
    \notag  \\[6pt]
    &\le 10^{|E|} \gamma^{|E_{\mathrm{int}}|} (1 + a^n) r^{|V_{\mathrm{b}}|} \,
    \frac{n^{|E|}}{(|E| - 1)!} \,
    N^{|V_{\mathrm{int}}| - |E_{\mathrm{int}}|}
    \label{equ:3.15},
\end{align} 
where in the second inequality we bound the sum by the term at \(k = |E_{\mathrm{int}}|\).

{\bf Step 2: Case  \ref{item:upper_boundary} with general  $|E_{\mathrm{int}}|$.}  

For general $|E_{\mathrm{int}}|$,  as in the derivation of   \eqref{equ:3.5} we can obtain  
\begin{align} \label{ub-eq1}
    \left|F_{\Gamma}(\{n_j\})\right|
    &\le 
    \sum_{\substack{w_e \\ \sum w_e \le n}} 
    \sum_{\eta: V(\Gamma) \to [N]} 
    \prod_{e \in E_{\mathrm{int}}} p_{w_e}(\eta(x), \eta(y))
    \prod_{(z,w) \in E_{\mathrm{b}}} \left| (A^{w_e})_{\eta(z) \eta(w)} \right|. 
    \end{align}
 The assumption  \ref{itm:B1}   in Definition \ref{mixingdef} shows that  \begin{equation}
 \sum\limits_{w_e=1}^{n}p_{w_e}(\eta(x),\eta(y))\le \frac{1}{N}(\gamma t_N)\lor n, \quad \forall e\in \mathcal{F}^c,\end{equation} 
  from which summing over $w_e$  from 1 to $n$ (relaxing the sum restriction) on the right-side hand of \eqref{ub-eq1} for every  edge $e\in \mathcal{F}^c$,   we obtain
 \begin{align}  
    \left|F_{\Gamma}(\{n_j\})\right|
    &\le 
    \sum_{\substack{\sum\limits_{e \in E_{\mathrm{b}} \cup \mathcal{F}} w_e \le n  }}
    \left(\frac{(\gamma t_N) \lor n}{N}\right)^{|\mathcal{F}^c|}
    \sum_{\eta: V(\Gamma) \to [N]} 
    \prod_{e \in \mathcal{F}} p_{w_e}(\eta(x), \eta(y)) 
    \prod_{(z,w) \in E_{\mathrm{b}}} \left|(A^{w_e})_{\eta(z) \eta(w)}\right| 
    \notag \\[6pt]
    &= 
    \sum_{\substack{\sum\limits_{e \in E_{\mathrm{b}} \cup \mathcal{F}} w_e \le n}}
    \left(\frac{(\gamma t_N) \lor n}{N}\right)^{|\mathcal{F}^c|}
    \sum_{\eta: V_{\mathrm{b}}(\Gamma) \to [N]}
    \prod_{(z,w) \in E_{\mathrm{b}}} \left|(A^{w_e})_{\eta(z) \eta(w)}\right| 
    \tag{summing over tree} \\[6pt]
    &\le 
    \sum_{\substack{\sum\limits_{e \in E_{\mathrm{b}} \cup \mathcal{F}} w_e \le n}}
    \left(\frac{(\gamma t_N) \lor n}{N}\right)^{|\mathcal{F}^c|} (1 + a^n) r^{|V_{\mathrm{b}}|}
    \tag{by \eqref{equ:3.11}} \\[6pt]
    &\le 
    (1 + a^n) r^{|V_{\mathrm{b}}|} \left(\frac{(\gamma t_N) \lor n}{N}\right)^{|\mathcal{F}^c|} 
    \frac{n^{|E_{\mathrm{b}} \cup \mathcal{F}|}}{(|E_{\mathrm{b}} \cup \mathcal{F}|)!} 
    \notag \\[6pt]
    &=
    (1 + a^n) r^{|V_{\mathrm{b}}|} \left(\frac{(\gamma t_N) \lor n}{N}\right)^{|E_{\mathrm{int}}| - |V_{\mathrm{int}}|} 
    \frac{n^{|E_{\mathrm{b}}| + |V_{\mathrm{int}}|}}{(|E_{\mathrm{b}}| + |V_{\mathrm{int}}|)!}.
    \label{equ:3.16}
 \end{align}
 Here  we sum over the evaluation of $\eta$ along the forest    in the second equality  and use  \eqref{equ:3.11} in the third inequality.

Since $|E_{\mathrm{b}}|=|V_{\mathrm{b}}|$, this indeed completes the general case. 

{\bf Step 3: Case  \ref{item:upper_noboundary}.} 

In this case, since $\Gamma$ contains no open edges, the proof follows similarly but more straightforwardly. The key distinction lies in taking a spanning tree instead of a spanning forest. By fixing the labeling of one vertex and viewing   the remaining vertices as $V_{\mathrm{int}}$ with corresponding edges $E_{\mathrm{int}}$, we follow the procedure in \eqref{equ:3.5}-\eqref{equ:3.10}. Note that $m = n$ here, and when $|E| \leq  {n}/{t_N}$, we obtain 
\begin{align}
    F_{\Gamma}(\{n_j\}) 
    &\le 
    N \cdot 5^{|E_{\mathrm{int}}|} \cdot 
    N^{|V_{\mathrm{int}}| - |E_{\mathrm{int}}|} \cdot 
    (\gamma t_N)^{|E_{\mathrm{int}}|} 
    \sum_{k \le \frac{n}{t_N} \wedge |E_{\mathrm{int}}|} 
    \frac{1}{k!} \left( \frac{n}{\gamma t_N} \right)^k 
    \tag{similar to \eqref{equ:3.10}} \\[6pt]
    &\le 
    N \cdot 5^{|E_{\mathrm{int}}|} \cdot 
    N^{|V_{\mathrm{int}}| - |E_{\mathrm{int}}|} \cdot 
    (\gamma t_N)^{|E_{\mathrm{int}}|} \cdot 
    \frac{1}{(|E_{\mathrm{int}}| - 1)!} 
    \left( \frac{n}{t_N} \right)^{|E_{\mathrm{int}}|}
    \tag{similar to \eqref{equ:3.15}} \\[6pt]
    &= 
    (5\gamma)^{|E_{\mathrm{int}}|} \cdot 
    \frac{n^{|E_{\mathrm{int}}|}}{(|E_{\mathrm{int}}| - 1)!} 
    \cdot N^{|V_{\mathrm{int}}| - |E_{\mathrm{int}}| + 1}
    \tag{in this case $|E|=|E_{\mathrm{int}}|, ~ |V|=|V_{\mathrm{int}}|-1$} \\[6pt]
    &= 
    \frac{(5 \gamma n)^{|E|}}{(|E| - 1)!} \cdot 
    N^{|V| - |E|}.
\end{align}

If $|E| \geq {n}/{t_N}$, we also fix the labeling of one vertex and designate the remaining vertices as $V_{\mathrm{int}}$ with corresponding edges $E_{\mathrm{int}}$. Noting that $|V_{\mathrm{b}}| = |E_{\mathrm{b}}| = 0$ and setting $a = 0$, we follow the procedure in \eqref{equ:3.16} to obtain
\begin{align}
    F_{\Gamma}(\{n_j\})&\le N (\frac{(\gamma t_N)\lor n}{N})^{|E_{\mathrm{int}}|-|V_{\mathrm{int}}|}\frac{n^{|E_{\mathrm{b}}|+|V_{\mathrm{int}}|}}{(|E_{\mathrm{b}}|+V_{\mathrm{int}})!}\notag\\
    &=\frac{n^{|V|-1}}{(|V|-1)!}(\frac{(\gamma t_N)\lor n}{N})^{|E|-|V|+1} N.
\end{align}

This completes case \ref{item:upper_noboundary},  and thus   the entire   proof.
\end{proof}

\subsection{Asymptotics for  diagram functions}
In order to  establish  approximations for   diagram functions,  let $\widetilde{F}_{\Gamma}(\{n_j\}_{j=1}^s)$ represent   diagram functions  corresponding to the special  case where  variance profile $\sigma_{ij}^2=1/N$ for all $1\le i,j\le N$.
\begin{theorem}\label{thm:F=tildeF}
Given  any fixed diagram $\Gamma$ with or without open edges, if all $n_j \gg t_N$ for $j=1,\ldots, s$,    and $\| A_N\|_{\mathrm{op}} \le 1+\tau n^{-1}$ with  $n=\sum_{j=1}^s n_j$,  for any fixed integer $s\geq 1$  and     constant   $\tau >0$,  then  
    \begin{equation}\label{equ:F_asy_equal}
     F_{\Gamma}(\{n_j\})=\widetilde{F}_{\Gamma}(\{n_j\}_{j=1}^s)+O\Big(\big(\frac{t_N}{n}+\delta\big)  n^{|E|}N^{|V_{\mathrm{int}}|-|E_{\mathrm{int}}|}\Big).
    \end{equation}
\end{theorem}
\begin{proof}
We consider the case that $\Gamma$ has open edges   since the no open edge case is quite similar.  In this case $|V|-|E|=|V_{\mathrm{int}}|-|E_{\mathrm{int}}|$.

Comparing the diagram functions $F_{\Gamma}$ in \eqref{equ:F_formula} with their GOE/GUE analogue $\widetilde{F}_{\Gamma}$, which is exactly expressed via 
\begin{equation}
    \widetilde{F}_{\Gamma}(\{n_j\}_{j=1}^s) = 
    \sum_{\eta : V(\Gamma) \to [N]}
    \sum_{\substack{2t_j + \sum_{e \in \partial D_j} w_e = n_j}}
    \prod_{(x, y) \in E_{\mathrm{int}}} \frac{1}{N}
    \prod_{(z, w) \in E_{\mathrm{b}}} (A^{w_e})_{\eta(z)\eta(w)},
\end{equation} we obtain
\begin{align}
    &\left| F_{\Gamma}(\{n_j\}_{j=1}^s) - \widetilde{F}_{\Gamma}(\{n_j\}_{j=1}^s) \right| \notag\\
    &\le 
    \sum_{\eta : V(\Gamma) \to [N]}
    \sum_{\substack{2t_j + \sum_{e \in \partial D_j} w_e = n_j}}
    \bigg| 
        \prod_{(x, y) \in E_{\mathrm{int}}} p_{w_e}(\eta(x), \eta(y)) 
        - 
        \prod_{(x, y) \in E_{\mathrm{int}}} \frac{1}{N}
    \bigg| 
    \prod_{(z, w) \in E_{\mathrm{b}}} 
    \left| (A^{w_e})_{\eta(z)\eta(w)} \right|.
\end{align}    
Introduce   partitions of interior edges,  denoted by  $S := \{e \in E_{\mathrm{int}} \mid w_e \le t_N\}.$ Fix the weights on the open edges and   labelings on all vertices,    consider the sum 
\begin{align}
    &\sum_{\substack{w_e,\ e \in E_{\mathrm{int}} \\ 2t_j + \sum_{e \in \partial D_j} w_e = n_j}}
    \bigg| 
        \prod_{(x, y) \in E_{\mathrm{int}}} p_{w_e}(\eta(x), \eta(y)) 
        - 
        \prod_{(x, y) \in E_{\mathrm{int}}} \frac{1}{N}
    \bigg|= \sum_{S \subset E_{\mathrm{int}}}
    \sum_{\substack{w_e,\ e \in E_{\mathrm{int}} \\ 2t_j + \sum_{e \in \partial D_j} w_e = n_j}}\notag\\
    &
    \bigg| 
        \prod_{(x, y) \in E_{\mathrm{int}}} p_{w_e}(\eta(x), \eta(y)) 
        - 
        \prod_{(x, y) \in E_{\mathrm{int}}} \frac{1}{N}
    \bigg| \mathbf{1}\left(\{w_e \le t_N\}_{e \in S} \right) 
    \cdot \mathbf{1}\left(\{w_e \ge t_N\}_{e \in S^c} \right). 
\end{align}

For  the trivial partition    $S=\emptyset$,   by   assumption  \ref{itm:B2}   of Definition \ref{mixingdef}   we have
\begin{equation}
     \bigg|\prod_{(x,y) \in E_{\mathrm{int}}}Np_{w_e}(\eta(x),\eta(y))-1\bigg| \le (1+\delta)^{|E_{\mathrm{int}}|}-1\le 2^{|E_{\mathrm{int}}|}\delta, 
\end{equation}
for fixed diagram 
and $N\ge N_{\delta}$,  which further implies  
\begin{equation}
    \bigg|\prod_{(x,y) \in E_{\mathrm{int}}}p_{w_e}(\eta(x),\eta(y))- \prod_{(x,y) \in E_{\mathrm{int}}}\frac{1}{N}\bigg|\le 2^{|E_{\mathrm{int}}|}\delta \,\bigg(\prod_{(x,y) \in E_{\mathrm{int}}}p_{w_e}(\eta(x),\eta(y))+\prod_{(x,y) \in E_{\mathrm{int}}}\frac{1}{N}\bigg).
\end{equation}
While for the non-trivial partition $S\ne \emptyset$, we use a trivial bound
\begin{equation}
    \bigg|\prod_{(x,y) \in E_{\mathrm{int}}}p_{w_e}(\eta(x),\eta(y))- \prod_{(x,y) \in E_{\mathrm{int}}}\frac{1}{N}\bigg|\le   \prod_{(x,y) \in E_{\mathrm{int}}}p_{w_e}(\eta(x),\eta(y))+\prod_{(x,y) \in E_{\mathrm{int}}}\frac{1}{N}.
\end{equation}

We proceed as in   the derivation  of \eqref{equ:3.10}. 
 Note that $S=\emptyset$  exactly corresponds  to the case $k=|E_{\mathrm{int}}|$ there, and also that  $ {n}/{ t_N}\ge |E_{\mathrm{int}}|$  for sufficiently  large $N$ and fixed  $\Gamma$, similar to \eqref{equ:3.10}  we conclude that   
\begin{align}
     \left| F_{\Gamma}(\{n_j\}) - \widetilde{F}_{\Gamma}(\{n_j\}_{j=1}^s) \right| &\le 
    2 \cdot 5^{|E_{\mathrm{int}}|} 
    N^{|V_{\mathrm{int}}| - |E_{\mathrm{int}}|} 
    (\gamma t_N)^{|E_{\mathrm{int}}|} 
    r^{|V_{\mathrm{b}}|} 
    (1 + a^n)\notag\\
    & \times 
    \sum_{k \le \frac{m}{t_N} \wedge |E_{\mathrm{int}}|}
    \frac{
        \mathbf{1}(k < |E_{\mathrm{int}}|) 
        + 2^{|E_{\mathrm{int}}|} \delta \cdot \mathbf{1}(k = |E_{\mathrm{int}}|)
    }{
        (k + 1)!
    }
    \left( \frac{n}{\gamma t_N} \right)^k 
    \frac{n^{|E_{\mathrm{b}}|}}{(|E_{\mathrm{b}}| - 1)!} \notag\\
    &\le 
    \Big( \frac{2 \gamma t_N}{n} + 2^{|E_{\mathrm{int}}| + 1} \delta \Big)
    \cdot 
    10^{|E|} 
    \gamma^{|E_{\mathrm{int}}|} 
    (1 + a^n) 
    r^{|V_{\mathrm{b}}|} 
    \frac{n^{|E|}}{(|E| - 1)!} 
    N^{|V_{\mathrm{int}}| - |E_{\mathrm{int}}|}.
\end{align}

Since $r$ and   $|E_{\mathrm{int}}|$ are fixed,  
  the desired result  
\begin{equation}
    F_{\Gamma}(\{n_j\})=\widetilde{F}_{\Gamma}(\{n_j\}_{j=1}^s)+O\Big(\big(\frac{t_N}{n}+\delta\big)  n^{|E|}N^{|V_{\mathrm{int}}|-|E_{\mathrm{int}}|}\Big) 
\end{equation}
immediately follows.
\end{proof}

\begin{remark}[\bf Sharpness]
    Meanwhile, for the case $A=\text{diag}(a,0,\ldots 0)$ with $|a-1|\le \tau n^{-1}$, one can easily show
    \begin{equation}
    \widetilde{F}_{\Gamma}(\{n_j\}_{j=1}^s)=\Omega\left((\min_{j} n_j)^{|E|}N^{|V|-|E|}\right).
\end{equation}
    Thus, the upper bound of Proposition \ref{prop:F_upper_bound} \ref{item:upper_noboundary} and \ref{item:upper_boundary} and the error term in \eqref{equ:F_asy_equal} are of good sharpness.
    
\end{remark}

\section{Asymptotic equivalence of mixed moments}\label{sec:section4}

\subsection{Mixed moments in Gaussian case}

 Crucially, when both $|E|$ and $|V|$ increase by one, the upper bounds established in \eqref{equ:upper_no_boundary} and \eqref{equ:upper_boundary} remain non-decreasing. To reduce   arbitrary diagrams $\Gamma\in \mathscr{D}_{s;\beta}^*$ into trivalent forms where each vertex has degree at most 3 while preserving $|E|-|V|$, we will construct a vertex-splitting map in the following lemma. For this, we need to introduce typical connected diagrams.

\begin{definition}[Typical   diagram]
A \emph{typical connected diagram} is a connected diagram 
$\Gamma\in \mathscr{D}_{s;\beta}^*$   satisfying    the following conditions: (i) exactly $s$  marked vertices have degree 2, and (ii) every other (unmarked) vertex has degree 3.
Introduce   a parameterization 
\begin{equation}
  |E| = 3\ell + s,
  \quad
  |V| = 2\ell + s,
\end{equation}
and denote  the  set of all typical connected diagrams with $\ell=|E(\Gamma)| - |V(\Gamma)|$ and   $s$ marked points 
  \begin{equation}
   \mathcal{T}_{\ell,s}
  := \Bigl\{
     \Gamma \in \mathscr{D}_{s;\beta}^* : 
       \substack{
      \substack{  \text{ $\Gamma$ is a typical diagram with $s$ marked points, } \\
      |E(\Gamma)|=3\ell+s,  \quad  |V(\Gamma)| = 2\ell+s}
       }
  \Bigr\}
 \end{equation}
and  the set of all typical connected diagrams with $s$ marked points 
\begin{equation}
    \mathcal{T}_{s} := \bigcup_{\ell \ge -1} \mathcal{T}_{\ell,s}.
\end{equation}
In the case $\ell = -1$, the set $\mathcal{T}_{\ell,s}$ is non-empty and consists of a single point diagram if and only if $s = 1$.
\end{definition}

\begin{lemma}[Vertex splitting]\label{lem:vertex_split}
There exists a map from the set of diagrams   with all unmarked vertices of degree at least $3$ to the set of typical connected diagrams 
\begin{equation}
    \phi : \mathscr{D}_{s;\beta}^* \to \mathcal{T}_{\ell, s}, \quad \Gamma \mapsto \phi(\Gamma)
\end{equation}
  such that  the following three conditions hold: 
\begin{itemize}
    \item[(i)] (Preserving $|E|-|V|$) 
    \begin{equation}
        |E(\phi(\Gamma))| - |V(\phi(\Gamma))| = |E(\Gamma)| - |V(\Gamma)|,
    \end{equation}
    
    \item[(ii)] (Monotonic decreasing) 
    \begin{equation}
        |E(\phi(\Gamma))| \ge |E(\Gamma)|, \qquad |V(\phi(\Gamma))| \ge |V(\Gamma)|,
    \end{equation}
    
    \item[(iii)] (Multiplicity controlling)  
       \begin{equation}
        |\phi^{-1}(\Gamma)| \le 2^{|E(\Gamma)|},\quad \forall \,\Gamma \in \mathcal{T}_{\ell, s}.
    \end{equation}
   
\end{itemize}
\end{lemma}

\begin{proof}
The construction of $\phi$ follows the vertex-splitting procedure introduced in \cite{feldheim2010universality} (see e.g. Step (iii) in Definition II.1.3, Figure 5 in \cite{geng2024outliers} or Lemma 2.22 in \cite{liu2023edge}). The key idea is to iteratively replace each unmarked vertex of degree $d > 3$ with  a tree of trivalent vertices connected via auxiliary edges, such that the total number of edges increases by $d-3$ and the total number of vertices increases by $d-3$. Crucially, such a replacement preserves the   difference $|E|-|V|$.

Each preimage in $\phi^{-1}(\Gamma)$ can be obtained by shrinking edges in $\Gamma$, so the total number of preimages under $\phi$ is at most $2^{|E|}$, which proves the last part.
\end{proof}
We will need an upper bound  for  the number of typical connected diagrams.

\begin{lemma}[{\cite[Proposition II.2.3]{feldheim2010universality}; \cite[Proposition 3.7]{geng2024outliers}}] \label{diagramno}
There exists a universal constant $C > 0$ such that for any  integers $\ell \ge 0$ and $s \ge 1$,
\begin{equation}
  \left| \mathcal{T}_{\ell,s} \right| \le \frac{(C(\ell+1))^{\ell+1}}{(s-1)!}.
\end{equation}
\end{lemma}

\begin{theorem}\label{prop:U_asy}
For the deformed Gaussian matrix $X$ in Definition \ref{def:inhomo} and  under the assumptions \ref{itm:B1}-\ref{itm:B2} of Definition \ref{mixingdef}, with $n = \sum_{j=1}^s n_j$, $a=\|A_N\|_{\mathrm{op}}$, the following hold.
\begin{enumerate}[label=(\roman*)]
  \item \label{item:U_upper} For all  $n_j\ge 1$, there exists a constant $C > 0$ such that
\begin{equation} \label{eq:Chebyshev_trace_bound}
\mathbb{E}\left[\prod_{j=1}^s \operatorname{Tr} \left(U_{n_j}\left(\frac{X}{2}\right)\right)\right]
\le (1 + a^n)(C n)^s \left\{e^{C \frac{n^{\frac{3}{2}}}{\sqrt{N}}} + (1\lor \frac{t_N}{n})^s e^{C \frac{n^2 t_N}{N}} + \Big((\frac{C n^2 t_N}{N})^{[\frac{n}{t_N}]}\land 1\Big) e^{C \frac{n^3}{N}}\right\}.
\end{equation}
  \item \label{item:U_comparison} If   $a \le 1 + O( N^{-\frac{1}{3}})$ and $t_N \ll n_j = O(N^{\frac{1}{3}})$, then
  \begin{equation}
   \mathbb{E}\Big[\prod_{j=1}^s \tr \big(U_{n_j}\big(\frac{X}{2}\big)\big)\Big]
    = \mathbb{E}\Big[\prod_{j=1}^s \tr \big( U_{n_j}\big(\frac{\widetilde{X}}{2}\big)\big)\Big]
    + o(n^s),
  \end{equation}
  where $\widetilde{X} =  \frac{1}{\sqrt{N}} W_{N}^{(\mathrm{G}\beta \mathrm{E})}+A_N$.
\end{enumerate}
\end{theorem}

\begin{proof}
    We prove \ref{item:U_upper} first. By Lemma   \ref{thm:u=F}, it suffices to obtain upper bounds for  every diagram function.
For this, let $\ell=|E|-|V|$ and introduce  a function  $G_{\Gamma}(n)$ defined by 
\begin{equation} \label{Gn-1}
G_{\Gamma}(n) = 
\begin{cases}
\displaystyle
\frac{(4 \gamma n)^{|E|}}{(|E|-1)!} N^{|V|-|E|} 
+ 10^{|E|} \gamma^{|E_{\mathrm{int}}|} (1+a^n) r^{|V_b|} \frac{n^{|E|}}{(|E|-1)!} N^{|V_{\mathrm{int}}| - |E_{\mathrm{int}}|},
&   \ell \le \frac{n}{t_N}, \\[10pt]
\displaystyle
\frac{n^{|V|-1}}{(|V|-1)!} 
\Big( \frac{(\gamma t_N) \lor n}{N} \Big)^{|E| - |V| + 1} N 
+ (1+a^n) r^{|V_b|} 
\Big( \frac{(\gamma t_N) \lor n}{N} \Big)^{|E_{\mathrm{int}}| - |V_{\mathrm{int}}|} 
\frac{n^{|V|}}{(|V|)!},
&  \frac{n}{t_N} < \ell < n, \\[10pt]

0, & \ell \ge n.
\end{cases}
\end{equation}
    As shown in   Proposition \ref{prop:F_upper_bound},  
    \begin{equation}
        F_{\Gamma}(\{n_j\}_{j=1}^s)\le G_{\Gamma}(n).
    \end{equation}
Since $|V|, |E| \le n$ and the quantity $|E| - |V|$ is preserved under the map $\phi$,  by Lemma~\ref{lem:vertex_split} we have
\begin{equation}
    G_{\Gamma}(n) \le G_{\phi(\Gamma)}(n).
\end{equation}
Hence, by the expansion of the cumulants established in  Lemma   \ref{thm:u=F}, 
\begin{align}
  \kappa_X(n_1,\dots,n_s) 
  &\le \sum_{\Gamma \in \mathscr{D}_{s;\beta}^*} G_\Gamma(n)  
   \le \sum_{\Gamma \in \mathcal{T}_s} |\phi^{-1}(\Gamma)| \cdot G_\Gamma(n) \notag \\
  &\le \sum_{\Gamma \in \mathcal{T}_s} 2^{|E|} \cdot G_\Gamma(n).
\end{align}

Note that
\begin{equation}\label{equ:EVL}
    |E| = 3\ell + s, \quad |V| = 2\ell + s, \quad |E_{\mathrm{int}}| - |V_{\mathrm{int}}| = |E| - |V| = \ell,
\end{equation}
where we use the fact $|E_{\mathrm{b}}|=|V_{\mathrm{b}}|$ (see \cite[{Lemma 2.9}]{geng2024outliers}). Rewriting   $G_\Gamma(n): = G(n,\ell)$  that depends only on $\ell$ and $s$, we use Lemma \ref{diagramno}  to further  arrive at  
\begin{equation} \label{kappab-1}
     \kappa_X(n_1,\dots,n_s) 
    \le \sum_{\ell \le \frac{n}{t_N}} \frac{(C(\ell+1))^{\ell+1}}{(s-1)!} G(n,\ell)
    + \sum_{\frac{n}{t_N} < \ell \le n} \frac{(C(\ell+1))^{\ell+1}}{(s-1)!} G(n,\ell).
\end{equation}
On the other hand, for any fixed $r$ and $\gamma$, it's easy to derive   from \eqref{Gn-1} that  
\begin{equation}
    G(n,\ell) \le 
    \begin{cases}
        (1+a^n)\dfrac{(Cn)^{3\ell+s}}{(3\ell+s-1)!} N^{-\ell}, & \ell \le \dfrac{n}{t_N}, \\[5pt]
        (1+a^n)\dfrac{(Cn)^{2\ell + s -1} \cdot (C(\gamma t_N \lor n))^{\ell+1}}{(2\ell + s)!} N^{-\ell}, & \dfrac{n}{t_N} < \ell \le n, \\[5pt]
        0, & \ell > n,
    \end{cases}
\end{equation}
for some constant $C>0$. Thus, we immediately see from  \eqref{kappab-1} that 
\begin{multline}\label{equ:4.20}
    \kappa_{X}(n_1,\ldots, n_s)
    \le \frac{(1+a^n)}{(s-1)!} \bigg\{
        \sum_{\ell\le \frac{n}{t_N}} \frac{(C(\ell+1))^{\ell+1} (Cn)^{3\ell+s}}{(3\ell+s-1)!} N^{-\ell} \\
        + \sum_{\ell > \frac{n}{t_N}} \frac{(C(\ell+1))^{\ell+1} (Cn)^{2\ell+s-1} (C(\gamma t_N \lor n))^{\ell+1}}{(2\ell + s)!} N^{-\ell}
    \bigg\}:=\frac{(1+a^n)}{(s-1)!}\big(\Sigma_{1}+\Sigma_{2}\big).
\end{multline}

For the first sum $\Sigma_{1}$ in \eqref{equ:4.20},   we  use  the simple fact  $(\ell+1)^{\ell+1} \le C^\ell \ell!$ for $\ell\ge 0$  to obtain 
\begin{align}
  \Sigma_{1} 
    &\le (Cn)^s \sum_{\ell=0}^\infty \frac{1}{(2\ell)!} \Big(\frac{Cn^3}{N}\Big)^\ell 
    \le (Cn)^s \exp\Big(C \frac{n^{\frac{3}{2}}}{\sqrt{N}}\Big).
\end{align}
While for the second sum in \eqref{equ:4.20}, we distinguish two cases according to  the factor $\gamma t_N \lor n$:
\begin{enumerate}
    \item[(1)] When   $n \le \gamma t_N$,   
    \begin{align}
        \Sigma_2
        &\le (Cn)^s \sum_{\ell=0}^\infty \frac{1}{\ell!}\frac{t_N}{n} \Big( \frac{C n^2 t_N}{N} \Big)^\ell 
        = (Cn)^{s-1}t_N\cdot \exp\Big(C \frac{n^2 t_N}{N} \Big).
    \end{align}
    
    \item[(2)] When $n \ge \gamma t_N$, let $m=[\frac{n}{t_N}]\ge 1$, we have
    \begin{align} \label{kappab-3}
         \Sigma_2
        &\le (Cn)^s \sum_{\ell=\frac{n}{t_N}}^\infty \frac{1}{\ell!} \Big( \frac{C n^3}{N} \Big)^\ell \le (Cn)^s \frac{1}{m!}\Big( \frac{C n^3}{N} \Big)^m\sum_{\ell=0}^\infty \frac{1}{\ell!} \Big( \frac{C n^3}{N} \Big)^\ell
        \le (Cn)^s \Big( \frac{C' n^2 t_N}{N} \Big)^m \exp\Big(C \frac{n^3}{N}\Big).
    \end{align}
    It is also crucial to note that  
    \begin{align} \label{equ:4.22}
        & (Cn)^s \sum_{\ell=\frac{n}{t_N}}^\infty \frac{1}{\ell!} \Big( \frac{C n^3}{N} \Big)^\ell \le (Cn)^s  \exp\Big(C \frac{n^3}{N}\Big).
    \end{align}
\end{enumerate}
Since $s$ is a fixed positive integer,  put \eqref{equ:4.20}-\eqref{equ:4.22} together and we obtain    \begin{equation} \label{kappab-4}
     \kappa_{X}(n_1,\ldots, n_s)
    \le (1 + a^n)(C n)^s \Big\{e^{C \frac{n^{\frac{3}{2}}}{\sqrt{N}}} + \frac{t_N}{n}e^{C \frac{n^2 t_N}{N}} + \Big((\frac{C n^2 t_N}{N})^{[\frac{n}{t_N}]}\land 1\Big) e^{C \frac{n^3}{N}}\Big\}.
  \end{equation}

Using \eqref{equ:T} in Definition \ref{Cumu} and rewriting  the mixed moment via the cumulants $\kappa$ and using
\begin{equation}
    (a+b+c)^{t}\le 3^t (a^t+b^t+c^t)
\end{equation}
  conclude the proof of  part \ref{item:U_upper}.

Next, we turn to part  \ref{item:U_comparison}.  Noticing the assumptions  $a \le 1 + O(N^{-1/3})$ and  $n_j = O(N^{\frac{1}{3}})$, the upper bound in  \eqref{kappab-4}  shows that the sum
\begin{equation}
    \sum_{\Gamma \in \mathscr{D}_{s;\beta}^*} \frac{1}{n^s} F_{\Gamma}(\{n_j\}_{j=1}^s)
\end{equation}
is absolutely summable. Therefore, for any sufficiently large constant $K > 0$, the contribution from those diagram functions $F_{\Gamma}$  such that   $|E(\Gamma)|, |V(\Gamma)| > K$ can be negligible.

Since the number of diagrams with $|E|, |V| \leq K$ is finite, it suffices to  analyze the  diagram-wise limit. By Theorem \ref{thm:F=tildeF}, for all  $n_j \gg t_N$, we have
\begin{equation}
    \frac{1}{n^s} F_{\Gamma}(\{n_j\}_{j=1}^s)
    = \frac{1}{n^s} \widetilde{F}_{\Gamma}(\{n_j\}_{j=1}^s) + o\bigl(n^{|E| - s} N^{|V| - |E|}\bigr)
    = \frac{1}{n^s} \widetilde{F}_{\Gamma}(\{n_j\}_{j=1}^s) + o(1).
\end{equation}

Hence, under the condition of part \ref{item:U_comparison}, we conclude
\begin{equation}
    \sum_{\Gamma \in \mathscr{D}_{s;\beta}^*} \frac{1}{n^s} F_{\Gamma}(\{n_j\}_{j=1}^s)
    = \sum_{\Gamma \in \mathscr{D}_{s;\beta}^*} \frac{1}{n^s} \widetilde{F}_{\Gamma}(\{n_j\}_{j=1}^s) + o(1),
\end{equation}
which completes the proof of part \ref{item:U_comparison}.

\end{proof}

\begin{theorem}\label{prop:super_4}
Given fixed integers $s \geq 1$ and $k_i \geq 1$ for $1 \leq i \leq s$, if  
\begin{equation}
    t_N \ll n_{i,j} \leq \tau N^{\frac{1}{3}}, \quad \forall  1 \leq i \leq s, \quad 1 \leq j \leq k_i,
\end{equation} for any fixed constant $\tau > 0$, and $ 
    \|A_N\|_{\mathrm{op}} \leq 1 + O(N^{-\frac{1}{3}}),
$ 
then  
\begin{equation}
    \mathbb{E}\bigg[\prod_{i=1}^s \tr \Big(\prod_{j=1}^{k_i} \frac{1}{n_{i,j}+1} U_{n_{i,j}}\big(\frac{X}{2}\big)\Big)\bigg]
    =  \mathbb{E}\bigg[\prod_{i=1}^s \tr \Big(\prod_{j=1}^{k_i} \frac{1}{n_{i,j}+1} U_{n_{i,j}}\big(\frac{\widetilde{X}}{2}\big)\Big)\bigg] + o(1).
\end{equation}
\end{theorem}
\begin{proof}
We first prove the case $s = 1$ with $t_1 = t$ and $n_{1,j} = m_j$ for $j = 1,\ldots, t$. The case $s > 1$ follows similarly.

Using iteratively   the product formula for Chebyshev polynomials of the second kind 
\begin{equation}
    U_k(x) U_l(x) = \sum_{m=0}^{\min(k,l)} U_{|l-k| + 2m}(x),
\end{equation} which can be directly verified by   summing terms on the right-hand side via their definition in \eqref{TUdef},
    we derive the expansion
\begin{equation} \label{prodexpansion}
    \prod_{j=1}^t 
    U_{m_j}(x)
    = \prod_{j=1}^t 
    \sum_{k \geq 0} c(\{m_j\}; k) 
    U_k(x).
\end{equation} Here $c(\{m_j\}; k)$ 
  denotes the number of distinct ways to obtain  $U_k$
 in the expansion, defined recursively as follows:
\begin{equation} 
    c(\{m_j\}; k)= \sum_{i_1=0}^{ m_1 \wedge m_2}
       \sum_{i_2=0}^{ \Delta_1 \wedge m_3} \cdots
     \sum_{i_{t-1}=0}^{ \Delta_{t-2} \wedge m_t}
     \mathbf{1}(\Delta_{t-1} = k)
\end{equation}
with
\[
\begin{aligned}
    \Delta_1  = |m_1 - m_2| + 2 i_1,  \ 
    \Delta_2  = |\Delta_1 - m_3| + 2 i_2,  \  \ldots,  \  
    \Delta_{t-1}  = |\Delta_{t-2} - m_t| + 2 i_{t-1}.
\end{aligned}
\]

Evaluating at $x = 1$ on both sides of  \eqref{prodexpansion} and using the fact $U_k(1)=k+1$, we get
\begin{equation}\label{equ:4.31}
    \sum_{k \geq 0} c(\{m_j\}; k)(k+1) = \prod_{j=1}^t (m_j + 1).
\end{equation}
Combine  the  upper  bound  for coefficients 
\begin{align}
    c(\{m_j\}; k)
    &\leq \sum_{i_1=0}^{m_2} \sum_{i_2=0}^{m_3} \cdots \sum_{i_{t-2}=0}^{m_{t-1}} 1 
    = \prod_{j=2}^{t-1} (m_j + 1),
\end{align}
 and choose a large number $M$ satisfying   $t_N \ll M \ll m_0:=\min_{1\leq j\leq t} \{m_j + 1\}$, we  obtain 
  \begin{align}
  \sum_{0 \leq k \leq M} c(\{m_j\}; k)(k+1) 
  \leq \prod_{j=2}^{t-1} (m_j + 1) \cdot \sum_{0 \leq k \leq M} (k+1)  
 \leq \Big( \frac{M}{m_0} \Big)^2 \prod_{j=1}^t (m_j + 1).
  \label{equ:upper_M}
  \end{align}

By the expansion \ref{prodexpansion}, we     divide   the expectation into two sums 
\begin{align}
    \mathbb{E}\Big[\tr \prod_{j=1}^t   U_{m_j}\big(\frac{X}{2}\big)\Big]
    &=   \Big(\sum_{0\leq k \leq M}   + \sum_{k > M} \Big)  
    \,c(\{m_j\}; k)  \,  \mathbb{E}\big[\tr   U_k\big(\frac{X}{2}\big)\big].
\end{align}
For the sum over $k \leq M$,   combining  Theorem ~\ref{prop:U_asy}~\ref{item:U_upper} and~\eqref{equ:upper_M},  we obtain
\begin{equation}
    \Big| \sum_{0\leq k \leq M}    
    \,c(\{m_j\}; k)  \,  \mathbb{E}\big[\tr   U_k\big(\frac{X}{2}\big)\big]  \Big| 
    \leq C \Big( \frac{M}{m_0} \Big)^2 \prod_{j=1}^t (m_j + 1)  = o(\prod_{j=1}^t (m_j + 1)).
\end{equation}
While for the sum over $k > M \gg t_N$,  applying Theorem ~\ref{prop:U_asy}~\ref{item:U_comparison} yields 
\begin{align}
    \sum_{k > M} \,c(\{m_j\}; k)  \,  \mathbb{E}\big[\tr   U_k\big(\frac{X}{2}\big)\big] 
    &= \sum_{k > M}   c(\{m_j\}; k) 
    \Big( \mathbb{E}\big[\tr   U_k\big(\frac{\widetilde{X}}{2}\big)\big] + o(\prod_{j=1}^t (m_j + 1)) \Big).
\end{align}

Put the above two estimates together, combining both estimates yields
\begin{equation}
   \mathbb{E}\Big[\tr \prod_{j=1}^t   U_{m_j}\big(\frac{X}{2}\big)\Big] 
    = \mathbb{E}\Big[\tr \prod_{j=1}^t   U_{m_j}\big(\frac{\widetilde{X}}{2}\big)\Big]  + o\big(\prod_{j=1}^t (m_j + 1)\big).
\end{equation}
This completes the proof in the case $s=1$.

For the case of $s > 1$, we can proceed in a similar way  to deal with every   trace  term separately and thus   complete the proof.
\end{proof}

\subsection{Proof of Theorem \ref{thm:main_thm}: Gaussian case}
We state two standard properties of Chebyshev polynomials, while their proofs can be found in the  literature.

\begin{lemma}[{\cite[Lemma 11.2]{erdHos2011quantum}}]\label{lem:B.8}
 \begin{equation}\label{equ:U_upper}
    |U_n(1+x)|\le 2n\land \frac{1}{\sqrt{-2x}},\quad -1\le x\le 0,
\end{equation} 
and  for some    constant $C>0$, 
\begin{equation}\label{equ:U_lower}
\frac{1}{n+1}U_n(1+x)\ge e^{C n\sqrt{x}}, \quad x\in [0,0.1).
\end{equation}
\end{lemma}

By the above inequality  we  can afford  a right-tail  estimate  for the spectral norm. 
\begin{corollary}\label{coro:tail_bound} Under the assumption of Theorem \ref{thm:main_thm}, if $x>0$ and $x^2 t_N N^{-\frac{1}{3}}< c_0$ for some small absolute constant, then there exist constants $c, C>0$, such that
\begin{equation}
    \mathbb{P}\Big(\|\frac{1}{2}X\|_{\mathrm{op}}>1+xN^{-\frac{2}{3}}\Big)\le C(1+ e^{C \tau x})e^{-c x^{\frac{3}{2}}}.
\end{equation}
\end{corollary}
\begin{proof}
By \eqref{equ:4.31} and the upper bound in Theorem  \ref{prop:U_asy} \ref{item:U_upper}, if $n\gg t_N$, we have
\begin{align}
    \mathbb{E}\big[\tr (\frac{1}{n+1}U_{n}(\frac{H}{2}))^4\big]\le C(1+a^n)\left\{
e^{C \frac{n^{3/2}}{\sqrt{N}}}
+ e^{C \frac{n^2 t_N}{N}}
+ \left(\frac{C n^2 t_N}{N}\right)^{\left\lfloor\frac{n}{t_N}\right\rfloor} e^{C \frac{n^3}{N}}
\right\}.
\end{align}
Thus using \eqref{equ:U_lower} we have for some constants $C_1, C_2$,  
    \begin{align}
        \mathbb{E}\Big[e^{C_1n\sqrt{(\|\frac{1}{2}X\|_{\mathrm{op}}-1})_{+}}\Big]&\le \mathbb{E}\Big[\tr (\frac{1}{n+1}U_{n}(\frac{H}{2}))^4\Big]\notag\\
        &\le C(1+a^n)\left\{
e^{C \frac{n^{3/2}}{\sqrt{N}}}
+ e^{C \frac{n^2 t_N}{N}}
+ \left(\frac{C n^2 t_N}{N}\right)^{\left\lfloor\frac{n}{t_N}\right\rfloor} e^{C \frac{n^3}{N}}
\right\}.
    \end{align}

    So  take $t_N\ll n=tN^{\frac{1}{3}}$, let $\epsilon=t_N N^{-1/3}$ and we have
    \begin{equation}
        \mathbb{E}\Big[e^{C_1tN^{\frac{1}{3}}\sqrt{\|\frac{1}{2}X\|_{\mathrm{op}}-1}}\Big]\le C_2(1+e^{t\tau})(e^{C_2t^{\frac{3}{2}}}+e^{C_3 \epsilon t^2}+(C_4 t^2\epsilon)^{\frac{t}{\epsilon}}e^{C_5 t^3} ).
    \end{equation}
    By Markov inequality and then by Theorem \ref{prop:U_asy}\ref{item:U_upper}, we have
    \begin{align}
        \mathbb{P}\Big(\|\frac{X}{2}\|_{\mathrm{op}}>1+xN^{-\frac{2}{3}}\Big)&\le \mathbb{E}\Big[e^{C_1t N^{\frac{1}{3}}\sqrt{(\|\frac{X}{2}\|_{\mathrm{op}}-1)_+}}\Big]e^{-C_1t\sqrt{x}} \nonumber\\
        &\le C_2(1+e^{t\tau})\big(e^{C_2t^{\frac{3}{2}}}+e^{C_3 \epsilon t^2}+(C_4 t^2 \epsilon)^{\frac{t}{\epsilon}}e^{C_5 t^3} \big)e^{-C_1t\sqrt{x}}.
    \end{align}
    Take $t=\frac{1}{2}\sqrt{(C_1C_2^{-1})}x$, we obtain
    \begin{equation}
        \mathbb{P}\Big(\|\frac{X}{2}\|_{\mathrm{op}}>1+xN^{-\frac{2}{3}}\Big)\le C_5(1+ e^{C_6 \tau x})(e^{-C_7 x^{\frac{3}{2}}}+e^{C_8 \epsilon x^2-C_9 x^{\frac{3}{2}}} + (C_{10} x^2 \epsilon)^{C_{11} x\epsilon^{-1}}e^{C_{12} x^3}).
    \end{equation}
When  $\epsilon \le \frac{C_9}{2C_8\sqrt{x}}$ and $C_{10}x^2\epsilon \le e^{-1}$ and $2C_{12}x^3 \le C_{11}x\epsilon^{-1}$, we have
\begin{equation}
    C_8 \epsilon x^2-C_9 x^{\frac{3}{2}}\le -\frac{C_{9}}{2}x^{\frac{3}{2}},~~~~(C_{10}x^2\epsilon)^{C_{11} x\epsilon^{-1}}e^{C_{12} x^3}\le e^{-C_{12} x^3}\le C_{13}e^{-C_{14} x^{\frac{3}{2}}}.
\end{equation}
   
    This thus completes  the proof.
\end{proof}
\begin{lemma}\label{lem:B.9_Chebyshev}
For any  fixed  number $t>0$, let $n = \lfloor tM\rfloor$, then for  $y\in\mathbb{R}$,
\begin{equation}
\lim_{M\to\infty}
\frac{1}{n+1}\,
U_n\!\Bigl(1 + \frac{y}{2M^2}\Bigr)
\;=\;
\frac{\sin\bigl(t\sqrt{-y}\bigr)}{t\sqrt{-y}},
\end{equation}
\end{lemma}

This lemma addresses the scaling limit at the hard edge, which  occurs  in   orthogonal polynomials and random matrices. While its proof is well-established in the literature, we provide a detailed derivation here for completeness.

\begin{proof}
Recalling  the representation for  the Chebyshev polynomial of the second kind  
\begin{equation}
U_n(x)
=
\begin{cases}
\dfrac{\sin\bigl((n+1)\theta\bigr)}{\sin\theta}, &x=\cos\theta,\;\theta\in(0,\pi),\\[1em]
\dfrac{\sinh\bigl((n+1)\varphi\bigr)}{\sinh\varphi},
&x=\cosh\varphi,\;\varphi>0,
\end{cases}
\end{equation}
  in either regime we may set
\begin{equation}
 x = 1 + \frac{y}{2M^2}, 
 \quad
 \Theta = 
 \begin{cases}
 \arccos(x), &y<0,\\
 i\,\mathrm{arccosh}(x), &y>0.
 \end{cases}
\end{equation}

So that in unified form
\begin{equation}
U_n(x) = \frac{\sin\bigl((n+1)\Theta\bigr)}{\sin\Theta},
\qquad
\Theta = \frac{\sqrt{-y}}{M}\,\Bigl(1 -\frac{y}{24M^2}+O(y^2M^{-4})\Bigr),
\end{equation}
Therefore, wo obtain 
\begin{equation}
\begin{aligned}
\frac{1}{n+1}U_n(x)&=\frac{\sin\bigl((n+1)\Theta\bigr)}{(n+1)\Theta}\cdot \frac{\Theta}{\sin \Theta}\\
&= (1+O(yM^{-2}))\frac{\sin\bigl((n+1)\Theta\bigr)}{(n+1)\Theta}\\ 
&=\frac{1+O(yM^{-2})}{t\sqrt{-y}}\sin(t\sqrt{-y})+O(|y|^{\frac{1}{2}}M^{-2}).
\end{aligned}
\end{equation}
As $M\rightarrow \infty$, we finish the proof.
\end{proof}

Now we are ready to prove the Gaussian case of  Theorem~\ref{thm:main_thm}.

\begin{proof}[Proof of Theorem~\ref{thm:main_thm}]  We first prove  the Gaussian case and leave the general case in the next section. 

   In Theorem~\ref{prop:super_4},   take   $k_i = t$ for $1 \le i \le s$,  and also  take $n_{i,j} = m_i=[\tau_i N^{1/3}]$ for $1 \le i \le s$, $1 \le j \le t$, we  then obtain 
    \begin{equation}\label{equ:4.64}
        \mathbb{E}\left[\prod_{i=1}^{s} \tr \Big(\frac{1}{m_i+1} U_{m_i}\big(\frac{X}{2}\big)\Big)^t\right]
        = \mathbb{E}\left[\prod_{i=1}^{s} \tr \Big(\frac{1}{m_i+1} U_{m_i}\big(\frac{\widetilde{X}}{2}\big)\Big)^t\right]+ o(1).
    \end{equation}
    Similarly, for $t = 4,8,10$,  we have
    \begin{multline}\label{equ:U7=U7}
\mathbb{E}\left[\prod_{i=1}^{s} \tr \bigg(\Big(\frac{1}{m_i+1} U_{m_i}\big(\frac{X}{2}\big)\Big)^{t-1}  \Big(\frac{1}{m_i+2} U_{m_i+1}\big(\frac{X}{2}\big)\Big)\bigg)
\right]\\
=\mathbb{E}\left[\prod_{i=1}^{s} \tr \bigg(\Big(\frac{1}{m_i+1} U_{m_i}\big(\frac{\widetilde{X}}{2}\big)\Big)^{t-1}  \Big(\frac{1}{m_i+2} U_{m_i+1}\big(\frac{\widetilde{X}}{2}\big)\Big)\bigg)
\right]+o(1). 
    \end{multline}

    Now consider the case $s=1$.  Let $n = [\tau N^{1/3}]$ and let   $\lambda_1 \le \cdots \le \lambda_N$ be the  eigenvalues of $X$. Given any  small constant  $\delta > 0$,  split the trace into three parts and we have   
   \begin{equation}
       \mathbb{E}\left[\tr \Big(\frac{1}{n+1} U_n(X)\Big)^8 \right] = \mathbb{E}\left[ \sum_{i=1}^N \Big(\frac{1}{n+1} U_n\big(\frac{\lambda_i}{2}\big)\Big)^8 \right] = \Sigma_1 + \Sigma_2 + \Sigma_3,
    \end{equation}
    where
    \begin{align}
        \Sigma_1 =   \sum_{i: |\lambda_i| \le 2 - N^{-\frac{5}{12} + \delta}} &\Big(\frac{1}{n+1} U_n\big(\frac{\lambda_i}{2}\big)\Big)^8,  \quad \Sigma_2 =   \sum_{i:|\lambda_i| \ge 2 + N^{-\frac{2}{3} + \delta}} \Big(\frac{1}{n+1} U_n\big(\frac{\lambda_i}{2}\big)\Big)^8, \\
        \Sigma_3 &=   \sum_{i: 2 - N^{-\frac{5}{12} + \delta} < |\lambda_i| < 2 + N^{-\frac{2}{3} + \delta}} \Big(\frac{1}{n+1} U_n\big(\frac{\lambda_i}{2}\big)\Big)^8.
    \end{align}
   For the first   sum,   by Lemma~\ref{lem:B.8}, we can obtain  
    \begin{equation}
        \Sigma_1 \le N \cdot N^{-8/3} \cdot N^{4(5/12 - \delta)} = O(N^{-4\delta}).
    \end{equation}
   Here we emphasize that we can  improve $N^{-5/12 + \delta}$ to $N^{-2/3 + \delta}$ by replacing  the exponent $8$   by a   sufficiently large constant $2k$.

    The second sum   $\Sigma_2$   is negligible    due to the exponential tail bound   in    Corollary~\ref{coro:tail_bound} and the uniform estimate in Theorem~\ref{prop:U_asy}~\ref{item:U_upper}. Consequently,   both  the probability and the expected mass of eigenvalues lying   in $[2 + N^{-2/3 + \delta}, \infty)$ decay exponentially.

    Lastly, for the third sum $\Sigma_3$,   we need to introduce a change of variables   $\lambda_i = 2 + N^{-2/3} y_i$ and apply Lemma~\ref{lem:B.9_Chebyshev}.  Combine these with     the matching moment estimates in \eqref{equ:4.64} and \eqref{equ:U7=U7}, we conclude from   the continuity theorem~\cite[Theorem B.10]{liu2023edge}  that in the supercritical regime, all $k$-point correlation functions of $X$ converge weakly to those of $\widetilde{X}$.

This indeed completes the proof by 
following the almost same strategy   pioneered by Soshnikov~\cite{soshnikov1999universality} and Sodin~\cite{sodin2010spectral}.
    \end{proof}

A related and interesting corollary is the evaluation of the sinc transform for correlation functions of (deformed) Airy point processes, including both real and complex cases. In the non-deformed setting, this result was first obtained by Sodin \cite{sodin2010spectral}.

\begin{definition}[Limit diagram function]
Given a typical diagram $\Gamma$,  for      $t_1,\dots,t_s>0$ introduce   a family of  linear constraint   inequalities    for edge–weighted variables $\{\alpha_e\geq 0\}$
\begin{equation}
\mathfrak{C}(\{\alpha_e\})
:\quad
\sum_{e\in E(\Gamma)} c_i(e)\,\alpha_e \;\le\; t_i, 
\quad i=1,\dots,s,
\end{equation}
with \(c_i(e)\in \{0,1,2\}\) counting  the times  of   \(e\) appearing in $\partial D_i$, a diagram 
 function   is defined as 
 \begin{equation}\label{eq:cor-F-integral}
\mathcal{F}_\Gamma(t_1,\dots,t_s)
\;=\;
\frac{C_\Gamma}{\,t_1\cdots t_s\,}\,
\int \cdots \int_{\mathfrak{C}(\{\alpha_e\})}
\prod_{b_j}\Bigl(\sum_{i=1}^q e^{\tau_i\,\ell_{b_j}}\Bigr)
\,\prod_{e\in E(\Gamma)}d\alpha_e,
\end{equation}
where for each open cycle \(b_j\subset E_b(\Gamma)\)
\begin{equation}
    \ell_{b_j}:\;=\;\sum_{e\in b_j}\alpha_e,
\end{equation} 
and \(C_\Gamma>0\) is a combinatorial constant depending only on \(\Gamma\) defined by
\begin{equation}
    C_\Gamma=\lim_{n_j\rightarrow\infty}\frac{\#\{(w_e)_{e\in E(\Gamma)}|
        t_j, w_e\in \mathbb{N},~t_j \ge 0,\; w_e \ge 1,~
        2t_j + \sum\limits_{e \in \partial D_j} w_e = n_j,~ j=1,\dots,s\}}{\mathrm{Vol}\{(\alpha_e)_{e\in E(\Gamma)}|\alpha_e\in \mathbb{R}^+, \sum_{e\in E(\Gamma)} c_i(e)\,\alpha_e \;\le\; n_j, 
\quad j=1,\dots,s\}}.
\end{equation}
\end{definition}
\begin{corollary} For any given real numbers  $t_1,\dots,t_s>0$,    $\tau_1,\dots,\tau_q$ and  $a_{q+1}, \ldots,a_{r}\in (-a,a)$, suppose that 
\begin{equation}
    n_i\,N^{-1/3}\;\longrightarrow\;t_i, \quad i=1,\ldots,s,
\end{equation}
and  \begin{equation}
N^{1/3}(a_j-1)\;\longrightarrow\;\tau_j,
\quad j=1,\dots,q, \end{equation}
as \(N\to\infty\).   Then the  scaled cumulants 
\begin{equation}\label{eq:cor-K-sum}
\lim_{N\to\infty}
\frac{\kappa_{\widetilde X}(n_1,\ldots, n_s)}{(n_1+1) \cdots (n_s+1)}=\widetilde{\mathcal{K}}(t_1,\dots,t_s):
=\sum_{\Gamma\in\mathcal{T}_s}
\mathcal{F}_\Gamma(t_1,\dots,t_s).
\end{equation}
\end{corollary}
\begin{proof}
Consider GOE/GUE case only and  only connected “typical’’ diagrams \(\Gamma\in\mathcal{T}_s\) contribute. By \eqref{equ:F_formula}, we have
\begin{equation}\label{pf1}
\begin{aligned}
        F_{\Gamma}(\{n_j\}) 
        &=N^{|V_{\mathrm{int}}|-|E_{\mathrm{int}}|}\sum_{\eta: V_{\mathrm{b}}(\Gamma) \to [N]}
    \sum_{\substack{
        t_j \ge 0,\; w_e \ge 1 \\
        2t_j + \sum\limits_{e \in \partial D_j} w_e = n_j
    }}\prod_{(z, w) \in E_b}
        (A^{w_e})_{\eta(z)\eta(w)},    
\end{aligned}
\end{equation}
Moreover, summing over boundary labels \(\eta:V_{\mathrm{b}}\to[N]\) factorizes into cycles \(b_j\)
\begin{equation}\label{pf3}
\sum_{\eta:V_{\mathrm{b}}\to[N]}
\prod_{(z,w)\in E_{\mathrm{b}}}(A^{w_e})_{\eta(z)\eta(w)}
= \prod_{j=1}^s \Bigl(\sum_{i=1}^r a_i^{m_j}\Bigr),
\end{equation}
where \(m_j=\sum_{e\in b_j}w_e\).  Next, let
\begin{equation}\label{pf4}
H_{\Gamma}(n_1,\dots,n_s)
=\sum_{\substack{t_j\ge0,\,w_e\ge1\\2t_j+\sum_{e\in\partial D_j}w_e=n_j}}1,
\end{equation}
counting these integer solutions.  As \(n_j=N^{1/3}t_j\to t_j\), a lattice-point to volume argument yields
\begin{equation}\label{pf5}
N^{-\lvert E(\Gamma)\rvert/3} H_{\Gamma}(n_1,\dots,n_s)
\;\longrightarrow\;
C_{\Gamma}\int_{\mathfrak{C}(\{\alpha_e\})}d\alpha_e.
\end{equation}

Finally, since \(a_i=1+\tau_iN^{-1/3}\) and \(m_j=N^{1/3}\ell_{b_j}\), we have
\begin{equation}\label{pf6}
\prod_{j=1}^s\Bigl(\sum_{i=1}^r a_i^{m_j}\Bigr)
\;\longrightarrow\;
\prod_{j=1}^s\Bigl(\sum_{i=1}^q e^{\tau_i\,\ell_{b_j}}\Bigr).
\end{equation}
Combining \eqref{pf1}–\eqref{pf6}, dividing by \( (n_1+1)\cdots (n_s+1)\sim N^{s/3}t_1\cdots t_s\), and letting \(N\to\infty\) gives \eqref{eq:cor-F-integral}.
\end{proof}

\section{Universality for sub-Gaussian inhomogenous  matrices}\label{sec:section5}

In this section, we extend the edge universality results for Gaussian IRM established previously to general $\theta$-sub-Gaussian cases, thereby completing the proof of Theorem \ref{thm:main_thm}.   For clarity of exposition, we restrict our analysis to the real symmetric case without deformation. We note that all arguments presented herein extend naturally to both deformed  and Hermitian settings; see  \eqref{prop:path_expansion_H_plus_A} for a deformed version of Chebyshev path expansion.

\begin{definition}
For a  GOE matrix   $\widetilde{W}=(\widetilde{W}_{ij})$ and for $\theta\geq 1$, we  
 define a \(\theta\)-GOE matrix \(\widehat{W} = (\widehat{W}_{ij})\) by
\begin{equation}
    \widehat{W}_{ij} = B_{ij} \cdot \widetilde{W}_{ij},
\end{equation}
\(\{B_{ij} \}_{i\leq j}\)
are   independent Bernoulli random variables,    distributed    as 
$\sqrt{\theta} \cdot \mathrm{Bern}(\theta^{-1})$. 
Also, 
let \( W_N = (W_{ij}) \) be a symmetric \emph{sub-Gaussian} matrix with  parameter \(\theta \geq 1\) as in 
Definition \ref{def:inhomo}, set 
\begin{equation}
    H=\Sigma_N\circ W_N,~~~\widehat{H}=\Sigma_N\circ \widehat{W}_N.
\end{equation}

\end{definition}

We need to introduce some necessary notations  where the first two ones have been defined in  \cite[Section 6.1]{EK11Quantum}.
\begin{definition}  \label{def:nonbacktracking}
\begin{itemize}
    \item[(i)] 

For $n\ge 2$,  we define  $V_n$ as the $n$-th non-backtracking power of $H$,
\begin{equation}\label{equ:def_V}
    \begin{aligned}(V_n)_{xy}:=\sum_{x_0,x_1,\ldots,x_n}\delta_{x,x_0}\delta_{x_n,y}\left[\prod_{i=0}^{n-2}\mathbf{1}(x_i\ne x_{i+2})\right] H_{x_0x_1}H_{x_1x_2}\cdots H_{x_{n-1}x_n},
    \end{aligned}
\end{equation}
while   $V_0:=\mathbb{I}, V_1:=H$ and $V_n:=0$ for $n<0$.

\item[(ii)] We introduce  two   matrices $\Phi_2$ and $\Phi_3$ through
\begin{equation}
    (\Phi_2)_{xy}:=\delta_{xy}\sum_{z}(|H_{xz}|^2-\sigma_{xz}^2), \quad (\Phi_3)_{xy}:=-|H_{xy}|^2H_{xy},
\end{equation}
and  also introduce the shorthand $\underline{\Phi_3V_n}$ defined by
\begin{equation}
(\underline{\Phi_3V_n})_{xy}:=\sum_{x_0,x_1,\ldots x_{n-2}}\delta_{x,x_0}\delta_{x_{n-2},y}\left[\prod_{i=0}^{n-4}\mathbf{1}(x_i\ne x_{i+2})\right] (\Phi_{3})_{x_0x_1}H_{x_1x_2}\cdots H_{x_{n-3}x_{n-2}},
\end{equation}
while by  convention  $\underline{\Phi_3V_0}=\underline{\Phi_3V_1}=\underline{\Phi_3V_2}=0$,  $\underline{\Phi_3V_3}=\Phi_3$ and $\underline{\Phi_2 V_n}:=\Phi_2V_{n-2}$. 

\item[(iii)]  Similarly, we introduce 
\begin{equation}\label{equ:Phi+}
    (\widehat{\Phi}_3)_{xy}:=|\widehat{H}_{xy}|^2\widehat{H}_{xy}, ~~(\widehat{\Phi}_2)_{xy}:=\delta_{xy}\sum_{z}B_{xz}^2\sigma_{xz}^2(|\widetilde{W}_{xz}|^2-1),
\end{equation}
and 
\begin{equation}
(\underline{\widehat{\Phi_3V_n}})_{xy}:=\sum_{x_0,x_1,\ldots x_{n-2}}\delta_{x,x_0}\delta_{x_{n-2},y}\left[\prod_{i=0}^{n-4}\mathbf{1}(x_i\ne x_{i+2})\right] (\widehat{\Phi}_{3})_{x_0x_1}\widehat{H}_{x_1x_2}\cdots \widehat{H}_{x_{n-3}x_{n-2}},
\end{equation}
while by   convention  $\underline{\widehat{\Phi_3V_3}}=\widehat{\Phi}_3$ and $\underline{\widehat{\Phi_2 V_n}}:=\widehat{\Phi}_2\widehat{V}_{n-2}$. Here $\widehat{V}$ is the resulting  matrix after replacing all $H$ by   $\widehat{H}$   in \eqref{equ:def_V}. 
\end{itemize} 
\end{definition}

We now present the Chebyshev path expansion
  due to Erd{\H o}s and Knowles, with  a deformed version   established  in  Proposition \ref{prop:path_expansion_H_plus_A}.   
\begin{proposition}[{\cite[Proposition 6.2]{EK11Quantum}}]
    \begin{equation}
        U_{n}\big(\frac{H}{2}\big)=\sum_{k\ge 0}\sum_{{{a}}\in \{2,3\}^k}\sum_{l_0+\cdots +l_k=n}V_{l_0}\underline{\Phi_{a_1}V_{l_1}}\cdots\underline{\Phi_{a_k}V_{l_k}},
    \end{equation}
    where the sum ranges over $l_i\ge 0$ for $i=0,\ldots, k$, and 
    $a=(a_1,\ldots ,a_k)$.
\end{proposition}

\begin{lemma}\label{lem:5.4}
    If all   $l_i^{(j)}\ge 0$, then  
    \begin{equation}\label{equ:5.9}
        \left|\mathbb{E}\Big[\prod_{j=1}^{s} \tr \Big(V_{l^{(j)}_0}\underline{\Phi_{a_1^{(j)}}V_{l_1^{(j)}}}\cdots\underline{\Phi_{a_{k_j}^{(j)}}V_{l_{k_j}^{(j)}}}\Big)\Big]\right|\le 4^{k_1+\cdots+k_s} \mathbb{E}\Big[\prod_{j=1}^{s} \tr \Big( \widehat{V}_{l^{(j)}_0}\underline{\widehat{\Phi_{2}V_{l_1^{(j)}}}}\cdots\underline{\widehat{\Phi_{2}V_{l_{k_j}^{(j)}}}}\Big)\Big].
    \end{equation}
\end{lemma}
\begin{proof} 
Expanding both sides through matrix multiplication and trace, we reduce the problem to obtaining term-by-term upper bounds.  
Consider any non-vanishing path under expectation 
in the expansion. By the independence of the entries of 
$H$, we may focus on individual matrix elements indexed by $\{x,y\}$. 
For the real symmetric case, note that  $H_{xy}=H_{yx}$, the proof reduces to verifying the inequality:
 \begin{equation}\label{equ:5.15}
    \left|\mathbb{E}\left[(-|H_{xy}|^2 H_{xy})^{t_1} (|H_{xy}|^2 - \sigma_{xy}^2)^{t_2} H_{xy}^{t_3} \right]\right|
    \leq 4^{t_1 + t_2} \, \mathbb{E}\left[ \left( (B_{xy})^2 \sigma_{xy}^2 (|\widetilde{W}_{xy}|^2 - 1) \right)^{t_1 + t_2} \widehat{H}_{xy}^{t_1 + t_3} \right],
\end{equation}
where  the  constraints  $t_1 + t_3 \equiv 0 \pmod{2}$ and $t_1+t_2+\frac{1}{2}(t_1+t_3)>1$ are assumed  to ensure
non-vanishing moments.

Since $H_{xy} = \sigma_{xy} W_{xy}$ and $\widehat{H}_{xy} = \sigma_{xy} \widehat{W}_{xy}$, dividing both sides by the common factor $\sigma_{xy}^{3t_1 + 2t_2 + t_3}$,   the left-hand side of \eqref{equ:5.15}  gives us 
\begin{align}
    &\left| \mathbb{E} \left[ (-|W_{xy}|^2 W_{xy})^{t_1} (|W_{xy}|^2 - 1)^{t_2} W_{xy}^{t_3} \right] \right| \notag \\
    &\quad = \left| \mathbb{E} \left[ (-|W_{xy}|^2 W_{xy})^{t_1} \sum_{i=0}^{t_2} \binom{t_2}{i} (-1)^{t_2 - i} |W_{xy}|^{2i} W_{xy}^{t_3} \right] \right| \notag \\
    &\quad \leq \sum_{i=0}^{t_2} \binom{t_2}{i} \left| \mathbb{E} \left[ (|W_{xy}|^2 W_{xy})^{t_1} |W_{xy}|^{2i} W_{xy}^{t_3} \right] \right| \label{eq:binomial_expansion} \\
    &\quad \leq 2^{t_2} \, \mathbb{E} \left[ |W_{xy}|^{3t_1 + 2t_2 + t_3} \right]   \leq 2^{t_2} \, \mathbb{E} \left[ |\widehat{W}_{xy}|^{3t_1 + 2t_2 + t_3} \right]. \notag
\end{align}

Let 
\begin{equation}
    a := t_1 + t_2, \quad b := \frac{1}{2}(t_1 + t_3),
\end{equation}
and let $g \sim \mathcal{N}(0,1)$, then   for the sub-Gaussian variable $\widehat{W}_{xy} = \theta^{1/2} \mathrm{Bern}(\theta^{-1}) g$ we have
\begin{equation}
    \mathbb{E} \left[ |\widehat{W}_{xy}|^{3t_1 + 2t_2 + t_3} \right]
    = \theta^{a + b - 1} \, \mathbb{E} \left[ |g|^{2a + 2b} \right].
\end{equation}
 It remains to control $\mathbb{E}[|g|^{2a + 2b}]$ by $\mathbb{E}[(g^2 - 1)^a g^{2b}]$. For $a + b > 1$, by Lemma \ref{lem:Gaussian_moment} (ii),  
  we have
\begin{align}
    2^{t_2} \, \mathbb{E} \left[ |\widehat{W}_{xy}|^{3t_1 + 2t_2 + t_3} \right]&\le 4^{t_1+t_2}\theta^{a + b - 1}\mathbb{E} \left[ (g^2 - 1)^a g^{2b} \right]\notag\\
    &=4^{t_1+t_2}\mathbb{E}\left[ \left((B_{xy})^2 (|\widetilde{W}_{xy}|^2 - 1) \right)^{t_1 + t_2} \widehat{W}_{xy}^{t_1 + t_3} \right].
\end{align}
This indeed 
  yields the inequality \eqref{equ:5.15} with a constant factor $4^{t_1 + t_2}$.

Consequently, for every non-vanishing path, the absolute value of the left-hand side expectation in \eqref{equ:5.9} 
is bounded by its right-hand counterpart up to a multiplicative factor of  $4^{k_1+\cdots+k_s}$, which completes the proof.
\end{proof}

An immediate result   of    Lemma \ref{lem:5.4} follows. 
\begin{lemma}\label{lem:lemma5.5}     Set 
\begin{equation}\label{def:Dn}
    \mathcal{D}_n:=U_{n_j}(\frac{H}{2})-V_{n_j}=\sum_{k\ge 1}\sum_{a\in \{2,3\}^k}\sum_{l_0+\ldots +l_k=n}V_{l_0}\underline{\Phi_{a_1}V_{l_1}}\cdots\underline{\Phi_{a_k}V_{l_k}}
\end{equation}
and
\begin{equation}\label{equ:5.13}
\widehat{\mathcal{D}}_n
:= \sum_{k \geq 1} 8^k
\sum_{\substack{
    l_0 + \cdots + l_k = n
}}
\widehat{V}_{l_0} \,
\widehat{\underline{\Phi_{2} V_{l_1}}} \cdots
\widehat{\underline{\Phi_{2} V_{l_k}}},
\end{equation} then for all $s\ge 1$, 
    \begin{equation}
        |\mathbb{E}\big[\prod_{j=1}^s \tr ({\mathcal{D}}_{n_j})\big]|\le \mathbb{E}\big[\prod_{j=1}^s \tr( \widehat{\mathcal{D}}_{n_j})\big]
    \end{equation}
\end{lemma}
\begin{proof}

We adjust all terms $\underline{\Phi_3 V_l}$ to $\underline{\widehat{\Phi_2 V_l}}$ using Lemma~\ref{lem:5.4}. For fixed $k$, there are at most $2^k$ possible placements of $\Phi_3$. Consequently, we increase the combinatorial factor from $4^k$ in Lemma~\ref{lem:5.4} to $8^k$ in \eqref{equ:5.13}, completing the proof.
\end{proof}

\begin{figure}[htbp]
    \centering
    \begin{minipage}[t]{0.48\textwidth}
        \centering
        \includegraphics[width=\textwidth]{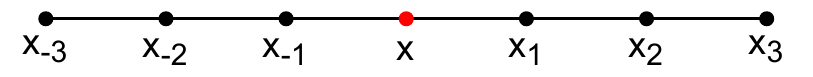}
        \caption{ \(\eta(x_{-1})=\eta(x_{1})\).  
    Condition \(\mathbf1(x_{i}\neq x_{i+2})\) forcing  \(\eta(x_{-3})\neq\eta(x_{-1})\) and \(\eta(x_{1})\neq\eta(x_{3})\),  
    but  \(\eta(x)=\eta(x_{-2})\) or \(\eta(x)=\eta(x_{2})\) possible.}
        \label{fig:red_point}
    \end{minipage}
    \hfill
    \begin{minipage}[t]{0.48\textwidth}
        \centering
        \includegraphics[width=\textwidth]{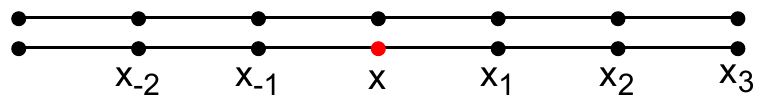}
        \caption{Red-point pairing, Case (i). No backtracking path occurs.}
        \label{fig:red_case1}
    \end{minipage}
\end{figure}
\begin{figure}[htbp]
    \centering
    \begin{minipage}[t]{0.45\textwidth}
        \centering
        \includegraphics[width=0.9\textwidth]{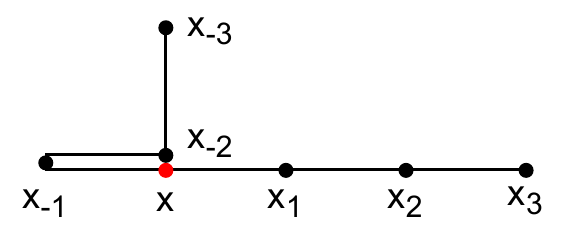}
        \caption{Red point pairing, Case (ii). The backtracking path creates a tail
        .}
        \label{fig:red_case2}
    \end{minipage}
    \hfill
    \begin{minipage}[t]{0.45\textwidth}
        \centering
        \includegraphics[width=0.9\textwidth]{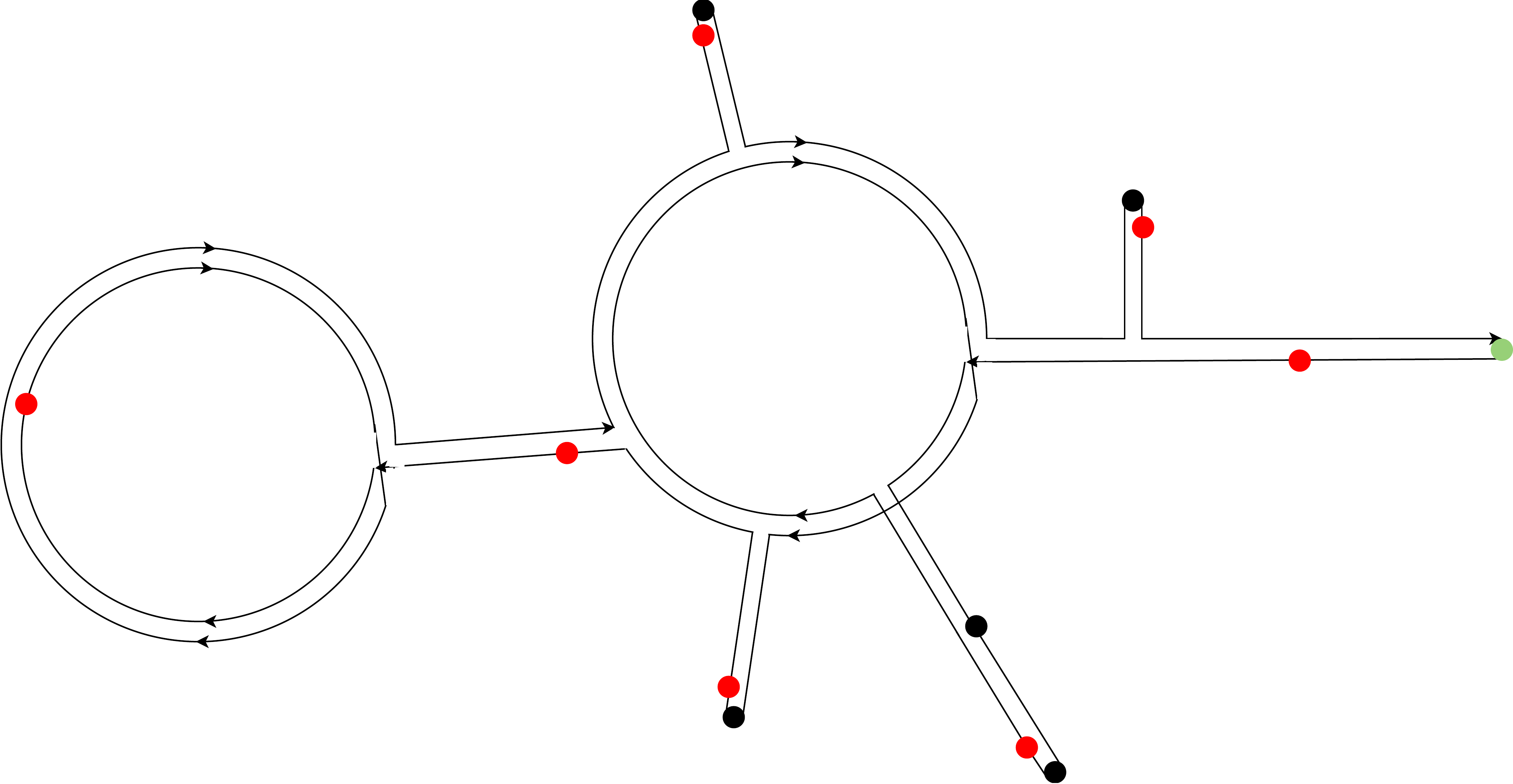}
        \caption{Full diagram with tail edges and red points. The green vertex is both  starting and ending  vertex; every backtracking path must attach at a  red point of case (ii), while other red points follow Case (i).}
        \label{fig:tail_diagram}
    \end{minipage}
\end{figure}

\begin{figure}[htbp]
    \centering
    \begin{minipage}[t]{0.45\textwidth}
        \centering
        \includegraphics[width=0.9\textwidth]{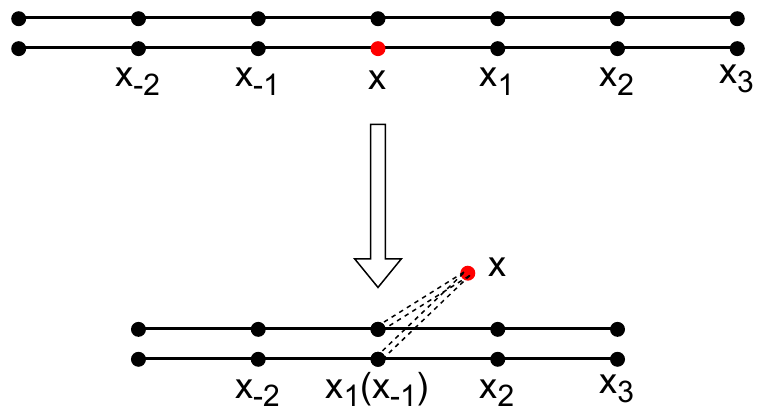}
        \caption{Shrink the edge adjacent to red point in pairing of Case (i).}
        \label{fig:red_shrink_case1}
    \end{minipage}
    \hfill
    \begin{minipage}[t]{0.45\textwidth}
        \centering
        \includegraphics[width=0.9\textwidth]{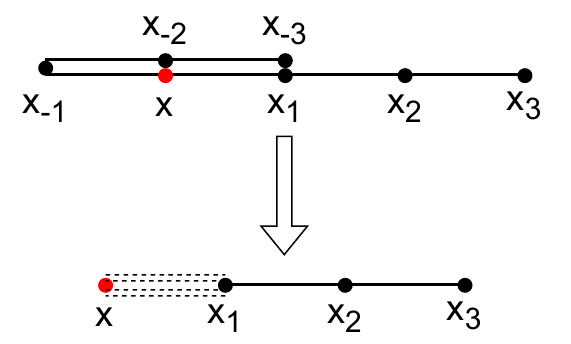}
        \caption{Shrink the edge adjacent to red point in pairing of Case (ii) when $\eta(x_{-3})= \eta(x_1)$.}
        \label{fig:red_shrink_case2}
    \end{minipage}
\end{figure}

\begin{figure}[htbp]
  \centering
  \includegraphics[width=\textwidth]{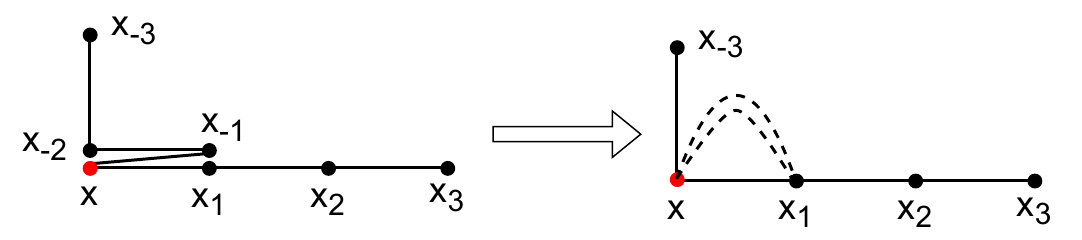}
  \caption{Shrink the edge adjacent to the red point in paring of case (ii) when \(\eta(x_{-3}) \neq \eta(x_1)\).}
  \label{fig:red_shrink_case3}
\end{figure}

\begin{figure}[htbp]
    \centering
    \includegraphics[width=\textwidth]{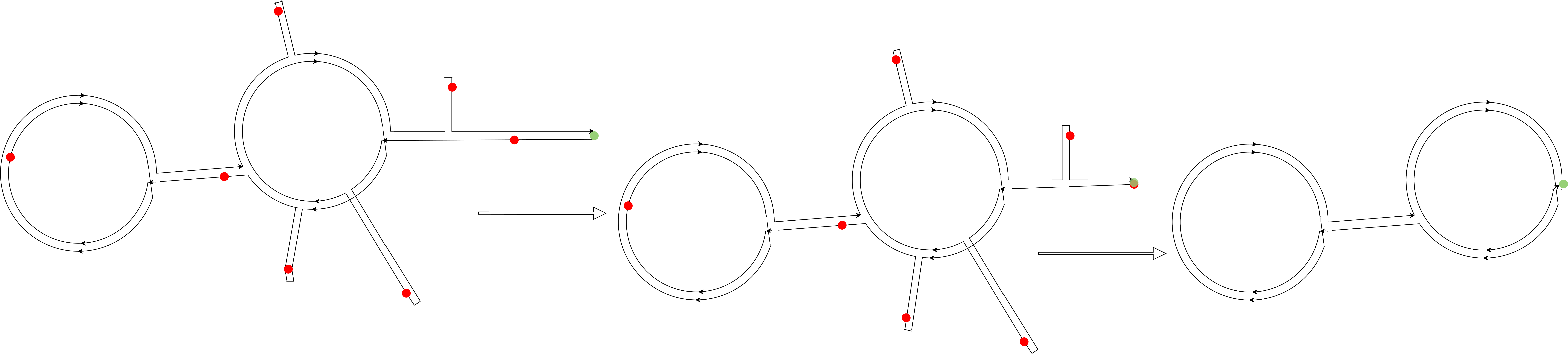}
    \caption{(a) First step: shrink the tail edge adjacent to the marked vertex (green). (b) Second step: shrink all tail edges induced by the red points.}
    \label{fig:tail_shrink}
\end{figure}

 The following lemma for  bounding  $\widehat{\mathcal{D}}_n$ plays a fundamental role in  establishing universality of sub-Gaussian   IRM.
\begin{lemma}\label{lem:D_n=0}
Let $n=n_1+\cdots+n_s$, 
if  $t_N \ll n  = O(N^{1/3})$ and $  \theta\gamma n^2 t_N \ll N$ as $N\to \infty$,  then
\begin{equation}\label{equ:5.22}
    \mathbb{E}\Bigl[\prod_{j=1}^s \tr\!\bigl(\tfrac{1}{n_j+1}\,\widehat{\mathcal{D}}_{n_j}\bigr)\Bigr] = o(1).
\end{equation}
\end{lemma}

\begin{proof}

We focus initially on the $s=1$ case; the general case follows by applying identical arguments term-by-term.

Similar to Gaussian IRM case, each trace moment is represented diagrammatically through Wick pairings of the centered Gaussian entries, conditioned on the Bernoulli random variables $B_{ij}$. 

\textbf{Reduction outline}  can be summarized  as follows. 

\begin{enumerate}
  \item[(i)] {\bf Wick pairings.} 
 Condition on Bernoulli variables $B_{ij}$ and apply Wick's formula to the Gaussian entries, and we  have   
    diagrammatic correspondence: each non-vanishing term corresponds to a pairing of $H$-edges in the graphical expansion.
  \item[(ii)]   {\bf Forbidden backtracks.}  By construction of the nonbacktracking powers $V_n$, no two \emph{consecutive} $H$–edges may pair.  However, when  inserting  operators $\underline{\widehat{\Phi_{a_i}V_{l_i}}}$, those forced “gaps” (the indicator $\mathbf{1}(x_i\neq x_{i+2})$) are \emph{not} imposed across the insertion points,  see Figure \ref{fig:red_point} for example).  Consequently, pairings may create local backtracks---these show up as extra “tails” in the diagram.  
 We call  such configurations \textbf{diagrams with tail edges}, analogous to cases without red points in \cite{feldheim2010universality, sodin2010spectral, liu2023edge}. 
  \item[(iii)] {\bf Two sources of tails.}
    \begin{itemize}
      \item[(a)] At the endpoints of the trace: since in the trace we require $x_0=x_n$ but do not enforce $\mathbf{1}(x_1\neq x_{n-1})$, the first and last edges can pair.
      \item[(b)] At a $\Phi_2$ insertion: if 
      \begin{equation}
        H_{x_{-1},x_0}\;(\widehat{\Phi}_2)_{x_0x_0}\;H_{x_0,x_1}=H_{x_{-1},x_0}\;\big(\sum_{y}|\widehat{H}_{x_0 y}|^2-\sigma_{x_0 y}^2\big)\;H_{x_0,x_1},
      \end{equation}
      then there is no restriction for $y$, so $y=x_1$ and $y=x_{-1}$ become possible, again creating a tail.
    \end{itemize}
  \item[(iv)] {\bf $\Phi_2$ itself forbidding  local backtracks.}  By definition,
  \begin{equation}
    (\widehat{\Phi}_{2})_{xx}
    \;=\;
    \sum_{y:y\ne x}B_{xy}^2\sigma_{xy}^2(|\widetilde{W}_{xy}|^2-1)+B_{xx}^2\sigma_{xx}^2(|\widetilde{W}_{xx}|^2-1),
  \end{equation} 
  the subtracted variance term precisely removes any contribution from a backtracking pairing ($\widetilde{W}_{xy}$ and $\widetilde{W}_{xy}$ are paired in opposite direction) of the two $H_{x y}$ edges at that vertex.  Hence no backtracking pairing can occur \emph{inside} a single $\Phi_2$ block.
  \item[(v)]{\bf  Integration over $B_{ij}$.}  Since $ 
\mathbb{E}[(B_{ij})^{2k}] \;=\;\theta^{\,k-1},$ 
whenever an (unordered) edge $\{x,y\}$ appears $2k$ times in the preimage of the diagram,  integrating  out the $B_{ij}$’s  contributes a factor $\theta^{k-1}$ .  We can thus introduce   a \textbf{\emph{coupling}} structure 
\begin{equation}
\mathcal{C} \;=\;\{\mathcal{C}_e\}_{e},    
\end{equation}
where each $\mathcal{C}_e$ is the multiset of preimage-edges glued to the base-edge $e$. See Figure \ref{fig:blue_edge} for coupling examples.
 \item[(vi)]{\bf  Tail shrinking.} 
We further shrink the tail edge incident to the marked point. This modification relaxes the boundary summation condition from $\sum_{e\in \partial D_j} w_e = n_j$ to $\sum_{e\in \partial D_j} w_e \leq n_j$, consistent with our prior analysis; see Figure~\ref{fig:tail_shrink}.

 \item[(vii)]{\bf  Red point.} 
Throughout our analysis, 
we will mark every vertex $x_i$ where a $\Phi_2$ insertion occurs as a \textbf{\emph{red point}} : at each red point its two neighboring labels must coincide, and further local backtracks can occur. 
See Figure~\ref{fig:red_point} for a visual representation.

\end{enumerate}

It remains to prove that any diagram containing at least one red point or nontrivial coupling contributes \(o(1)\). We establish this key claim through two steps:

\medskip
\noindent\textbf{Step 1: Suppression by red-point coincidences.}

Fix a red point \( v \not\in V \), distinct from all original diagram vertices. The vertex labeling at \( v \) admits two configurations in  Figure~\ref{fig:red_point}:
Case(i): $\eta({x})\ne \eta(x_{-2})$ and $\eta(x)\ne \eta(x_{2})$ (Figure \ref{fig:red_case1}); Case (ii): $\eta({x})= \eta(x_{-2})$ or $\eta(x)= \eta(x_{2})$ (Figure \ref{fig:red_case2}). In case (ii), $x_{-2}x_{-1}x$ forms  a backtracking path, and thus contributes a tail to the diagram.

For both case (i) and case (ii), we can shrink the edge adjacent to red points  as its two adjacent preimage-edges must share the same label, see Figure \ref{fig:red_shrink_case1}, \ref{fig:red_shrink_case2} and \ref{fig:red_shrink_case3}. In Figure \ref{fig:red_shrink_case1}, \ref{fig:red_shrink_case2}, summing over the label of $x$ yields at most
\begin{equation}
\sum_{x}
p_{1}(x_1,x)\,p_{1}(x,x_1)=p_2(x_1,x_1)
\;\le\;
\frac{\gamma\,t_N}{N}.
\end{equation}
Such red points couple the adjacent edges, which leads to an extra $\theta$ factor. In Figure \ref{fig:red_shrink_case3}, we remove the dashed line directly. If we remove the edges $(x_1,x),~(x,x_1)$, the rest is still the diagram with edge weight $w_e'= w_{e}-2$ (Figure \ref{fig:red_shrink_case1}, \ref{fig:red_shrink_case2}) or $w_e-1$ (in case $\eta(x_{-3})\ne \eta(x_1)$, Figure \ref{fig:red_shrink_case3}). It is also possible to couple $\{x,x_1\}$ with other edge, but there are at most $k$ such pairs, contributing a factor \begin{equation} \sum_{i=0}^{k}\theta^i\binom{k}{i}\Big(\frac{n\cdot \gamma t_N}{N}\Big)^i=\Big(1+\frac{ \gamma\theta n t_N}{N}\Big)^k.\end{equation} 
Here  the binomial factor $\binom{k}{i}$ appears since   there are $i$ edges (chosen from $k$ edges) coupled with other edges and $n^i$ edges coming  from the total possible choices of the position of coupled edges. 

Now we sum over all the possible position of red points and weight of tail edges. We first choose the position of red points to grow the tail edge and then sum over the length of the tails, thus the number of such choices does not exceed ${n^{2k}}/{k!}$. After removing all tail edges and their incident vertices, the reduced diagram lies in $\mathscr{D}_{s;\beta}^*$ (Figure~\ref{fig:tail_shrink}). 
The subsequent summation over red points then yields a global combinatorial factor
\begin{equation}
\sum_{k\ge 1}\frac{n^{2k}}{k!}\,\Bigl(\tfrac{\theta \gamma t_N}{N}\Bigr)^{k}(1+\frac{ \gamma\theta n t_N}{N})^k=e^{\frac{\theta \gamma n^2 t_N }{N}(1+\frac{\theta n\gamma t_N}{N})}-1,
\end{equation}
which vanishes as \(N \to \infty\) under the scaling condition
\begin{equation}
N \gg \theta\gamma n^2 t_N.
\end{equation}

If a red point instead lies within \(V\) or on any edge with weight $w_e\le2$, a similar (and strictly stronger) suppression bound holds.

\noindent\textbf{Step 2: Control of    couplings.}

We now show that any nontrivial coupling (i.e.\ \(l>0\)) yields only a negligible contribution.  Let
\begin{equation}
l \;=\;\sum_{e}\bigl(\lvert\mathcal C_e\rvert-1\bigr),
\qquad
t \;=\;l \;+\;\sum_{e:\,\lvert\mathcal C_e\rvert>1}1,
\end{equation}
so there are $t-l$ coupled-edge blocks.  Merging each block into a single weight-1 edge and contracting those edges to “blue” vertices produces $\overline\Gamma$, see Step 1 and 2 in Figure~\ref{fig:blue_edge}.  Assuming no two blocks touch each other (i.e., they do have no endpoint in common) or an vertex in $V$,
\begin{equation}
|V(\overline\Gamma)|=|V(\Gamma)|+t-l,
\qquad
|E(\overline\Gamma)|=|E(\Gamma)|+t.
\end{equation}

By Proposition~\ref{prop:F_upper_bound}, removing edges outside a spanning tree of $\overline\Gamma$ (see Step 3 in Figure~\ref{fig:blue_edge}) and summing over labels yields
\begin{equation}
\frac{n^{\,\lvert V(\overline\Gamma)\rvert-1}}
     {(\lvert V(\overline\Gamma)\rvert-1)!}
\Bigl(\frac{(\gamma\,t_N)\lor n}{N}\Bigr)^{\,\lvert E(\overline\Gamma)\rvert-\lvert V(\overline\Gamma)\rvert+1}.
\end{equation}
Restoring all $t$ contracted edges (see Step 4 in Figure~\ref{fig:blue_edge}) and removing $l$ multiple edges,  we introduce the factor
\begin{equation}
\Bigl(\frac{\theta\,\gamma\,t_N}{N}\Bigr)^{l}.
\end{equation}
Hence, for a fixed coupling,
\begin{align}
&\frac{n^{\,\lvert V(\overline\Gamma)\rvert-1}}
     {(\lvert V(\overline\Gamma)\rvert-1)!}
\Bigl(\tfrac{(\gamma\,t_N)\lor n}{N}\Bigr)^{\,\lvert E(\overline\Gamma)\rvert-\lvert V(\overline\Gamma)\rvert+1}
\Bigl(\tfrac{\theta\,\gamma\,t_N}{N}\Bigr)^{l}\\
&=\Bigl(\tfrac{\theta\,\gamma\,t_N}{N}\Bigr)^{l}\frac{n^{\,\lvert V(\overline\Gamma)\rvert-1}}
     {(\lvert V(\overline\Gamma)\rvert-1)!}
\Bigl(\tfrac{(\gamma\,t_N)\lor n}{N}\Bigr)^{\,\lvert E(\Gamma)\rvert-\lvert V(\Gamma)\rvert+1+l}\tag{use $|E(\overline \Gamma)|-|V(\overline \Gamma)|=|E(\Gamma)|-|V(\Gamma)|+l$}\\
&\le \Bigl(\tfrac{\theta\,\gamma\,t_N}{N}\Bigr)^{l}\Bigl(\tfrac{n}{N}\Bigr)^{l}\frac{n^{\,\lvert V(\overline\Gamma)\rvert-1}}
     {(\lvert V(\overline\Gamma)\rvert-1)!}
\Bigl(\tfrac{(\gamma\,t_N)\lor n}{N}\Bigr)^{\,\lvert E(\Gamma)\rvert-\lvert V(\Gamma)\rvert+1}\tag{use $n\ge \gamma t_N$}\\
&\le \Bigl(\tfrac{\theta\,\gamma\,t_N}{N}\Bigr)^{l}\Bigl(\tfrac{n}{N}\Bigr)^{l}\frac{n^{\,\lvert V(\Gamma)\rvert-1+t-l}}
     {(\lvert V(\Gamma)\rvert-1)!}
\Bigl(\tfrac{(\gamma\,t_N)\lor n}{N}\Bigr)^{\,\lvert E(\Gamma)\rvert-\lvert V(\Gamma)\rvert+1}\tag{use $|V(\overline\Gamma)|=|V(\Gamma)|+t-l$}\\
&=(\theta\,\gamma\,t_N)^{l}
\,n^{\,t}N^{-2l}\,
\frac{n^{\,\lvert V(\Gamma)\rvert-1}}{(\lvert V(\Gamma)\rvert-1)!}
\Bigl(\tfrac{(\gamma\,t_N)\lor n}{N}\Bigr)^{\,\lvert E(\Gamma)\rvert-\lvert V(\Gamma)\rvert+1}\notag
\\
&\le
(\theta\,\gamma\,t_N)^{l}
\,\big(\frac{n}{N}\big)^{\,t}\,
\frac{n^{\,\lvert V(\Gamma)\rvert-1}}{(\lvert V(\Gamma)\rvert-1)!}
\Bigl(\tfrac{(\gamma\,t_N)\lor n}{N}\Bigr)^{\,\lvert E(\Gamma)\rvert-\lvert V(\Gamma)\rvert+1}\tag{use $2l\ge t$}.
\end{align}

Next, we sum over all couplings. Suppose that there are $h_i$ blocks of size $i\ge2$ in $\mathcal{C}$, then  we have
\begin{equation}
l=\sum_{i\ge2}(i-1)h_i,
\quad
t=\sum_{i\ge2}i\,h_i.
\end{equation}
The number of way for choosing  $t$ edges to couple  is   at most
\begin{equation}
2^t\frac{\bigl(t+\lvert E(\Gamma)\rvert\bigr)^{\,\lvert E(\Gamma)\rvert-1}}
     {(\lvert E(\Gamma)\rvert-1)!}.
\end{equation}
  This is because, once  the number of coupled edge   in $\mathcal C_e$ with  $e\in E(\Gamma)$ being denoted by  $a_e$, we have  $\sum_{e\in E(\Gamma)}a_e=t$, and there are 2 choices for each coupled edge  since each $e\in E(\Gamma)$ is traversed twice and has two side. 
  The number of arranging of each block of size $i$ is 
\begin{equation}
\frac{t!}{\prod_{i=2}^\infty h_i!\,(i!)^{h_i}}.
\end{equation}
 Hence the total coupling contribution is
\begin{equation}\label{equ:5.29}
\sum_{\substack{h_i\ge0\\\sum\limits_{i\ge2} i\,h_i=t}}
(\theta\,\gamma\,t_N)^{l}
\,(\frac{n}{N})^{\,t}\,2^t\frac{\bigl(t+\lvert E(\Gamma)\rvert\bigr)^{\,\lvert E(\Gamma)\rvert-1}}
     {(\lvert E(\Gamma)\rvert-1)!}\frac{t!}{\prod_{i=2}^\infty h_i!(i!)^{h_i}}.
\end{equation}
To estimate \eqref{equ:5.29}, introduce the exponential generating function (set $M=\theta\,\gamma\,t_N$)
\begin{equation}
A(z)
=\sum_{h_i\ge0}
z^{\sum i\,h_i}
\prod_{i=2}^\infty
\frac{M^{(i-1)h_i}}{h_i!\,(i!)^{h_i}}
=\exp\!\Bigl(\tfrac{e^{Mz}-1-Mz}{M}\Bigr).
\end{equation}
Then the coefficient  for the term $z^t$
\begin{equation}
    \begin{aligned}
{[z^t]}A(z)&\le\;
[z^t]\,
\exp\!\bigl(Mz^2e^{Mz}\bigr)=\sum_{i=0}^{\lfloor t/2\rfloor}
\frac{M^{i}}{i!}
\frac{M^{\,t-2i}}{(t-2i)!}\\
&=
\frac{M^{t}}{t!}\!\sum_{i=0}^{\lfloor t/2\rfloor}\frac{t!}{i!\,(t-2i)!i!}i!M^{-i}\le \frac{3^tM^t}{t!}\sum_{i=0}^{\lfloor t/2\rfloor}\big(\frac{t}{M}\big)^{i}\le 3^t\frac{M^t+t^{[t/2]}M^{\lceil t/2\rceil}}{(t-1)!}.
\end{aligned}
\end{equation}
So, using $t\le n$, the sum in \eqref{equ:5.29} can be bounded by
\begin{equation}
\begin{aligned}
    &\sum_{t=1}^n
\Bigl(\frac{n}{N}\Bigr)^{t}2^t
\frac{\bigl(t+\lvert E(\Gamma)\rvert\bigr)^{\,\lvert E(\Gamma)\rvert-1}}
     {(\lvert E(\Gamma)\rvert-1)!}t
3^t\bigl((\theta\gamma\,t_N)^{t}+(\theta\gamma\,t_N\,n)^{t/2}\bigr)\\
&\le \sum_{t=1}^n
\Bigl(\frac{n}{N}\Bigr)^{t}12^t
e^{t+|E(\Gamma)|}
\bigl((\theta\gamma\,t_N)^{t}+(\theta\gamma\,t_N\,n)^{t/2}\bigr)\\
&\le C^{|E(\Gamma)|}\sum_{t=1}^n
\Bigl(\frac{Cn}{N}(\theta\gamma t_N+\sqrt{\theta\gamma n t_N})\Bigr)^{t}
= o(1),
\end{aligned}
\end{equation}
provided
\begin{equation}
N\;\gg\;n\bigl(\theta\gamma t_N + \sqrt{\theta\gamma n t_N}\bigr).
\end{equation}

Hence the uncoupled case (i.e. $t=0$) dominates, and in that case any diagram with red point has been  shown to be $o(1)$. 

Combining these two steps completes completes the proof of \eqref{equ:5.22}.
\end{proof}

\begin{figure}
    \centering
    \includegraphics[width=\linewidth]{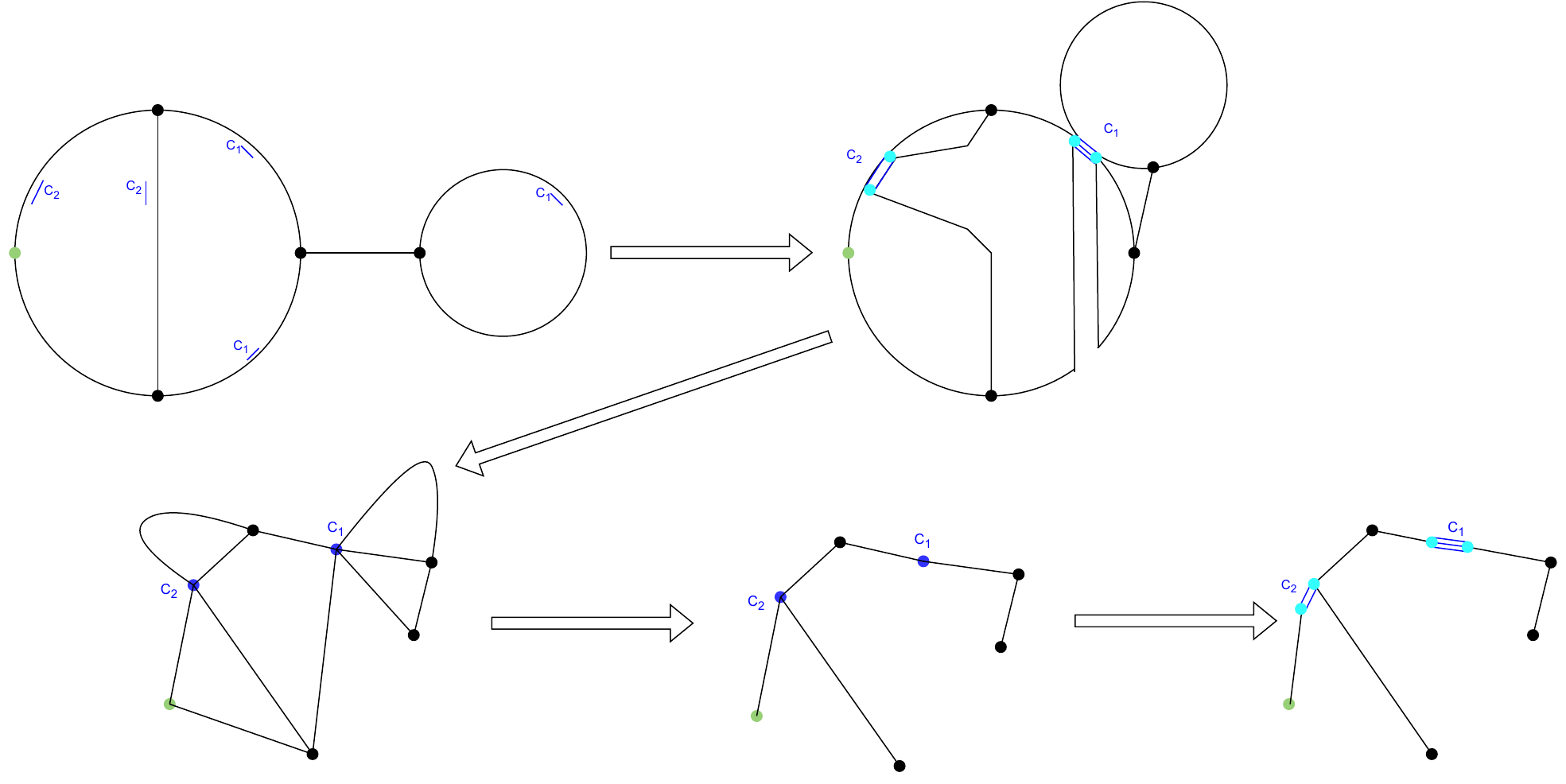}
    \caption{Coupling of blue edges.  \emph{(a)}~In the original graph each (non–blue) edge is really a pair of directed edges (omitted here).  \emph{(b)}~First we couple the edges in $\mathcal C_1$ and $\mathcal C_2$.  \emph{(c)}~Next we merge parallel edges into single blue vertices.  \emph{(d)}~We extract a spanning tree of the resulting graph.  \emph{(e)}~Finally, we restore the original parallel edges.}
    \label{fig:blue_edge}
\end{figure}

At this stage, with Lemma~\ref{lem:D_n=0}, we are now able to complete the proof of Theorem~\ref{thm:main_thm}. 
 
\begin{proposition}\label{prop:V=U}
If  $t_N \ll n_j = O(N^{1/3})$ and $  \theta\gamma  n^2t_N \ll N$,  then 
\begin{equation}
\mathbb{E}\Big[\prod_{j=1}^s \tr\Big( \frac{1}{n_j+1} V_{n_j} \Big)\Big]
=\mathbb{E}\Big[\prod_{j=1}^s \tr\Big( \frac{1}{n_j+1} \widehat{V}_{n_j} \Big)\Big]  + o(1).
\end{equation}
\end{proposition}

\begin{proof}
From Step 2 in the proof of Lemma~\ref{lem:D_n=0}, we know that diagrams involving any non-trivial coupling of edges can be  neglected. In the absence of coupling, the expectations for $V_n$ and $\widehat{V}_n$ agree exactly by construction. Thus, their difference vanishes in expectation, completing the proof.
\end{proof}

\begin{corollary}\label{coro:U=U}
    If   $t_N \ll n_j=O(N^{\frac{1}{3}})$ and $  \theta\gamma n^2 t_N \ll N $,  then
    \begin{equation}
        \mathbb{E}\Big[\prod_{j=1}^s \tr \frac{1}{n_j+1}U_{n_j}(\frac{H}{2})\Big]=\mathbb{E}\Big[\prod_{j=1}^s \tr \frac{1}{n_j+1}U_{n_j}(\frac{\widehat{H}}{2})\Big]+o(1).
    \end{equation}
\end{corollary}
\begin{proof}
    By \eqref{def:Dn}, we obtain
    \begin{align}
        &\left|\mathbb{E}\Big[\prod_{j=1}^s \tr \frac{1}{n_j+1}U_{n_j}(\frac{H}{2})\Big]-\mathbb{E}\Big[\prod_{j=1}^s \tr \frac{1}{n_j+1}V_{n_j}\Big]\right|\notag\\&=\left|\mathbb{E}\Big[\prod_{j=1}^s \tr \frac{1}{n_j+1}(V_{n_j}+\mathcal{D}_{n_j})\Big]-\mathbb{E}\Big[\prod_{j=1}^s \tr \frac{1}{n_j+1}V_{n_j}\Big]\right|\notag\\
        &\le \sum_{S\subset [s],~S\neq \emptyset}\left|\mathbb{E}\Big[\prod_{j\in S} \tr \frac{1}{n_j+1}V_{n_j}\prod_{j\in S^c}\tr\frac{1}{n_j+1}\mathcal{D}_{n_j}\Big]\right|\\
        &\le \sum_{S\subset [s],~S\neq \emptyset}\Big(\prod_{j\in S}\mathbb{E}\Big[ (\tr \frac{1}{n_j+1}V_{n_j})^{2s}\Big]\prod_{j\in S^c}\mathbb{E}\Big[(\tr\frac{1}{n_j+1}\mathcal{D}_{n_j})^{2s}\Big]\Big)^{\frac{1}{2s}}=o(1).\notag
    \end{align}
    Here in the last line we use the H\"older inequality and bounds in Lemma \ref{lem:D_n=0}. 
    
    Combining  Proposition \ref{prop:V=U}, we thus    complete the proof.
\end{proof}
\begin{remark}
    The proof of Corollary \ref{coro:U=U} for the deformed case is similar but requires slight modifications. We first adopt the Chebyshev expansion (Proposition \ref{prop:path_expansion_H_plus_A}) for deformed case and sum over the subscripts of consecutive $\Phi_A$ to obtain $(A^k)_{xy}$ in $\mathcal{D}_n$. Then, we apply Lemma \ref{lem:5.4} to establish an upper bound by replacing $(A^k)_{xy}$ with its absolute value, $|(A^k)_{xy}|$, similar to the step in Lemma \ref{lem:lemma5.5}. From this point, we proceed with the Wick pairing step from Lemma \ref{lem:D_n=0} to obtain an upper bound for the coupling of the inner edges. Following the argument in Lemma \ref{lem:D_n=0}, the diagram function is consequently bounded by the right-hand side of \eqref{equ:3.5} (from the Gaussian case), where the $(A^k)_{xy}$ terms have been replaced by $|(A^k)_{xy}|$, and the remainder of the proof proceeds identically.
\end{remark}
To conclude this section, we complete the proof of the general case of Theorem~\ref{thm:main_thm}.

\begin{proof}[Proof of Theorem~\ref{thm:main_thm}: general case]

By Corollary~\ref{coro:U=U}, the mixed moments of Chebyshev polynomials under the conditions of Theorem~\ref{thm:main_thm} exhibit asymptotics identical to those of the GOE/GUE ensemble. The proof of Theorem~\ref{thm:main_thm} therefore follows a procedure largely identical to the Gaussian case.
\end{proof}

\section{Extension to inhomogeneous Wishart matrices}\label{sec:section6}

\begin{definition}[Inhomogeneous Wishart matrices]\label{def:wishart}
Let \(M\le N\) and  $\alpha = {M}/{N}$ where $M$ may depend on $N$, define   
\begin{equation} \label{non-H}
H_N \;=\; \Sigma_N \circ W_N,
\end{equation} as an inhomogeneous Wishart  matrix  and the corresponding additive perturbation (namely, signal-plus-noise matrices)   \begin{equation}\label{eq:deformed-wishart}
X_N \;=\; (H_N + A_N)\,(H_N + A_N)^*.
\end{equation}
where \(W_N=(W_{ij})\) is an \(M\times N\) random matrix with real or complex   entries  and \(\Sigma_N=(\sigma_{ij})\) is a deterministic  \(M\times N\)  matrix  with nonnegative entries. 
Precisely, the following assumptions hold:

\begin{enumerate}[leftmargin=*,label=(C\arabic*)]
  \item \textbf{(Entry moment)}  
 The entries $\{W_{ij}\}_{i,j=1}^N$ are independent, symmetric, and satisfy
 \begin{equation}
      \begin{cases}
       \mathbb{E}[W_{ij}^2]=1, & \text{real 
        case},\\
        \mathbb{E}[|W_{ij}|^2]=1,\;\mathbb{E}[W_{ij}^2]=0, & \text{complex 
        case}.
      \end{cases}
    \end{equation} 
    and for some \(\theta\ge1\) and all integers \(k\ge2\),
    \begin{equation}
      \mathbb{E}[|W_{ij}|^{2k}]
      \;\le\;
      \theta^{\,k-1}(2k-1)!!,
    \end{equation}
  where   \(\theta=\theta_N\to\infty\)  may be allowed  if needed.

  \item \textbf{(Variance profile)}  
    The profile matrix \(\Sigma_N=(\sigma_{ij})\) satisfies
    \begin{equation}
      \sum_{j=1}^N \sigma_{ij}^2 = 1
      \quad(\forall\,i\in[M]),
      \qquad
      \sum_{i=1}^M \sigma_{ij}^2 = \alpha
      \quad(\forall\,j\in[N]).
    \end{equation}
    Equivalently, the bipartite transition matrix on \(S=[M]\sqcup[N]\) is
    \begin{equation}\label{equ:transition_matrix}
      P_{S}
      = \begin{pmatrix}
          0 & (\sigma_{ij}^2)_{M\times N} \\
          \alpha^{-1}(\sigma_{ij}^2)_{N\times M}^T & 0
        \end{pmatrix}.
    \end{equation}

  \item \textbf{(Rank-\(r\) perturbation)}  
    The deterministic matrix \(A_N\) admits a decomposition
\begin{equation}
    A_N = Q_1\,\Lambda\,Q_2^*, 
    \quad 
    \Lambda = \big( \mathrm{diag}(a_1,\dots,a_r, \underbrace{0,\dots,0}_{M-r}), \; 0_{M \times (N-M)} \big),
\end{equation}
where \(Q_1\) and \(Q_2\) are orthogonal (or unitary) matrices of size \(M\) and \(N\), respectively, and \(\underbrace{0,\dots,0}_{M-r}\) denotes the \(M-r\) zeros in the diagonal block.
     Moreover, for some fixed $\tau>0$, 
    \begin{equation}
      \|A_N\|_{\mathrm{op}}
      \;\le\;
      \sqrt{\alpha} \;+\;\tau\,N^{-1/3}.
    \end{equation}
\end{enumerate}
When  \(W_N\) has i.i.d.\ Gaussian entries,   the resulting model is denoted  by \(\widetilde X_N\) with \(\beta=1,2\).
\end{definition}

\begin{definition}[Short-to-Long Mixing on bipartite chain]\label{def:wishart_mixing}
Given  a sequence of positive integers  \(t_N\),   the bipartite Markov chain with state space \(S=[M]\sqcup[N]\) and transition matrix  $P_S$ given in \eqref{equ:transition_matrix}
is said to be   \emph{Short-to-Long Mixing} at time \(t_N\),  if its \(n\)-step transition probabilities \(p_n(x,y)\) obey:

\begin{enumerate}[leftmargin=*,label=(D\arabic*)]
  \item \label{itm:E1} \textbf{(Short-time average mixing)}  
    There exists \(\gamma\ge1\) and \(N_0>0\) so that for all \(N\ge N_0\) and all \(x,y\in S\),
    \begin{equation}
      \frac1{t_N}\sum_{n=1}^{t_N}p_n(x,y)
      \;\le\;
      \begin{cases}
        \displaystyle \frac{\gamma}{N}, & y\in[N],\\[6pt]
        \displaystyle \frac{\gamma}{M}, & y\in[M].
      \end{cases}
    \end{equation}

  \item \label{itm:E2} \textbf{(Long-time uniform mixing)}  
    There exists \(\delta\in(0,0.1)\) such that for all \(n\ge t_N\) and all \(x,y\in S\),
    \begin{equation}
      \bigl|p_n(x,y)+p_{n+1}(x,y)-\tfrac1{N}\bigr|\;\le\;\frac{\delta}{N}
      \quad\text{if }y\in[N],
      \quad
      \bigl|p_n(x,y)+p_{n+1}(x,y)-\tfrac1{M}\bigr|\;\le\;\frac{\delta}{M}
      \quad\text{if }y\in[M].
    \end{equation}
\end{enumerate}
The additional assumption that  $\delta = \delta_N \to 0$  may be imposed   when necessary.
\end{definition}

\begin{figure}[htbp]
  \centering
  \begin{minipage}{0.48\textwidth}
    \centering
    \includegraphics[width=\textwidth]{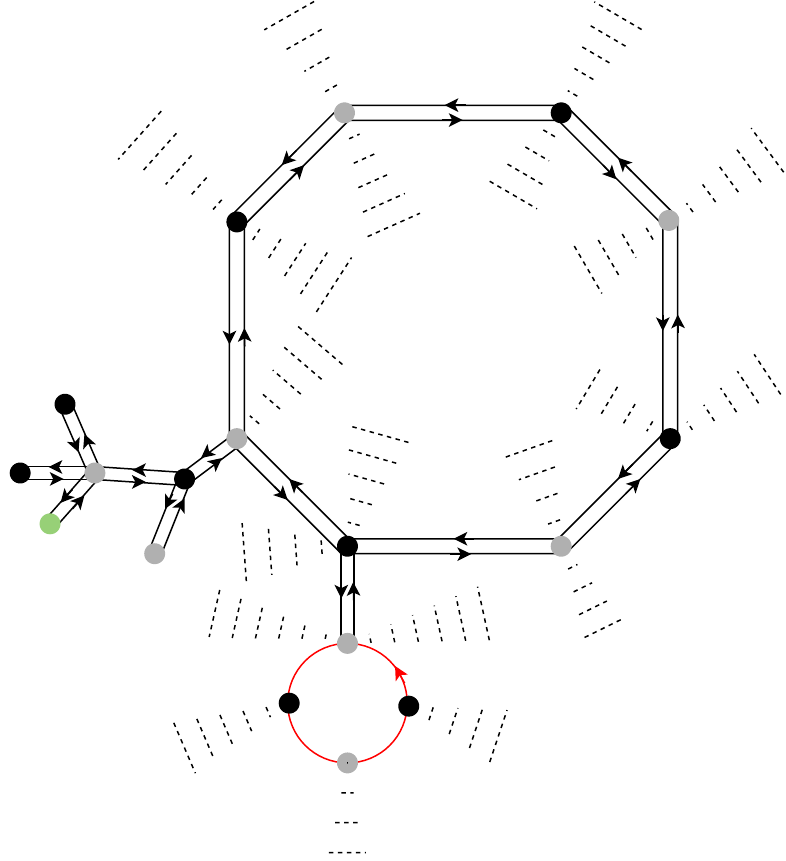}
    \caption{Ribbon-graph examples in the Wishart case.}
    \label{fig:wishart_ribbon_example}
  \end{minipage}
  \hfill
  \begin{minipage}{0.48\textwidth}
    \centering
    \includegraphics[width=\textwidth]{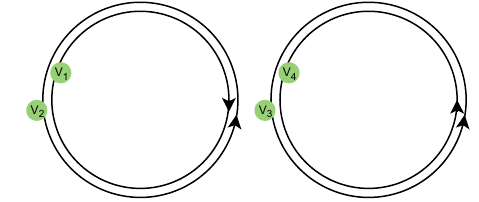}
    \caption{Graphical diagram for the leading term in linear statistics. In real case both orientations are possible but in complex case the second orientation is forbidden.}
    \label{fig:linear_statistics}
  \end{minipage}
\end{figure}

\begin{figure}[t]
    \centering
    \includegraphics[width=0.6\textwidth]{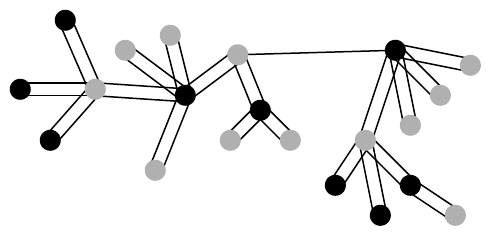}
    \caption{Catalan tree enumeration.  Each white vertex carries weight \(\alpha\), while each black vertex carries weight 1.  The total weight of a path is the product of its vertex weights, giving the \(\alpha\)-Catalan numbers in the Wishart moment expansion.}
    \label{fig:catalan_tree}
\end{figure}

Recall that the Marchenko–Pastur 
law with aspect ratio $0<\alpha\le1$ admits the   density
\begin{equation}
\mu_{\mathrm{MP},\alpha}(\mathrm{d}x)
= \frac{\sqrt{(\lambda_{+}-x)(x-\lambda_{-})}}{2\pi\alpha\,x}\,\mathbf{1}_{[\lambda_{-},\lambda_{+}]}(x)\,\mathrm{d}x,
\end{equation}
where \begin{equation}
     \lambda_{-}=(1-\sqrt{\alpha})^2,\quad \lambda_{+}=(1+\sqrt{\alpha})^2. 
\end{equation}
Its $k$-th moment  is given by 
\begin{equation} 
\quad \mu_0=1,\quad 
\mu_k
=\sum_{j=0}^{k-1}\frac{1}{j+1}\binom{k}{j}\binom{k-1}{j}\,\alpha^j, \ 
 k\ge1.
\end{equation}
\begin{proposition}\label{prop:wishart_moment}
Let $X$ be the real symmetric (or Hermitian) deformed Gaussian IRM–Wishart matrix from \eqref{eq:deformed-wishart}, and fix nonnegative integers $m_1,\dots,m_s$.  For each face $D_j$, set $n_j \;=\;\sum_{e\in\partial D_j}w_e$,  then
\begin{multline}
\label{eq:trace-moments}
\mathbb{E}\Bigl[\prod_{j=1}^s \tr \big(X^{m_j}\big)\Bigr]
=
\sum_{\Gamma \in \mathscr{D}_{s;\beta}}
\;\sum_{\substack{\eta:V(\Gamma)\to [M]\cup[N]}}
\;\sum_{\substack{w_e \ge 1 \\ \sum_{e \in \partial D_j} w_e = n_j, \ \forall j}}
\;\Biggl[\,
\prod_{j=1}^s C_{\mathrm{MP}}(m_j,n_j)\;\Biggr]
\;\\ \times 
\Biggl[\,
\prod_{e=(x,y)\in E_{\mathrm{int}}}p_{w_e}\bigl(\eta(x),\eta(y)\bigr)
\;\prod_{e=(z,w)\in E_b}(A^{(w_e)})_{\eta(z)\eta(w)}
\Biggr].  
\end{multline}
Here
\begin{itemize}
  \item $E_{\mathrm{int}}$ (resp.\ $E_b$) denotes the set of interior (resp.\ open) edges of~$\Gamma$.
  \item For an open edge $e=(z,w)$, the matrix product $(A^{w_e})_{\eta(z)\eta(w)}$ means
    \begin{equation}
(A^{\,(w_e)})_{\eta(z)\,\eta(w)}
:=
\begin{cases}
(A\,A^*)^{w_e/2}, 
& z,w\in[N],\;w_e\text{ even},\\[6pt]
(A^*\,A)^{w_e/2}, 
& z,w\in[M],\;w_e\text{ even},\\[6pt]
A\,(A^*A)^{\frac{w_e-1}{2}},
& z\in[N],\;w\in[M],\;w_e\text{ odd},\\[6pt]
A^*\,(A\,A^*)^{\frac{w_e-1}{2}},
& z\in[M],\;w\in[N],\;w_e\text{ odd}.
\end{cases}
\end{equation}
  \item With  $\mu_k$ being     the $k$-th moment of the Marchenko–Pastur law  with   $\alpha=M/N$,  
    \begin{equation}
    C_{\mathrm{MP}}(m,n):
    = 
    \begin{cases}
      \mu_m, 
        & n=0, \\[6pt]
      \dfrac{m}{n}
\displaystyle\sum_{\substack{k_1+\cdots+k_n=m\\k_1,\ldots,k_n\ge1}}
      \prod_{i=1}^n \mu_{k_i},
        & 1\le n\le m, \\[6pt]
      0, & n>m.
    \end{cases}
    \end{equation}
\end{itemize}
\end{proposition}

\begin{proof}
We first express each trace-moment $\mathbb{E}\bigl[\prod_{j=1}^s \tr X^{m_j}\bigr]$ as a sum over ribbon graphs~\(\Gamma\), see for example Figure \ref{fig:wishart_ribbon_example}.

Color vertices in \([N]\) black and those in \([M]\) white.  Each circuit of length \(m_j\) in~\(\Gamma\) decomposes into Catalan paths (which we shrink via Okounkov contraction as in Definition \ref{def:Okounkov_contraction}) together with the remaining edges in $\partial D_j$, an even number \(2n_j\).

It remains to compute the total weight of all ways to insert Catalan subpaths so that \(m_j-n_j\) pairs of steps are removed.  Write \(b_i\) (resp.\ \(w_i\)) for the total weight of all black (resp.\ white) Catalan subpaths of length \(2i\).  Then the total Catalan-insertion weight in the \(j\)th circuit is
\begin{equation}
\sum_{\substack{t_1+\cdots+t_{2n_j}=\,m_j-n_j\\t_i\ge0}}
\;\prod_{i=1}^{n_j} b_{t_{2i-1}}\;w_{t_{2i}}.
\end{equation}
Now we sum over the adjacent black and white vertices, see Figure \ref{fig:catalan_tree}. Since
\begin{equation}
\sum_{i=0}^k b_i\,w_{k-i} \;=\;\mu_{k+1},
\end{equation}
where \(\mu_{k+1}\) is the \((k+1)\)-th moment of the Marchenko–Pastur law, it follows that
\begin{equation}
\sum_{\substack{t_1+\cdots+t_{2n_j}=\,m_j-n_j}}
\prod_{i=1}^{n_j} b_{t_{2i-1}}\,w_{t_{2i}}
\;=\;
\sum_{\substack{k_1+\cdots+k_{n_j}=\,m_j\\k_i\ge1}}
\prod_{i=1}^{n_j}\mu_{k_i}.
\end{equation}

Finally, distributing the marked point uniformly over the \(n_j\) “skeleton” edges introduces the factor
\(\tfrac1{n_j}\sum_{r=1}^{n_j}k_r = \tfrac{m_j}{n_j}\).  Hence by setting 
\begin{equation}
 C_{\mathrm{MP}}(m_j,n_j)
\;=\;
\frac{m_j}{n_j}
\sum_{\substack{k_1+\cdots+k_{n_j}=\,m_j\\k_i\ge1}}
\prod_{i=1}^{n_j}\mu_{k_i},
\end{equation}
we know that the Catalan trees contribute the weight of 
\begin{equation}
\prod_{j=1}^s C_{\mathrm{MP}}(m_j,n_j),
\end{equation}
and the remaining part contributes
\begin{equation}
    \sum_{\Gamma \in \mathscr{D}_{s;\beta}}
\;\sum_{\substack{\eta:V(\Gamma)\to [M]\cup[N]}}
\;\sum_{\substack{w_e \ge 1 \\ \sum_{e \in \partial D_j} w_e = n_j, \ \forall j}}
\Biggl[\,
\prod_{e=(x,y)\in E_{\mathrm{int}}}p_{w_e}\bigl(\eta(x),\eta(y)\bigr)
\;\prod_{e=(z,w)\in E_b}(A^{(w_e)})_{\eta(z)\eta(w)}
\Biggr].  
\end{equation}
Thus we complete  the proof. 
\end{proof}

Similar to Proposition \ref{prop:ribbon}, we have 
\begin{theorem}\label{thm:Pn_moments}
Let \(X=(H_N + A_N)(H_N + A_N)^*\) be the deformed inhomogeneous Wishart matrix of Definition~\ref{def:wishart} with Gaussian entries, and let \(\{\mathcal P_n\}_{n\ge0}\) be a   family of polynomials defined by
\begin{equation}\label{def:Pn}
    x^m \;=\;\sum_{n=0}^{m} C_{\mathrm{MP}}(m,n)\,\mathcal P_{n}(x), \quad m=0,1,\ldots,
\end{equation}
then for any nonnegative integers \(n_1,\dots,n_s\),
\begin{equation}\label{eq:Pn_moments}
\mathbb{E}\Bigl[\prod_{j=1}^s \tr\mathcal P_{n_j}(X)\Bigr]
=
\sum_{\Gamma\in\mathscr{D}_{s;\beta}}
\;\sum_{\eta:V(\Gamma)\to [M]\cup[N]}
\;\sum_{\substack{w_e \ge 1 \\ \sum_{e \in \partial D_j} w_e = n_j, \ \forall j}}
\;\prod_{e\in E_{\mathrm{int}}}p_{w_e}\bigl(\eta(x_e),\eta(y_e)\bigr)
\;\prod_{e\in E_b}(A^{w_e})_{\eta(z_e)\eta(w_e)}.
\end{equation}
Here
\begin{itemize}[leftmargin=2em]
  \item \(\mathscr{D}_{s;\beta}\) is the set of \(s\)-face ribbon graphs (orientable if \(\beta=2\)) (see Definition \ref{def:diagram}),  
  \item \(E_{\mathrm{int}}\) denotes the interior edges of \(\Gamma\), and \(E_b\) its boundary edges (see Definition \ref{def:diagram}),  
  \item for each interior edge \(e=(x_e,y_e)\),  \(p_{w_e}(\eta(x_e),\eta(y_e))\) denotes the \(n\)-step transition probability  (see \eqref{equ:transition_matrix} for the Markov chain),  
  \item for each boundary edge \(e=(z_e,w_e)\), \((A^{(w_e)})_{\eta(z_e)\eta(w_e)}\) is defined as in Proposition \ref{prop:wishart_moment}.  
\end{itemize}
\end{theorem}
\begin{proof}
Start from Proposition~\ref{prop:wishart_moment} and proceed  by induction on the total degree \(n = n_1 + \cdots + n_s\), one can complete the proof.
\end{proof}

\begin{proposition}\label{cor:Pn_chebyshev}
For every integer \(n\ge1\),   \(\mathcal P_n\) defined  in \eqref{def:Pn} can be expressed in terms of Chebyshev polynomials of first kind as  
\begin{equation}\label{equ:Pn_to_Tn}
\frac{1}{2\,\alpha^{n/2}}\,\mathcal P_n(x)
\;=\;
T_n\!\Bigl(\frac{x - (1+\alpha)}{2\sqrt{\alpha}}\Bigr)
\;-\;
\int T_n\!\Bigl(\frac{t - (1+\alpha)}{2\sqrt{\alpha}}\Bigr)\,d\mu_{\mathrm{MP},\alpha}(t).
\end{equation}
Moreover, it can also  explicitly be    given by  Chebyshev polynomials of second  kind 
\begin{equation}\label{equ:Un_Pn}
 U_n\!\bigl(\tfrac{x-(1+\alpha)}{2\sqrt\alpha}\bigr)
   +\sqrt\alpha\,U_{n-1}\!\bigl(\tfrac{x-(1+\alpha)}{2\sqrt\alpha}\bigr)
\;=\;
\sum_{i=0}^{n-1} \frac{1}{2}\alpha^{\lfloor \frac{i}{2}\rfloor-\frac{n}{2}}\;\mathcal P_{\,n-i}(x).
\end{equation}
\end{proposition}

\begin{proof}
{\bf Step 1.}  Take $A=0$ in  the diagrammatic expansion   and consider the classical Wishart ensemble as established in \eqref{eq:Pn_moments}, in the large $N$ limit one can prove that  the traces
$ N^{-1}\tr\mathcal P_n(X) $ for every  $n\ge1$
have mean zero and satisfy
\begin{equation}
\lim_{N\to\infty}
\mathrm{Cov}\bigl(\tr\mathcal P_{n_1}(X),\;\tr\mathcal P_{n_2}(X)\bigr)
=\delta_{n_1,n_2}\,\frac{2n_1}{\beta}\,\alpha^{n_1}.
\end{equation}
Therefore,   \(\{\alpha^{-n/2}\mathcal P_n\}\)  (up to normalization) consists of   the unique orthogonal family for the linear statistics of sample covariance matrices with parameter \(\alpha\).  

{\bf Step 2.}   It is a well-known fact (e.g.\ \cite[Theorem 1]{MR2320312}) that the centered and shifted Chebyshev polynomials
\begin{equation}
\widetilde T_n(x)
:=T_n\!\Bigl(\tfrac{x-(1+\alpha)}{2\sqrt\alpha}\Bigr)
\;-\;
\int T_n\!\Bigl(\tfrac{t-(1+\alpha)}{2\sqrt\alpha}\Bigr)\,d\mu_{\mathrm{MP}}(t)
\end{equation}
exactly form the orthogonal family.  By the uniqueness of orthogonal polynomials, one immediately obtains the relation \eqref{equ:Pn_to_Tn}.  

{\bf Step 3.}   Finally, the linear combination
\begin{equation}
U_n\!\bigl(\tfrac{x-(1+\alpha)}{2\sqrt\alpha}\bigr)
\;+\;
\sqrt\alpha\,U_{n-1}\!\bigl(\tfrac{x-(1+\alpha)}{2\sqrt\alpha}\bigr)
\end{equation}
is     orthogonal with respect to   \(\mu_{\mathrm{MP}}\), as established in   \cite[Prop.\,I.5.2]{feldheim2010universality}.  Direct moment-matching on $\mu_{\mathrm{MP}}$ consequently yields expansion \eqref{equ:Un_Pn}.
\end{proof}

Introduce a family of  polynomials 
\begin{equation}
Q_n(x) =U_n\!\bigl(\tfrac{x-(1+\alpha)}{2\sqrt\alpha}\bigr)
\;+\;
\sqrt\alpha\,U_{n-1}\!\bigl(\tfrac{x-(1+\alpha)}{2\sqrt\alpha}\bigr),
\end{equation}
then we can obtain a path expansion; see Proposition \ref{prop:Q_path_expansion}. Then, adapting the argument for inhomogeneous Wigner matrices yields
\begin{theorem}
    Given any fixed integers  $s\ge 1$ and  $t_i\ge 1$ for $1\le i\le s$, $1\le j\le t_i$, if  $\|A\|\le \sqrt{\alpha}+O(N^{-\frac{1}{3}})$ with $ \epsilon\le \alpha\le 1$ for some fixed $\epsilon>0$, and
    \begin{equation}
        \theta t_N\ll n_{i,j}\le  \tau N^{\frac{1}{3}}
    \end{equation}
    for any fixed constant $\tau$  such that   $t_1n_1+\cdots+t_kn_k$ is even, then   
    \begin{equation}
        \mathbb{E}\Big[\prod_{i=1}^{s}\tr \prod_{j=1}^{t_i} (\frac{1}{n_{i,j}}Q_{n_{i,j}}(\frac{X}{2}))\Big]=\mathbb{E}\Big[\prod_{i=1}^{s}\tr \prod_{j=1}^{t_i}(\frac{1}{n_{i,j}}Q_{n_{i,j}}(\frac{\widetilde{X}}{2}))\Big]+o(1).
    \end{equation}
\end{theorem}

\begin{theorem}[{\bf Universality for Wishart-type  IRM}]\label{thm:main_thm_Wishart} For $\alpha\in(0,1]$, with   $X_N=(H_N +A_N)(H_N +A_N)^*$ given  in Definition \ref{def:wishart}, assume that  $A_N$  has $q$ critical  eigenvalues $a_j = \sqrt{\alpha} + \tau_j N^{-1/3}$ for $j=1,\ldots,q$ with $\tau_j \in \mathbb{R}$, and $r-q$ bulk eigenvalues $a_j \in (-\sqrt{\alpha},\sqrt{\alpha})$ for $j=q+1,\ldots,r$, where $0 \leq q \leq r$ is fixed. If $(P_S, S)$ satisfies  the Short-to-Long Mixing    conditions \ref{itm:E1} and \ref{itm:E2} with $\theta t_N\ll N^{\frac{1}{3}}$, then
the first $k$ largest  eigenvalues of $X_N$ 
converge   in distribution  to those of  deformed GOE/GUE matrix. Similar results hold for first $k$  smallest eigenvalues   when $\alpha<1$.
\end{theorem}
\begin{proof}
The proof is almost the same as that of  Theorem \ref{thm:main_thm},  with one essential distinction:  the polynomial $Q_n$ tests the largest eigenvalues with weight $1+\sqrt{\alpha}$ and the smallest eigenvalues with weight $1-\sqrt{\alpha}$. For $\alpha=1$, only the largest eigenvalues contribute in the limit, rendering the smallest eigenvalues untestable in this case.
\end{proof}

Theorem~\ref{thm:main_thm_Wishart} is broadly applicable, much like Theorem~\ref{thm:main_thm} (see Section~\ref{appl} below).
  These include non-Hermitian matrices $H_N$ in \eqref{non-H} chosen as random band matrices with independent entries and general variance profiles, as well as  sparse random matrices with  structured variance profiles.

\section{Applications and further questions}\label{sec:section7}
 In this final section, we explore broader applications and open questions arising from our methods and main results. We begin by demonstrating immediate applications of Theorem~\ref{thm:main_thm} to classical random matrix ensembles, facilitated by the short-to-long mixing property introduced in Definition~\ref{mixingdef}. This mechanism establishes a fundamental connection between edge universality and random walk mixing on large graphs. We conclude by outlining several open questions naturally suggested by our framework.

\subsection{Applications} \label{appl}

We list direct applications of our main theorem to inhomogeneous variants of several canonical random matrix models, without a full treatment of the background  and   recent developments concerning these models.  Unless otherwise indicated, all notation follows Definition \ref{def:inhomo}, Definition~\ref{mixingdef} or  Theorem~\ref{thm:main_thm}.

\subsubsection{Generalized Wigner matrices}
Consider  generalized Wigner matrices with \texorpdfstring{\(N\)}{N}-dependent variance lower and upper bounds \(c_N\) and \(C_N\).   This allows \(c_N\) and \(C_N\) to grow or decay with \(N\), extending earlier results---see, e.g., \cite{MR3253704}---where \(c_N\) and \(C_N\) are assumed to be fixed constants.
\begin{theorem}[Generalized Wigner matrices]\label{thm:GW}
If   the variance profile of the random matrix  \(H\) defined in  Definition \ref{def:inhomo} 
satisfies  
\begin{equation}\label{eq:GW_N_dependent}
    \sum_{j}\sigma_{ij}^2=1, \quad 0< c_N < N\,\sigma_{ij}^2 < C_N< \infty,
\end{equation}
then under the constraint \begin{equation} \label{cC}
    \frac{C_N}{c_N}\, \theta\, \log N\;\ll\;N^{\frac{1}{3}},
\end{equation}
  the edge universality in Theorem~\ref{thm:main_thm} holds for \(H\).
\end{theorem}

\begin{proof}
Let $P = 
\bigl(\sigma_{ij}^2\bigr)_{i,j=1}^N$,  rewrite    \(P\) as
\begin{equation}
    P := c_N\,J + (1-c_N)\,T,
    \quad
    J = \frac1N \mathbf{1}\mathbf{1}^T,
\end{equation} where $\mathbf{1}$ denotes a column vector with all 1 entries and    \(T\) is  a  doubly stochastic matrix. Note that
$ TJ=JT=J$, 
by a standard coupling argument, the total-variation distance of \(P\) satisfies
\begin{equation}
    \bigl\|p_n(x,\cdot)-\pi\bigr\|_{\mathrm{TV}}
    \;\le\;(1-c_N)^n,
\end{equation}
where \(\pi\) is the uniform distribution on \(\{1,\dots,N\}\).  Hence for all \(x,y\),
\begin{equation}
    p_n(x,y)
    \;\le\;
    \frac{C_N}{N}
    \;\wedge\;
    \Bigl((1-c_N)^n + \frac{1}{N} \Bigr).
\end{equation}
Choose
\begin{equation}
    t_N = 100  \frac{\,C_N }{c_N} \log N,
\end{equation}
split  a sum over $n$ into two parts  according to whether \(n \le ( 100\log N)/{c_N}\) or \(n \ge ( 100\log N)/{c_N}\), and  we  then arrive at  
\begin{equation}
    \sum_{n=1}^{t_N} p_n(x,y)
    \;\le\;
    \frac{100\,C_N \log N}{c_N\,N}
    + \frac{2\,t_N}{N}
    \;\le\;
    \frac{3\,t_N}{N},
\end{equation} 
This verifies condition~\ref{itm:B1}.

Similarly, for \(n \ge t_N\), we have
\begin{equation}
    \bigl| N\,p_n(x,y) - 1 \bigr|
    \;\le\; N(1-c_N)^n
    \;\le\; N^{-99},
\end{equation}
which establishes condition~\ref{itm:B2}.

By Theorem~\ref{thm:main_thm}, any generalized Wigner matrix that satisfies  \eqref{eq:GW_N_dependent} and has  \(\theta\)-sub-Gaussian entries, with $\theta t_N\ll N^{{1}/{3}}$, exhibits Tracy-Widom edge universality, provided the constraint \eqref{cC}.
\end{proof}

\subsubsection{Random band matrices}

Given a box lattice 
$ 
\Lambda_L \;=\; (\mathbb Z / L\mathbb Z)^d$ with  a large integer $L>0$, 
  and   a density  function   \(f:\mathbb{R}^d\to\mathbb{R}\),  we  define a  \emph{variance profile}
\begin{equation} \label{rbmprofile}
\sigma_{xy}^2
\;=\;
\frac{1}{M}\sum_{n\in\mathbb{Z}^d}
f\Bigl(\tfrac{x-y+nL}{W}\Bigr),
\qquad x,y\in\Lambda_L,
\end{equation}
where $W>0$ is the bandwidth parameter and the discrete normalization 
\begin{equation}
M \;=\;\sum_{n\in\mathbb{Z}^d}f\bigl(nW^{-1}\bigr)
\end{equation}
 ensures \(\sum_y\sigma_{xy}^2=1\).  The  resulting matrix  \(H\) is  a  random band matrix  of size $N:=L^d$ on $d$-dimensional  periodic lattice. 
Applying the local central limit theorem and upper bound for the transition probability, as established in e.g. \cite{gu2024local}, we can easily verify condition \ref{itm:B1} and \ref{itm:B2}.

We impose the following assumptions on $f$.
\begin{definition}\label{ass:ft}
There exist constants $ \alpha \in (0, 2]$, $c_0 > 0$, $\varepsilon > 0$, and $0 < \rho < 1$ such that the following assumptions are satisfied:
\begin{enumerate}
  \item (normalization) $f$ is a probabiltiy density
  \begin{equation}\label{eq:prob-density}
    f(x)\ge 0, \quad \int_{\mathbb R^d} f(u)\,du \;=\; 1;
  \end{equation}
  \item (small-frequency control) for all $|\xi|\le \varepsilon$,
  \begin{equation}\label{eq:small-freq}
    |\widehat f(\xi)| \;\le\; \exp\!\bigl(-c_0 |\xi|^\alpha \bigr);
  \end{equation}
  \item (high-frequency contraction) for all $|\xi| > \varepsilon$,
  \begin{equation}\label{eq:high-freq}
    |\widehat f(\xi)| \;\le\; \rho < 1;
  \end{equation}
  \item (regularity/decay) $f$ is a continuous function and for some   constants $T, K > d$ and $C_T, C_{K} > 1$,   
  \begin{equation}\label{equ:D1}
      f(x)\le \frac{C_T}{(1+|x|)^{T}},
    \qquad |\widehat f(\xi)| \;\le\; \frac{C_K}{(1+|\xi|)^{K}} 
      \qquad  x,\xi \in \mathbb R^d.
  \end{equation}
\end{enumerate}
\end{definition}

\begin{theorem}[Random band matrices]\label{thm:RBM}
Let \(H\) be a random band matrix with variance profile \eqref{rbmprofile}, whose entries are symmetric and uniformly sub-Gaussian. 
Suppose that the density  function $f$ satisfies the assumptions in Definition~\ref{ass:ft}, and that 
\[
W^{-K}\ll N^{-\frac{4}{3}}.
\]
Then edge universality holds under the following bandwidth condition
\begin{equation}\label{equ:bandwidth}
W \gg
\begin{cases}
L^{\,1-\tfrac{d}{3\alpha}}, & d<\alpha,\\[0.3em]
L^{\frac{2}{3}}\log L, & d=\alpha,\\[0.3em]
L^{\frac{2}{3}}, & d>\alpha.
\end{cases}
\end{equation}
\end{theorem}

\begin{proof}
Fix a time scale $\left( {L}/{W}\right)^{\alpha}\ll t_N\ll N^{1/3}$.  
For $t_N \le n = O(N^{1/3})$, by Theorem \ref{thm:clt-upper} and the estimate $W^{-K}\ll N^{-4/3}$, we obtain that for $n\gg (L/W)^{\alpha}$,  
\begin{equation}
    |Np_n(0,x)-1|=o(1),
\end{equation}
which verifies condition~\ref{itm:B2}.  

We now turn to condition~\ref{itm:B1}. Applying \eqref{equ:band_mixing} gives
\begin{equation}
\sum_{n=1}^{t_N} p_n(x,y)
\le C \sum_{n=1}^{t_N}\Bigl(W^{-d} n^{-\frac{d}{\alpha}}+N^{-1}\Bigr).
\end{equation}
The right-hand side can be estimated as
\begin{equation}
\sum_{n=1}^{t_N} p_n(x,y)
\le \frac{C t_N}{N} +
\begin{cases}
C' W^{-d}\, t_N^{\,1 - \frac{d}{\alpha}}, & d < \alpha, \\[0.3em]
C' W^{-d}\, \log t_N, & d = \alpha, \\[0.3em]
C' W^{-d}, & d > \alpha,
\end{cases}
\end{equation}
for some $N$-independent constants $C, C' > 0$.  

Finally, for the the bandwidth $W$ satisfying \eqref{equ:bandwidth}, the condition~\ref{itm:B1} is fulfilled. Thus by applying Theorem \ref{thm:main_thm}, we finish the proof.
\end{proof}

\noindent\textbf{Power-law random band matrices.} If the profile $f$ has a finite second moment (so that the small-frequency exponent is $\alpha=2$ in \eqref{eq:small-freq}), the classical central limit theorem applies and the associated random walk is diffusive with a Gaussian fixed point. 
By contrast, when $f$ exhibits heavy tails with decay $|x|^{-d-\alpha}$ for some $0<\alpha<2$, the CLT breaks down and the random walk converges to an $\alpha$-stable process. 
In that heavy-tailed regime the mixing is faster and the critical bandwidth changes accordingly; the precise power-law thresholds are given in \eqref{equ:bandwidth} and reflect the $\alpha$-stable scaling. 
Theorems \ref{thm:main_thm} and \ref{thm:RBM} treat both the Gaussian attraction domain ($\alpha=2$) and the stable attraction domain ($0<\alpha<2$) in a unified way.

\noindent\textbf{Non-periodic and structured variants.} 
It is straightforward to check that, under the same assumptions, conditions~\ref{itm:B1} and~\ref{itm:B2} also hold in the non-periodic setting. 
However, in order to preserve the doubly stochastic structure of the variance profile, the transition probability should be chosen to correspond to a reflecting random walk. 
This yields a partial answer to the problem stated in~\cite[Section~6, Remark~V]{liu2023edge} in the super-critical regime.

\subsubsection{Random block matrices}

Define a block tridiagonal model  that is also  know as  the block   Wegner orbital  model \cite{wegner1979disordered,MR3915294}
\begin{equation}
X = \sqrt{1-\lambda}\,H + \sqrt{\lambda}\,\Lambda,
\qquad \lambda\in[0,1],
\end{equation}
where \begin{equation}
H=\mathrm{diag}(H_1,\dots,H_D), 
\quad 
\Lambda=\begin{pmatrix}
0      & A_1    & 0      & \cdots & A_D^* \\
A_1^*  & 0      & A_2    & \cdots & 0     \\
\vdots & \ddots & \ddots & \ddots & \vdots\\
A_D    & 0      & \cdots & A_{D-1}^* & 0
\end{pmatrix}.
\end{equation}
Here every   \(H_i\)   is an   \(M\times M\) Wigner matrix with   $\theta$-sub-Gaussian  entries and every   \(A_i\) is an independent non-Hermitian \(M\times M\) matrix with i.i.d. $\theta$-sub-Gaussian  entries, and all matrices are independent.   This model has been studied
in\cite{MR3915294,fan2025localization,stone2025random,yang2025delocalization,khang2025localization}.  

An interesting problem is the phase transition for eigenvalue statistics: when  $\lambda=0$, the eigenvalue statistics come from $D$ independent matrices; for $\lambda={1}/{2}$, the model is a block random band matrix  in dimension  $d=1$; for    $\lambda=1$, it is  a block bi-diagonal model.
\begin{theorem}[Random block matrices]\label{thm:block}
If   there exist $\theta$ and $t_N$ such that  \(\theta t_N\ll N^{1/3}\) and    
\begin{equation}\label{eq:block-cond1}
\sqrt{t_N\,\lambda}\;\gg\;D,
\quad\text{and}\quad
(1-\lambda)\,t_N\;\gg\;1,
\end{equation}
then  Assumptions~\ref{itm:B1}-\ref{itm:B2} hold,  implying that   edge universality holds for $X$.
\end{theorem}

\begin{proof}
Identify the block index set with \(\mathbb{Z}/D\mathbb{Z}\), and let \(q_n(a,b)\) be the \(n\)-step transition probability of the Markov chain that  
 stays at the same block with probability $1-\lambda$ and 
 {moves to a neighbor} with probability $ \lambda/2$.
Then for \(x\) in block \(a\) and \(y\) in block \(b\),
\begin{equation}\label{eq:block-pn}
p_n(x,y)
= \frac{1}{M}\,q_n(a,b).
\end{equation}
Classical mixing on \(\mathbb{Z}/D\mathbb{Z}\) implies that if \(n\ge t_N\) satisfies the condition in \eqref{eq:block-cond1}, the chain mixes rapidly, yielding
\begin{equation}\label{eq:block-B2}
\bigl|N\,p_n(x,y)-1\bigr|\;\ll\;1,
\end{equation}
so Assumption~\ref{itm:B2} holds.  Moreover, summing \eqref{eq:block-pn} over \(1\le n\le t_N\) gives
\begin{equation}\label{eq:block-B1}
\sum_{n=1}^{t_N}p_n(x,y)
\;\le\;
\sum_{n=1}^{t_N}\frac{C}{M\sqrt{1+n\lambda}}
\;\le\;
\frac{\gamma t_N}{N},
\end{equation} 
verifying Assumption~\ref{itm:B1} for some constant $\gamma>0$.  The theorem then follows from Theorem~\ref{thm:main_thm}.
\end{proof}
Similar conclusions hold in higher dimensions, which we omit for brevity.  
We also note, following \cite{fan2025localization}, that when \(D\) is fixed, a sharp phase transition in the edge statistics occurs at \(\|A\|_{\mathrm{HS}} \sim N^{1/3}\),
which is equivalent to \(\lambda \sim N^{-1/3}\).  
Therefore, the condition~\eqref{eq:block-cond1} in Theorem~\ref{thm:block} is optimal for fixed \(D\).  
Moreover, in contrast to the settings of \cite{MR3915294,fan2025localization,stone2025random,yang2025delocalization,khang2025localization}, the matrix entries in our model are neither restricted to be Hermitian nor Gaussian, and may even exhibit heavy-tailed behavior and more general variance profile.

A more refined analysis could potentially characterize the complete localization–delocalization transition at the spectral edge, but we leave such an investigation to future work.

\subsubsection{Inhomogeneous sparse  random matrices}
Let \(G=(G_{ij})\) be the adjacency matrix of a weighted inhomogeneous Erd\H{o}s-R\'enyi graph on \(N\) vertices, with
\begin{equation}
  G_{ij} \;=\; \mathrm{Bern}(p_{ij})\,\sqrt{w_{ij}},
  \qquad
  \sum_{j}p_{ij}w_{ij}=d, \quad\forall i.
\end{equation}
Define the normalized Hadamard product
\begin{equation}
  H \;=\;\frac{1}{\sqrt{d}}\;G\circ W,
\end{equation}
where \(W\) is a Wigner matrix with symmetric and  \(\theta_{ij}\)-sub-Gaussian entries. Set 
\begin{equation}
\theta \;=\;\max_{i,j}\frac{\theta_{ij}}{p_{ij}},
\end{equation}
then 
\begin{theorem}[Sparse IRM]\label{thm:sparse}
If 
\(\Sigma_N: = (\sqrt{p_{ij}w_{ij}/d})\)  
satisfies Assumptions~\ref{itm:B1}-\ref{itm:B2} with  $\theta\,t_N \;\ll\; N^{1/3}$, then  
 the edge universality holds for \(H\).
\end{theorem}
Previous results \cite{MR2964770,MR3098073} show that in the homogeneous case, where $w_{ij}=1$ and \(p_{ij}=\theta^{-1}\), the largest eigenvalue obeys Tracy-Widom whenever
$  
  \theta^{-1}\;\gg\;N^{-1/3}.
$ 
In our framework, one verifies that Assumptions~\ref{itm:B1}-\ref{itm:B2} hold with $t_N=1$ thus as long as
 $ 
  \theta \;\ll\; N^{1/3},$ 
Theorem~\ref{thm:sparse} yields edge universality in this sparse regime.  Moreover, Lee-Schnelli \cite{lee2018} proved that for homogeneous Erd\H{o}s-R\'enyi graph if
 $
  \theta \;\gtrsim\;N^{1/3},
$ 
then a deterministic shift appears at the spectral edge, showing this threshold is sharp. See e.g. \cite{huang2022edge,he2023edge} and references therein for recent results and progress for edge universality of homogeneous sparse Erd\H{o}s-R\'enyi graph.

\subsubsection{IRM with heavy-tailed entries}

\begin{theorem}[IRM  with heavy-tailed entries]\label{thm:heavy} If the IRM matrix  $X$  in Definition~\ref{def:inhomo}  satisfies  Assumptions~\ref{itm:B1}–\ref{itm:B2} with $t_N \ll N^{\tfrac13 - \zeta}$ 
 for some $\zeta\in (0, \tfrac13)$, and 
the  entries of $W$ have uniform moments   with  some fixed $\epsilon_0>0$, 
\begin{equation} \label{uniformM}
\sup_{i,j}\mathbb{E}\bigl[\lvert W_{ij}\rvert^{\tfrac{4+\epsilon_0}{\zeta}}\bigr]\le C,
\end{equation}
then the largest eigenvalue of \(X\) 
 converges in distribution to the Tracy–Widom law.
\end{theorem}
\begin{proof}
    In this case, it suffices to take \(\theta \ll N^\zeta\) in order to satisfy the moment condition in Theorem \ref{thm:main_thm}. Consider a Wigner-type matrix \(W = (W_{ij})\) whose entries have the uniform moments,  \eqref{uniformM}
    and we introduce  a truncated matrix ${W}^{<}$ that has entries defined  by 
\begin{equation}
{W}^{<}_{ij} := W_{ij} \cdot \mathbf{1}(|W_{ij}| < N^{\frac{\zeta}{2}}).
\end{equation}
By the Markov inequality,  we have
\begin{equation}
    \mathbb{P}\big(W_{ij}\ne W_{ij}^<\big)\le \frac{1}{N^{2+\frac{\epsilon}{2}}} \mathbb{E}\big[|W_{ij}|^{\frac{4+\epsilon_0}{\zeta}}\big]\le CN^{-2-\frac{\epsilon_0}{2}},
\end{equation}
  which implies  that with high probability
\begin{equation}
\mathbb{P}\left(W_{ij} = {W}^{<}_{ij}, \quad \forall \,  1 \le i,j \le N\right) \ge 1 - C N^{-\frac{\epsilon_0}{2}}.
\end{equation}
Therefore, the truncation error becomes asymptotically negligible as \(N \to \infty\). Indeed, since \({W}^{<}_{ij}\) is \(N^{\zeta}\)-sub-Gaussian, we obtain
\begin{equation}
\begin{aligned}
\mathbb{E}\big[|W_{ij} - {W}^{<}_{ij}|^2\big]
&= \mathbb{E}\big[|W_{ij}|^2 \cdot \mathbf{1}(|W_{ij}| \ge N^\frac{\zeta}{2})\big] 
\le \sqrt{\mathbb{E}[|W_{ij}|^4]} \cdot \sqrt{\mathbb{P}(|W_{ij}| \ge N^\frac{\zeta}{2})} \\
&\le \sqrt{\mathbb{E}[|W_{ij}|^4]} \cdot \sqrt{\frac{1}{N}\mathbb{E}[|W_{ij}|^{\frac{2}{\zeta}}]} = O(N^{-\frac{1}{2}}).
\end{aligned}
\end{equation} 
Furthermore,  we have
\begin{equation}
    \mathbb{E}\big[|{W}^{<}_{ij}|^2\big]=1+O(N^{-\frac{1}{2}}).
\end{equation}

Such errors do not affect the asymptotic behavior of diagram functions, as \(n N^{-\frac{1}{2}} \to 0\) when \(n \sim N^{\frac{1}{3}}\). Consequently, choosing \(\theta_N = N^{\zeta}\) in Theorem~\ref{thm:main_thm} preserves edge universality.
\end{proof}

Edge universality holds for heavy-tailed Wigner and sample-covariance matrices   when entries have  only polynomial-decay tails, 
as shown in \cite{MR2191882,MR2899796,MR3161313,MR3809475}  
and references therein.

As a concrete application of Theorem \ref{thm:heavy}, for Wigner matrices whose entries satisfy \(\mathbb{E}[|W_{ij}|^{12+\epsilon}] < \infty\), edge universality still holds true. This result is somewhat weaker than those in \cite{MR3161313,MR3809475}. The distinction lies in the regime of interest: if the distributions of \(W_{ij}\) vary with \(N\), and their fourth moments diverge as \(\mathbb{E}[|W_{ij}|^4] \geq N^{\frac{1}{3}}\), then a non-universal correction emerges near the spectral edge (see \cite{lee2018} for examples). In contrast, for fixed heavy-tailed distributions, the truncated matrix has much smaller fourth moments, well below the \(N^{1/3}\) threshold.  For critical tail behavior in  random band or sparse   matrices, we refer to \cite{MR4546632,han2024deformed}.

\subsubsection{Weighted \texorpdfstring{$d$}{d}-regular graphs}
Let $G_N$ be the adjacency matrix of a $d$-regular graph on $N$ vertices, and can be treated  $G_N$ as the variance-profile matrix $\Sigma_N$.  If all nontrivial eigenvalues satisfy $|\lambda_i(G_N)|=o(d)$ for $i=2,\ldots, N$, then Assumptions~\ref{itm:B1}-\ref{itm:B2} hold whenever
  $d\gg N^{2/3}.$
This  ensures edge universality for  $d$-regular  weighted matrix  $H$. 

\smallskip\noindent
In particular, equip each edge of $G_N$ with an independent Rademacher sign $\sigma_{ij}\in\{\pm1\}$, and denote the signed adjacency by $\sigma\circ G_N$.  Equivalently, this is the “new’’ part of the adjacency of the random 2-lift $\widetilde G_N$.  By \cite[Lemma 3.1]{MR2279667},
\begin{equation}
  \mathrm{Spec}(\widetilde G_N)
  = \mathrm{Spec}(G_N)\;\cup\;\mathrm{Spec}(\sigma\circ G_N).
\end{equation}
 An immediate corollary of   Theorem~\ref{thm:main_thm}  is as follows.

\begin{corollary}\label{coro:sign_d_regular}
  Let $G_N$ be the adjacency matrix of a $d$-regular graph on $N$ vertices,   assume $d\gg N^{2/3+\tau}$ for some fixed $\tau>0$ and $|\lambda_2(G_N)|,|\lambda_N(G_N)|=o(d)$.  Then edge universality holds for the extreme eigenvalues of $\sigma\circ G_N$. In particular, the rescaled largest eigenvalue converges weakly to $\mathrm{TW}_1$ which  has  the Tracy-Widom $F_1$ distribution, that is, 
  \begin{equation}
    N^{\frac{2}{3}}\Bigl(\tfrac{\lambda_1(\sigma\circ G_N)}{\sqrt{d}} - 2\Bigr)  \longrightarrow   \mathrm{TW}_1.
  \end{equation}
\end{corollary}

\smallskip\noindent
On the other hand, when  $G_N$ is drawn uniformly at random,   \cite{MR4777082} shows for $d\gg N^{2/3+\tau}$,
\begin{align}
  N^{\frac{2}{3}}\Bigl(\tfrac{\lambda_2(G_N)+d/N}{\sqrt{d(N-d)/N}} - 2\Bigr)
  &\;\xrightarrow[N\to\infty]{\mathrm{D}}\;\mathrm{TW}_1, \label{eq:second-eig-TW}
\end{align}
and with high probability
\begin{align}
  \lambda_2(G_N)
  = \bigl(2\sqrt{d(N-d)/N}-d/N\bigr)&\bigl(1+O(N^{-2/3+\varepsilon})\bigr)
   \label{eq:second-eig-bound}
\end{align}
Hence, in probability,
\begin{equation}\label{equ:lambda<lambda}
  \lambda_2(G_N) < \lambda_1(\sigma\circ G_N).
\end{equation}

This demonstrates that for  the random 2-lift $\widetilde{G}_N$, the second eigenvalue of its adjacency matrix $\lambda_2(A_{\widetilde{G}_N})$ is governed by the "new" spectrum. Combining these observations yields:

\begin{theorem}[Random 2-lifts]\label{thm:2lift}
  Let $\widetilde G_N$ be the  random 2-lift of a uniformly random $d$-regular graph with $d\gg N^{2/3+\tau}$, then
  \begin{equation}
    N^{\frac{2}{3}}\Bigl(\tfrac{\lambda_2(\widetilde G_N)}{\sqrt{d-1}} - 2\Bigr)
    \;\xrightarrow[N\to\infty]{\mathrm{D}}\;
    \mathrm{TW}_1.
  \end{equation}
\end{theorem}

\begin{corollary}[Repeated 2-lifts]\label{cor:k-lift-TW-max}
For any given integer  $k\ge1$, define iteratively $G_N^{(i)}$ as a random 2-lift of $G_N^{(i-1)}$ for $i=1,\dots,k$, where    $G_N^{(0)}$ is a uniformly random $d$-regular graph.   If $d\gg N^{2/3+\tau}$, then
\begin{equation}
  (2^{k-1}N)^{\frac{2}{3}}\bigl(\tfrac{\lambda_2(G_N^{(k)})}{\sqrt{d}}-2\bigr)
  \;\xrightarrow[N\to\infty]{\mathrm{D}}\;
  \max_{1\le i\le k} 2^{\frac{2(k-i)}{3}}\mathcal{X}_i,
\end{equation} 
where $\mathcal{X}_i$ are i.i.d. random variable with the same  Tracy-Widom distribution $F_1$.
\end{corollary}
\begin{proof}
We proceed by induction on the number of 2-lift steps \(k\). For \(k=1\), the claim follows directly from Corollary~\ref{coro:sign_d_regular} and \eqref{equ:lambda<lambda}. 

Assume now that the statement holds for \(k-1\) lifts. Consider \(G_N^{(k)}\) formed by signed matrix \(\sigma^{(k)} \circ G_N^{(k-1)}\). Under the assumption \(d \gg N^{2/3+\tau}\), and with high probability that all nontrivial eigenvalues of \(G_N^{(k-1)}\) are \(o(d)\), Corollary~\ref{coro:sign_d_regular} applies again, showing that the largest new eigenvalue (after rescaling) converges in distribution to \(\mathrm{TW}_1\), independently of $G_N^{(k-1)}$.

At step \(i\), the graph \(G_N^{(i)}\) has \(2^i N\) vertices. Therefore, the appropriate scaling for Tracy–Widom fluctuations at this step is \((2^i N)^{2/3} = 2^{2i/3} N^{2/3}\). Rewriting this in terms of the final graph size \(2^k N\), the eigenvalue arising from the \(i\)-th lift should be rescaled by \(2^{2(k-i)/3} N^{2/3}\) to match the normalization at scale \(2^k N\).

The largest new eigenvalues from each lift are asymptotically independent. The second largest eigenvalue of \(G_N^{(k)}\) is thus asymptotically determined by the maximum among these \(k\) independent Tracy–Widom fluctuations, each scaled by \(2^{2(k-i)/3}\). This proves the claim.
\end{proof}

\subsubsection{Dimension-free deviation inequality}
In \cite{brailovskaya2024extremal}, 
 both small and large deviation inequalities were established for inhomogeneous Gaussian matrices under minimal sparsity assumptions.  As a direct application of Corollary~\ref{coro:tail_bound}, in the undeformed case we recover Tracy-Widom fluctuations in the small-deviation regime, provided that assumptions \ref{itm:B1} and \ref{itm:B2} hold.

If one is only interested in the \(O(N^{-2/3})\) fluctuation scale,  assumption \ref{itm:B2} may be weakened to
\begin{equation}
    p_n(x,y)\;\le\; C\,N^{-1}, 
    \quad \forall \, n\ge t_N.
\end{equation}
When  $t_N\ll N^{^{{1}/{3}}}$, we have 
\begin{equation}
    \Big(\Big(\frac{C n^2 t_N}{N}\Big)^{[\frac{n}{t_N}]}\land 1\Big)\le 1
\end{equation}
in Theorem \ref{prop:U_asy} \ref{item:U_upper}. Following the argument in Corollary \ref{coro:tail_bound}, we can remove the restriction $x^2 t_N N^{-\frac{1}{3}}< c_0$ in Corollary~\ref{coro:tail_bound} with a worser small-deviation bound for all $x\ge 0$,
\begin{equation}
    \mathbb{P}\Bigl(\|\tfrac12X\|_{\mathrm{op}}>1+xN^{-\frac{2}{3}}\Bigr)
    \;\le\;
    C\,e^{-c\,x^{\frac{3}{4}}}.
\end{equation}
This bound may be sharpened by a more detailed analysis of Theorem~\ref{prop:U_asy}.  Notably, in contrast to \cite{brailovskaya2024extremal}, there is no dimension factor \(N\) in the tail probability bound. Such tail decay could extend to $\theta$-sub-Gaussian case by carefully analyzing the bound in Lemma \ref{lem:D_n=0}.   

In summary, we have 
\begin{theorem}[Deviation inequalities]\label{thm:deviation}
Under Assumptions  ~\ref{itm:B1}-\ref{itm:B2}, 
there are   constants $C$ and $c$    only depending on  $\gamma$ and $\theta$,  such that for $t_N\ll N^{1/3}$, 
\begin{equation}
\mathbb{P}\Bigl(\|\tfrac12 X\|_{\mathrm{op}}>1+xN^{-\frac{2}{3}}\Bigr)\le C\,e^{-c\,x^{\frac{3}{4}}}.
\end{equation}
\end{theorem}
\subsubsection{GOE-GUE transition for 
IRM}

We can also extend  an interpolating ensemble   with unequal real and imaginary parts, which has been studied  by Feldheim and Sodin    \cite[Section V]{feldheim2010universality},  to the inhomogeneous setting.  Indeed,  replace the GOE/GUE   matrix  $W$ in our definition \ref{def:inhomo}   $H=\Sigma_N\circ W$ with   the general Gaussian ensemble,
\begin{equation}\label{equ:mix_gaussian}
    (W_N)_{ij}\overset{d}{=}\begin{cases}
        \mathcal{N}(0,1/(1+\alpha^2))+i\mathcal{N}(0,\alpha^2/(1+\alpha^2)),& i\ne j\\
        \mathcal{N}(0,2/(1+\alpha^2)),&i=j,
    \end{cases}
\end{equation}
where the real parts are symmetry and the imaginary parts are anti-symmetric. Note that   for $i\ne j$,
\begin{equation}
    \mathbb{E}[W^2_{ij}]=\frac{1-\alpha^2}{1+\alpha^2},~~\mathbb{E}[|W_{ij}|^2]=1,
\end{equation}
we immediately obtain GOE for T  $\alpha=0$, $\alpha=1$ for GUE, and an anti-Symmetric GOE for $\alpha=\infty$(  see \cite[Chapter 13]{Mehta2004} for the third ensemble).
\begin{theorem}
    Let $H_N$ be an IRM matrix   defined in Definition \ref{def:inhomo} but with real/complex Gaussian entries replaced by \eqref{equ:mix_gaussian},  the following limits for the largest eigenvalue $\lambda_1$ hold.
\begin{enumerate}
    \item[(1)] If $0\le \alpha\ll N^{-\frac{1}{6}}$, then 
    \begin{equation}
        N^{\frac{2}{3}}(\lambda_1-2)\overset{D}{\rightarrow} \mathrm{TW}_1.
    \end{equation}
    \item[(2)] If $N^{-\frac{1}{6}}\ll \alpha\le \infty$, then 
    \begin{equation}
        N^{\frac{2}{3}}(\lambda_1-2)\overset{D}{\rightarrow}  \mathrm{TW}_2.
    \end{equation}
\end{enumerate}
\end{theorem}
\begin{proof}
The argument follows essentially the same approach as in \cite[Section V]{feldheim2010universality}, with the sole modification being the adaptation to our inhomogeneous setting.
\end{proof}
\subsubsection{IRM with quaternion entries}

  Our methods naturally extend to   matrix models with quaternion entries ($\beta=4$). In particular, the edge eigenvalue analysis developed by Feldheim and Sodin for quaternion Wigner and Wishart ensembles \cite[Section V]{feldheim2010universality} can be directly adapted to the inhomogeneous edge setting with only minor modifications.

\subsection{Further questions}
 
We conclude this paper by highlighting    several basic questions.  

\begin{enumerate}[label=(\roman*)]
    \item \textbf{Edge universality beyond Assumption~\ref{itm:B1}.}  
    Assumption~\ref{itm:B1} is both essential and technical. 
    Its violation can introduce non-universal corrections at the spectral edge when analyzing Tracy-Widom fluctuations. Nevertheless, the core edge universality may remain intact, provided these corrections are incorporated; see \cite{liu2023edge} for an example in the context of band matrices.
    
    \item \textbf{Moment condition and edge corrections.}    In Theorem \ref{thm:main_thm}, we require the moment condition $\theta t_N \ll N^{1/3}$ in order to establish edge universality. Violating this condition may also lead to non-universal corrections at the spectral edge. However, as in the previous case, edge universality may be still true after such corrections. This phenomenon has been observed in the setting of sparse random matrices; see e.g.\ \cite{huang2022edge}. We also refer the reader to \cite{aggarwal2022mobility} for   edge statistics results in the heavy-tailed setting. 

 For simplicity, 
we also assume all odd moments of the matrix entries vanish. However, extending Theorem~\ref{thm:main_thm} to non-symmetric settings or to sharp moment assumptions remains an important  open question.

\item \textbf{Variance profile without the doubly stochastic condition.}
Assumption~\ref{itm:A2} imposes a doubly stochastic structure on the variance profile, which guarantees that the limiting empirical spectral distribution is the semicircle law.  This restriction, however, is somewhat artificial and may be only technical. Removing Assumption~\ref{itm:A2} is a very interesting problem; nonetheless, local universality of spectral statistics may still hold without it.  For related results in matrix models that do not possess a doubly stochastic structure but require comparable variance profile, see e.g.\ \cite{MR3719056} for bulk universality, \cite{MR4089499} for edge universality and \cite{ducatez2024large} for a large deviation principle.

    \item \textbf{Delocalization of edge eigenvectors.}  
    This paper focuses solely on eigenvalue statistics. A natural further question is whether the eigenvectors associated with the edge eigenvalues are delocalized under Assumptions~\ref{itm:B1} and~\ref{itm:B2}. We refer to \cite{yang2025delocalization2} for recent progress on edge eigenvector delocalization for band matrix models.
    
    \item \textbf{Inhomogeneous non-Hermitian random matrices.} In Theorem \ref{thm:main_thm_Wishart}, we establish the universality of the largest singular values for inhomogeneous non-Hermitian matrices. However, the precise conditions under which eigenvalue universality holds for inhomogeneous non-Hermitian matrices remain unclear; we refer to \cite{MR4221653} for  edge universality of non-Hermitian random matrices with  i.i.d. entries and to \cite{liu2024} for    phase transition of eigenvalues of  deformed Ginibre ensembles.

\end{enumerate}

\section*{Acknowledgments} 

We thank Jiaqi Fan and Fan Yang for valuable discussions and for kindly explaining their work~\cite{stone2025random,fan2025localization} to us. We also thank Ruohan Geng for valuable conversations,  particularly  on the diagram expansion of the deformed Tracy-Widom law. We appreciate Jiaoyang Huang for valuable discussions on random lifts of $d$-regular graphs. We are thankful to Ramon van Handel for valuable discussions on the fluctuation scale at the spectral edge during Guangyi's visit to Princeton. We are grateful to Yuanyuan Xu for valuable discussions on the challenge 
 of edge statistics for non-mean-field random matrices. This work was supported by the National Natural Science Foundation of China under Grants \#12371157 and \#12090012.

\let\oldthebibliography\thebibliography
\let\endoldthebibliography\endthebibliography
\renewenvironment{thebibliography}[1]{
  \begin{oldthebibliography}{#1}
    \setlength{\itemsep}{0.5em}
    \setlength{\parskip}{0em}
}
{
  \end{oldthebibliography}
}

\bibliographystyle{alpha}
\begin{spacing}{0}
\small
\bibliography{Reference}
\end{spacing}

\appendix

\section{Chebyshev expansion in  deformed Hermitian case} \label{Chebyshevdeform}

\begin{proposition}[Deformed path expansion]\label{prop:path_expansion_H_plus_A}
Let \(U_n(x)\) be the Chebyshev polynomial of the second kind, 
and let \(V_n,\Phi_2,\Phi_3,\underline{\Phi_2V_n},\underline{\Phi_3V_n}\) be as in Definition~\ref{def:nonbacktracking}.  Introduce
\begin{equation}
\underline{\Phi_A V_n}\;\equiv\;\sum_{i=1}^n A^i\,V_{n-i},
\end{equation}
recording that each insertion of \(A\) of total multiplicity \(i\) carries no backtracking constraint, then
\begin{equation}\label{eq:main_path_expansion_HA}
U_n\bigl((H+A)/2\bigr)
=
\sum_{k\ge0}\;\sum_{\alpha\in\{2,3,A\}^k}\;\sum_{\substack{\ell_0+\cdots+\ell_k=n}} 
V_{\ell_0}
\;\underline{\Phi_{\alpha_1} V_{\ell_1}}\;\cdots\;\underline{\Phi_{\alpha_k} V_{\ell_k}},
\end{equation}
where each symbol \(\alpha_i\in\{2,3,A\}\) indicates an insertion of \(\Phi_2,\Phi_3,\Phi_A\) respectively, and the summation   runs over all nonnegative integers \(\ell_i\ge0\).
\end{proposition}

\begin{proof}
Set
\begin{equation}
D_n \;=\; U_n\bigl((H+A)/2\bigr)
-\sum_{k,\alpha,\ell_i}
V_{\ell_0}\,\underline{\Phi_{\alpha_1}V_{\ell_1}}\cdots\underline{\Phi_{\alpha_k}V_{\ell_k}}, 
\end{equation}
we will verify \(D_n=0\) by induction on \(n\).

\medskip
\noindent\textbf{Base step.}  For \(n=0\), both sides equal \(U_0=I\).  For \(n=1\),
\begin{equation}
U_1((H+A)/2)=H+A
=V_1+\underline{\Phi_A V_0},
\end{equation}
since \(V_1=H\) and \(\underline{\Phi_A V_0}=A\).  Hence \(D_1=0.\)

\medskip
\noindent\textbf{Induction step.}  Assume \(D_m=0\) for all \(m<n\).  Start from  the recursion relation for Chebyshev polynomials,
\begin{equation}
U_n((H+A)/2)
={(H+A)}\,U_{n-1}((H+A)/2)
- U_{n-2}((H+A)/2),
\end{equation}
apply the induction hypothesis to \(U_{n-1},U_{n-2}\), split
\begin{equation}
(H+A)U_{n-1}
=H U_{n-1} + A U_{n-1},
\end{equation}
and we then proceed in  two cases. 

\textbf{(i) The \(H\)-term.}
By induction hypothesis,
\begin{equation}
U_{n-1}((H+A)/2)
=\sum_{k,\alpha,\ell_i}
V_{\ell_0}\,\underline{\Phi_{\alpha_1}V_{\ell_1}}\cdots\underline{\Phi_{\alpha_k}V_{\ell_k}},
\end{equation}
hence
\begin{equation}
H U_{n-1}
=\sum_{k,\alpha,\ell_i}(H V_{\ell_0})\,\underline{\Phi_{\alpha_1}V_{\ell_1}}\cdots
\underline{\Phi_{\alpha_k}V_{\ell_k}}.
\end{equation}
Using \cite[Lemma 6.1]{EK11Quantum} (note the are slight different in notation $\underline{\Phi_3V_{\ell_0-2}}$),
\begin{equation}
H V_{\ell_0}
=V_{\ell_0+1}+V_{\ell_0-1}+\Phi_2V_{\ell_0-1}+\underline{\Phi_3V_{\ell_0+1}},
\end{equation}
(with the understanding that terms with negative index vanish), one splits this into four sums.  Reindexing \(\ell_0+1\to\ell_0\) in the first and \(\ell_0-1\to\ell_0\) in the others, and then subtracting \(U_{n-2}\) removes exactly those with total insertion–weight one higher.  The net result reproduces precisely the terms in \eqref{eq:main_path_expansion_HA} with the case $l_0>0$ and the cases  $l_0=0$ with \(\alpha_1\in\{2,3\}\). 

\textbf{(ii) The \(A\)-term.}
Similarly,
\begin{equation}
A U_{n-1}
=\sum_{k,\alpha,\ell_i}
A\Bigl(V_{\ell_0}\,\underline{\Phi_{\alpha_1}V_{\ell_1}}\cdots\underline{\Phi_{\alpha_k}V_{\ell_k}}\Bigr).
\end{equation}
By definition of \(\underline{\Phi_A V_{\ell_0}}=\sum_{i=1}^{\ell_0}A^iV_{\ell_0-i}\), this exactly yields the terms in \eqref{eq:main_path_expansion_HA} with the first insertion \(\alpha_1=A\). 

Hence \(D_n=0\) for all \(n\), proving the proposition.
\end{proof}

\section{Chebyshev expansion in the Wishart Case}

\begin{definition}[Nonbacktracking powers of $H$]\label{def:NBP-Wishart}
Let $H\in\mathbb{C}^{N\times M}$ be a $N\times M$ matrix, and let $x,y$  respectively denote  elements in $[N], [M]$.  For each integer $n\ge0$, define the nonbacktracking power matrix $V_n$ (viewed as an operator) as follows:

\paragraph{Even case ($n=2m$):}
\begin{itemize}
  \item $V_{2m}:\,\mathbb{C}^{[N]}\to\mathbb{C}^{[N]}$, with
  \begin{equation}
  (V_{2m})_{x_0x_{2m}} = \sum_{y_1,x_2,\dots,y_{2m-1}}
  \Bigl[\prod_{i=0}^{2m-2}\mathbf{1}(\zeta_i\neq\zeta_{i+2})\Bigr]\,
  H_{x_0y_1}H^*_{y_1x_2}\cdots H^*_{y_{2m-1}x_{2m}},
  \end{equation}
  summing over $y_{2k+1}\in[M],\,x_{2k}\in[N]$. Here $\zeta_i$ denotes $x_i$ or $y_i$.
  \item $V_{2m}:\,\mathbb{C}^{[M]}\to\mathbb{C}^{[M]}$, with
  \begin{equation}
  (V_{2m})_{y_0y_{2m}} = \sum_{x_1,y_2,\dots,x_{2m-1}}
  \Bigl[\prod_{i=0}^{2m-2}\mathbf{1}(\zeta_i\neq\zeta_{i+2})\Bigr]\,
  H^*_{y_0x_1}H_{x_1y_2}\cdots H_{x_{2m-1}y_{2m}},
  \end{equation}
  summing over $x_{2k+1}\in[N],\,y_{2k}\in[M]$.
\end{itemize}

\paragraph{Odd case ($n=2m+1$):}
\begin{itemize}
  \item $V_{2m+1}:\,\mathbb{C}^{[N]}\to\mathbb{C}^{[M]}$, with
  \begin{equation}
  (V_{2m+1})_{x_0y_{2m+1}} = \sum_{y_1,x_2,\dots,x_{2m}}
  \Bigl[\prod_{i=0}^{2m-1}\mathbf{1}(\zeta_i\neq\zeta_{i+2})\Bigr]\,
  H_{x_0y_1}H^*_{y_1x_2}\cdots H_{x_{2m}y_{2m+1}},
  \end{equation}
  summing over $y_{2k+1}\in[M],\,x_{2k}\in[N]$.
  \item $V_{2m+1}:\,\mathbb{C}^{[M]}\to\mathbb{C}^{[N]}$, with
  \begin{equation}
  (V_{2m+1})_{y_0x_{2m+1}} = \sum_{x_1,y_2,\dots,y_{2m}}
  \Bigl[\prod_{i=0}^{2m-1}\mathbf{1}(\zeta_i\neq\zeta_{i+2})\Bigr]\,
  H^*_{y_0x_1}H_{x_1y_2}\cdots H^*_{y_{2m}x_{2m+1}},
  \end{equation}
  summing over $x_{2k+1}\in[N],\,y_{2k}\in[M]$.
\end{itemize}
We also set 
\begin{equation}
V_0 = \begin{cases}I_N:\mathbb{C}^{[N]}\to\mathbb{C}^{[N]},\\
I_M:\mathbb{C}^{[M]}\to\mathbb{C}^{[M]},\end{cases}
\quad
V_1 = \begin{cases}H:\mathbb{C}^{[N]}\to\mathbb{C}^{[M]},\\
H^*:\mathbb{C}^{[M]}\to\mathbb{C}^{[N]},\end{cases}
\quad
V_n = 0~(n<0).
\end{equation}
\end{definition}

\begin{definition}\label{def:Phi-Wishart}
Denote
\begin{equation}
(\Phi_2)_{xx'} = \delta_{xx'}\sum_{y\in[M]}\bigl(|H_{xy}|^2-\mathbb{E}[|H_{xy}|^2]\bigr),
\quad
(\Phi_3)_{xx'} = -|H_{x x'}|^2\,H_{x x'},
\end{equation}
for $x,x'\in[N]$, and
\begin{equation}
(\Phi_2)_{yy'} = \delta_{yy'}\sum_{x\in[N]}\bigl(|H_{xy}|^2-\mathbb{E}[|H_{xy}|^2]\bigr),
\quad
(\Phi_3)_{yy'} = -|H_{y y'}|^2\,H^*_{y y'},
\end{equation}
for $y,y'\in[M]$.  In addition,
\begin{equation}
(\Phi_A)_{xy}=A_{xy},
\quad
(\Phi_{A^*})_{yx}=(A^*)_{yx}\end{equation}
For an insertion $\alpha \in \{2,3\}$ and $n \ge 0$, define
\begin{equation}
(\underline{\Phi_\alpha V_{n+\alpha}})_{v_0 v_{n+1}}
= \sum_{v_1,\dots,v_n}
\Biggl[\prod_{i=0}^{n-1} \mathbf{1}(v_i \neq v_{i+2}) \Biggr]
(\Phi_\alpha)_{v_0v_1} \,(Y)_{v_1v_2}\,(Y^*)_{v_2v_3}\,(Y)_{v_3v_4} \cdots (Y/Y^*)_{v_n v_{n+1}},
\end{equation}
where the product alternates between $Y$ and $Y^*$, and   $Y = H$ if $v_1 \in [N]$ and $Y = H^*$ if $v_1 \in [M]$.
For $\alpha\in \{A,A^*\}$, we define 
\begin{equation}
    \underline{\Phi_\alpha V_{n+1}}=\Phi_{\alpha}V_n.
\end{equation}
\end{definition}

\begin{lemma}\label{lem:B.3}
    For the matrix  $V$  in Definition \ref{def:NBP-Wishart}, we have 
    \begin{equation}
        V_n=(H+A)^*V_{n-1}-\alpha V_{n-2}-\underline{\Phi_2 V_n}-\underline{\Phi_3 V_{n}}-\underline{\Phi_{A^*}V_{n}},
    \end{equation}
and
    \begin{equation}
        V_n=(H+A)V_{n-1}- V_{n-2}-\underline{\Phi_2 V_n}-\underline{\Phi_3 V_{n}}-\underline{\Phi_{A}V_{n}},
    \end{equation}
\end{lemma}
\begin{proof}
    The proof is the same as that in \cite[Lemma 6.1]{EK11Quantum}. 
\end{proof}

\begin{proposition}[Deformed path expansion: Wishart case
]\label{prop:Q_path_expansion}
Define  a family of  polynomials $\mathcal{Q}_n(x)$ recursively by
\begin{equation} \label{Qrec}
\mathcal{Q}_0(x) = 1, \quad
\mathcal{Q}_1(x) = x - 1, \quad
\mathcal{Q}_n(x) = (x - 1 - \alpha)\mathcal{Q}_{n-1}(x) - \alpha \mathcal{Q}_{n-2}(x) \quad (n \ge 2),
\end{equation}
with the convention that $\mathcal{Q}_n(x) = 0$ for $n < 0$. Then for $X = (H+A)(H+A)^*$ and  for each $n \ge 0$, we have
\begin{equation}
\mathcal{Q}_n(X)
= \sum_{k \ge 0} \sum_{\alpha \in \{2,3,A,A^*\}^k} \sum_{\substack{\ell_0 + \cdots + \ell_k = 2n}}
V_{\ell_0}\,\underline{\Phi_{\alpha_1} V_{\ell_1}} \cdots \underline{\Phi_{\alpha_k} V_{\ell_k}}.
\end{equation}
Here, each $\alpha_i$ denotes insertion of $\Phi_2$, $\Phi_3$, $\Phi_A$, or $\Phi_{A^*}$, and the indices satisfy $\ell_i \ge 0$ with $\sum \ell_i = 2n$.
\end{proposition}

\begin{proof}
By Lemma~\ref{lem:B.3}, we have
\begin{equation}
\begin{aligned}
    (H+A)^* V_{\ell_0}
    = V_{\ell_0+1} + \alpha V_{\ell_0-1}
    + \underline{\Phi_2 V_{\ell_0+1}} + \underline{\Phi_3 V_{\ell_0+1}} + \underline{\Phi_{A^*} V_{\ell_0+1}}.
\end{aligned}
\end{equation}
Applying $(H+A)$ to both sides gives
\begin{equation}
\begin{aligned}
    (H+A)(H+A)^* V_{\ell_0}
    &= (H+A)V_{\ell_0+1} + \alpha (H+A)V_{\ell_0-1} \\
    &\quad + (H+A)\underline{\Phi_2 V_{\ell_0+1}} + (H+A)\underline{\Phi_3 V_{\ell_0+1}} + (H+A)\underline{\Phi_{A^*} V_{\ell_0+1}}.
\end{aligned}
\end{equation}
Each of these terms expands similarly:
\begin{align}
    (H+A)V_{\ell_0+1}
    &= V_{\ell_0+2} + V_{\ell_0}
    + \underline{\Phi_2 V_{\ell_0+2}} + \underline{\Phi_3 V_{\ell_0+2}} + \underline{\Phi_A V_{\ell_0+2}}, \\
    (H+A)V_{\ell_0-1}
    &= V_{\ell_0} + V_{\ell_0-2}
    + \underline{\Phi_2 V_{\ell_0}} + \underline{\Phi_3 V_{\ell_0}} + \underline{\Phi_A V_{\ell_0}}, \\
    (H+A)\underline{\Phi_{\alpha} V_{\ell_0+1}}
    &= V_1 \underline{\Phi_{\alpha} V_{\ell_0+1}} + \underline{\Phi_A V_1}\, \underline{\Phi_{\alpha} V_{\ell_0+1}}.
\end{align}

Proceeding inductively and following the argument in \cite[Proposition~6.2]{EK11Quantum}, we can obtain the claimed path expansion for $\mathcal{Q}_n(X)$.
\end{proof}

\begin{remark}
    $Q_n(x)$ defined in \eqref{Qrec} admits a representation via   Chebyshev polynomials of the second kind,  
\begin{equation}\label{equ:Q=U}
\mathcal Q_n(x) = \alpha^{\tfrac{n}{2}}\left(
    U_n\left(\tfrac{x - (1+\alpha)}{2\sqrt{\alpha}}\right)
   + \sqrt{\alpha}\, U_{n-1}\left(\tfrac{x - (1+\alpha)}{2\sqrt{\alpha}}\right)
\right).
\end{equation}
To prove it, observe that both sides satisfy the same recurrence relation. Furthermore, direct computation of the first few values (e.g., $n = 0,1$) confirms  that they match. Thus, by induction, the expressions agree for all $n \ge 0$, completing the proof.

On the other hand, by contracting consecutive terms of the form $\underline{\Phi_{A}V_1}\underline{\Phi_{A^*}V_1}\cdots\underline{\Phi_{A}V_1}$, one can derive an useful expression analogous to \eqref{eq:main_path_expansion_HA}.
\end{remark}

\section{Polynomial moments of Gaussian variables}

We derive recursion formulas and inequality estimates for the mixed normal moments.

\begin{lemma}\label{lem:Gaussian_moment}
Let $g\sim\mathcal{N}(0,1)$. For non-negative integers $a$ and $b$, define
\begin{equation}
I(a,b) := \mathbb{E}\!\left[(g^2 - 1)^a g^{2b}\right].
\end{equation}
Then we have
\begin{itemize}
\item[(i)] \textbf{(Recursion)} For $a\ge 2$,
\begin{equation}\label{eq:Ia0-rec}
I(a,0)=  2(a-1)\big(I(a-1,0)+I(a-2,0)\big),
\end{equation} 
and for $a\ge 1, b\ge 1$,
\begin{equation}\label{eq:Iab-rec}
I(a,b)=2a\, I(a-1,b)+ (2b-1)\,I(a,b-1),
\end{equation}
with initial conditions
\begin{equation}
I(0,0) =1,\quad I(1,0) =0,\quad I(0,b) = (2b-1)!!.
\end{equation}

\item[(ii)] \textbf{(Inequality)} For $b=0,\,a\ge 2$ or $b\ge 1,\,a\ge 0$,
\begin{equation}
 \mathbb{E}[g^{2a+2b}] \leq 2^a\, \mathbb{E}[(g^2-1)^a g^{2b}].
\end{equation}
\end{itemize}
\end{lemma}

\begin{proof}

We first recall the classical Stein identity for a standard Gaussian variable \(g\):
\begin{equation}\label{stein}
\mathbb{E}[g h(g)] = \mathbb{E}[h'(g)],
\end{equation}
which holds for any polynomial \(h\).

For \(b=0\) and \(a\ge 2\), applying \eqref{stein} with \(h(g) = g (g^2-1)^{a-1}\) yields
\begin{align*}
I(a,0) + I(a-1,0)& =  \mathbb{E}\left[g \cdot \left( g (g^2-1)^{a-1} \right) \right] \\
&= \mathbb{E}\left[(g^2-1)^{a-1} + 2(a-1)g^2 (g^2-1)^{a-2} \right]\\
&=I(a-1,0)+2(a-1)\big(I(a-1,0)+I(a-2,0)\big).
  \end{align*}
Hence, we further have 
\begin{equation} \label{a0rec}
  I(a,0)=  2(a-1)\big(I(a-1,0)+I(a-2,0)\big), \quad a\geq 2.
\end{equation}
On the other hand,  by \eqref{stein} again,  for  $b
\geq 1$  we have   
    \begin{align} \label{abre}
 I(a,b) &= \mathbb{E}\left[g \left( g^{2b-1} (g^2 - 1)^a \right)\right] = \mathbb{E}\left[(2b-1)g^{2b-2}(g^2 - 1)^a + 2a  \cdot g^{2b} (g^2 - 1)^{a-1} \right] \nonumber\\
 &= (2b-1)I(a,b-1) + 2aI(a-1,b).
 \end{align}   
It's easy to verify the initial conditions. This thus completes the proof of part (i).

To prove part (ii). 
Combining   $I(0,0) = 1$ and $I(1,0) = 0$, 
we easily  see  from \eqref{a0rec} that 
\begin{equation}
  I(a,0) \geq I(k,0)\geq 0, \quad \forall \, 0 \leq k \leq a, \quad a \geq 2.
\end{equation}
So by above inequalities,  when \( a \geq 2 \),  we arrive at 
\begin{align*}
\mathbb{E}[g^{2a}]  
= \sum_{k=0}^{a} \mathbb{E}\left[\binom{a}{k} (g^2-1)^k\right]
&= \sum_{k=0}^{a} \binom{a}{k} I(k,0) \leq 2^a I(a,0),
\end{align*}
that is, 
\begin{equation} \label{a0ineq}
\mathbb{E}[g^{2a}] \leq 2^a \mathbb{E}[(g^2-1)^a], \quad a \geq 2.
\end{equation}

By induction using \eqref{a0rec} and the initial values, we see that
\begin{equation}
I(a,b) \ge 0, \quad \forall a,b \ge 0.
\end{equation}
Moreover, for fixed \(b\ge 1\), the recursion \eqref{abre} implies that \(I(a,b)\) is non-decreasing in \(a\):
\begin{equation}
I(a,b) \ge I(k,b), \quad \forall 0 \le k \le a.
\end{equation}
Hence, for \(b\ge 1\),
\begin{equation} \label{abineq}
\mathbb{E}\left[g^{2a+2b}\right] = \mathbb{E}\left[(g^2 - 1 + 1)^a g^{2b}\right]
= \mathbb{E}\left[\sum_{k=0}^{a} \binom{a}{k} (g^2 - 1)^k g^{2b}\right] = \sum_{k=0}^{a} \binom{a}{k} I(k,b)
\leq 2^a I(a,b)
\end{equation}
where the inequality follows from the monotonicity of \(I(a,b)\).

Combining \eqref{a0ineq} and \eqref{abineq} completes the proof of part (ii).
\end{proof}

\section{
Upper bounds for transition probabilities   
}
\label{sec:clt-upper}

\begin{theorem}[Upper bounds for transition probabilities]\label{thm:clt-upper}
If   $f$ satisfies the assumptions in Definition  ~\ref{ass:ft},   
then there exist constants $C = C(d,\alpha,f,K) > 0$ and $C' = C'(d,\alpha,f,K) > 0$ such that for all integers $n \ge 1$ and all $x \in \Lambda_L$,
\begin{equation}\label{equ:band_mixing}
    \Bigl|\, p_n(0,x) - \frac{1}{N}\,\Bigr| 
    \;\le\; C\,W^{-d}\,n^{-\frac{d}{\alpha}} \;+\; C' n W^{-K}.
\end{equation}
In addition,    if
\begin{equation}\label{eq:negligible}
  nW^{-K} \;\ll\; N^{-1},
\end{equation}
then  
\begin{equation}\label{eq:upper-final}
  p_n(0,x) \;\le\; C\bigl(W^{-d}\,n^{-\frac{d}{\alpha}} + N^{-1}\bigr).
\end{equation}
\end{theorem}

\begin{proof}
We begin with the discrete Fourier representation
\begin{equation}\label{equ:D2}
p_n(0,x)=\frac{1}{N}\sum_{k\in\Lambda_L}\bigl(\widehat p(k)\bigr)^n e^{2\pi i k\cdot x/L},
\end{equation}
where the discrete Fourier transform of the variance profile
\begin{align}
    \widehat{p}(k)&=\frac{1}{M}\sum_{x\in \Lambda}\sum_{n\in \mathbb{Z}^d}
    f\!\left(\frac{x+nW}{W}\right)e^{-\frac{2\pi i k\cdot x}{L}}  
    =\frac{1}{M}\sum_{x\in \mathbb{Z}^d} f\!\left(\frac{x}{W}\right)e^{-\frac{2\pi i k\cdot x}{L}}\notag\\
    &=\frac{W^d}{M}\sum_{n\in \mathbb{Z}^d}\widehat{f}\!\left(W\Bigl(\tfrac{k}{L}+n\Bigr)\right)  
    =\frac{W^d}{M}\widehat{f}\!\left(\tfrac{Wk}{L}\right)+O(W^{-K}), 
\end{align} where in the last two equalities  we have used {Poisson summation} and dropped $|n|>0$ terms via \eqref{equ:D1}. In the Poisson summation step, both series converge absolutely by \eqref{equ:D1}, hence uniformly.  Furthermore,  the same two series  define bounded continuous functions. By Plancherel identity  they agree in the sense of $L^2$ distance, and since two continuous functions that are equal a.e. must coincide everywhere on $\mathbb{R}^d$, the identities hold pointwise.

Taking absolute values in \eqref{equ:D2} and applying the triangle inequality, we obtain
\begin{equation}\label{equ:D5}
|p_n(0,x)-N^{-1}| \le \frac{1}{N}\sum_{k\in\Lambda_L,\, k\ne 0}\Bigl|\widehat f\!\Bigl(\tfrac{Wk}{L}\Bigr)\Bigr|^n \;+\; O\!\left(nW^{-K}\right),
\end{equation}
where the factor $n$ in the error term comes from expanding $(\widehat f+O(W^{-K}))^n$.

Next, we proceed to estimate the main sum: 
\begin{align*}
    \frac{1}{N}\sum_{k\in\Lambda_L,\, k\ne 0}\Bigl|\widehat f\!\Bigl(\tfrac{Wk}{L}\Bigr)\Bigr|^n
    &\le \frac{1}{L^d}\sum_{k\in\mathbb{Z}^d,\, k\ne 0}\Bigl|\widehat f\!\Bigl(\tfrac{Wk}{L}\Bigr)\Bigr|^n\\
    &=\frac{1}{L^d}\Biggl(\sum_{0<|\tfrac{Wk}{L}|<\epsilon} \;+\;
      \sum_{\epsilon\le |\tfrac{Wk}{L}|<100 C_K} \;+\;
      \sum_{100 C_K\le |\tfrac{Wk}{L}|}\Biggr)
      \Bigl|\widehat f\!\Bigl(\tfrac{Wk}{L}\Bigr)\Bigr|^n \\
    &:= I_1+I_2+I_3.
\end{align*}

For $I_1$, Assumption~\ref{ass:ft}(2) implies
\begin{equation}
    I_1 \;\le\; C W^{-d}\int_{\mathbb R^d}\exp\!\bigl(-c_0 n|\xi|^\alpha\bigr)\,d\xi
    \;\le\; C W^{-d} n^{-\tfrac{d}{\alpha}}.
\end{equation}
For $I_2$, Assumption~\ref{ass:ft}(3) gives
\begin{equation}
    I_2 \;\le\; \frac{1}{L^d}\Bigl(\tfrac{100C_K L}{W}\Bigr)^d \rho^n
    \;\le\; C'_K W^{-d}\rho^n
    \;\le\; C''_K W^{-d} n^{-\tfrac{d}{\alpha}},
\end{equation}
while for $I_3$,   Assumption~\ref{ass:ft}(4) leads to 
\begin{equation}
    I_3 \;\le\; C W^{-d}\int_{|\xi|>50C_K}\Bigl(\tfrac{C_K}{(1+|\xi|)^{K}}\Bigr)^n d\xi
    \;\le\; C' W^{-d} n^{-\tfrac{d}{\alpha}}.
\end{equation}

Finally, combining these bounds with \eqref{equ:D5}, we conclude the proof.
\end{proof}

\end{document}